\documentclass{amsart}

\usepackage{amsmath, amscd}
\usepackage{amssymb}
\usepackage[alphabetic]{amsrefs}
\usepackage[T1]{fontenc}
\usepackage[all]{xy}
\usepackage{mathtools}
\usepackage{footmisc}

\usepackage{relsize} 
\usepackage[bbgreekl]{mathbbol} 
\usepackage{amsfonts} 
\usepackage[colorlinks=true, hyperindex, linkcolor=cyan, citecolor=, pagebackref=false, urlcolor=cyan, linktocpage]{hyperref}

\DeclareSymbolFontAlphabet{\mathbb}{AMSb} 
\DeclareSymbolFontAlphabet{\mathbbl}{bbold} 
\newcommand{\Prism}{{\mathlarger{\mathbbl{\Delta}}}}

\theoremstyle{plain}
\newtheorem{thm}{Theorem}[section]
\newtheorem{lem}[thm]{Lemma}
\newtheorem{prop}[thm]{Proposition}
\newtheorem{cor}[thm]{Corollary}

\theoremstyle{definition}
\newtheorem{rem}[thm]{Remark}
\newtheorem{defn}[thm]{Definition}
\newtheorem{exam}[thm]{Example}

\newtheorem{const}[thm]{Construction}
\newtheorem{notation}[thm]{Notation}

\newcommand\bA{\mathbb A}
\newcommand\bQ{\mathbb Q}
\newcommand\bN{\mathbb N}
\newcommand\bZ{\mathbb Z}
\newcommand\bF{\mathbb F}

\newcommand\cO{\mathcal O}

\DeclareMathOperator{\Spec}{Spec}
\DeclareMathOperator{\Spf}{Spf}
\DeclareMathOperator{\gp}{gp}
\DeclareMathOperator{\Inf}{inf}

\DeclareMathOperator{\Shv}{Shv}
\DeclareMathOperator{\et}{\acute{e}t}
\DeclareMathOperator{\crys}{crys}
\DeclareMathOperator{\CRYS}{CRYS}
\DeclareMathOperator{\BK}{BK}
\DeclareMathOperator{\pr}{pr}
\DeclareMathOperator{\id}{id}

\title{Logarithmic prismatic cohomology I}
\author{Teruhisa Koshikawa}
\date{}

\begin{document}

\begin{abstract}
We introduce a logarithmic variant of the notion of $\delta$-rings, which we call $\delta_{\log}$-rings, and use it to define a logarithmic version of the prismatic site introduced by Bhatt and Scholze. In particular, this enables us to construct the Breuil-Kisin cohomology in the semistable case. 
\end{abstract}

\maketitle

\tableofcontents

\section{Introduction}
Bhatt and Scholze introduced the notion of prisms and prismatic site \cite{BS}. It is a mixed-characteristic generalization/refinement of the classical crystalline theory, and recovers the $A_{\Inf}$-cohomology and the Breuil-Kisin cohomology in the smooth case developed in \cites{BMS1, BMS2}. The prismatic theory and its variant,  $q$-crystalline theory, also provide a construction of the $q$-de Rham cohomology conjectured by Scholze \cite{Scholze}\footnote{It is instructive to compare the prismatic theory and $q$-crystalline theory with Pridham's earlier attempt \cite{Pridham}.}.   

In this article, we begin to develop a logarithmic generalization of their theory, which recovers the $A_{\Inf}$-cohomology in the semistable case \cite{CK} and provides its Breuil-Kisin descent. 
(A logarithmic generalization of \cite{BMS2} is not yet known, so this is the first construction of the Breuil-Kisin cohomology in the semistable case.)

Let us first recall some basics of prisms. 
A \emph{prism} is a pair $(A, I)$ that consists of a \emph{$\delta$-ring} $A$ and an ideal $I\subset A$ satisfying some properties. In particular, it comes with a Frobenius lift $\phi=\phi_A\colon A\to A$. 
In practice, $I$ is a principal ideal generated by a nonzerodivisor; such a prism is called orientable.  
Some examples of orientable prisms are the following:

\begin{exam}
Fix a prime number $p$. 
\begin{enumerate}
    \item (A special case of crystalline prisms) Let $k$ be a perfect field of characteristic $p$. Then, $(W(k), (p))$ is a prism with $\phi$ being the canonical Frobenius lift. 
    \item (A special case of perfect prisms) Let $C$ be the completion of the algebraic closure of $W(k)[1/p]$. There is a usual surjection $\theta\colon A_{\Inf}\to \cO_C$. Then, $(A_{\Inf}, \ker \theta)$ is a prism with $\phi$ being the canonical Frobenius lift. 
    \item (Breuil-Kisin prisms) Let $K$ be a finite totally ramified extension of $W(k)[1/p]$. Fix a uniformizer $\pi$ and let $E(u) \in W(k)[u]$ denote the (monic) Eisenstein polynomial such that $E(\pi)=0$. Then, $(W(k)[\![ u]\!], (E(u)))$ is a prism with $\phi$ sending $u$ to $u^p$. 
\end{enumerate}
\end{exam}
Note that these prisms are related as follows: if $\overline{k}$ denotes the residue field of $C$, then a natural surjection $A_{\Inf}\to W(\overline{k})$ induces a map of prisms
\[
(A_{\Inf}, \ker \theta) \to (W(\overline{k}), (p)). 
\]
For a compatible system $\pi, \pi^{1/p}, \pi^{1/p^2}, \dots$ of $p$-power roots of $\pi$ inside $C$, a natural map $W(k) [\![ u ]\!]\to A_{\Inf}; u\mapsto [\pi^\flat]$ induces a map of prisms:
\[
(W(k)[\![ u]\!], (E(u))) \to (A_{\Inf}, \ker \theta). 
\]

Let $(A, I)$ be a bounded prism, i.e., $A/I[p^{\infty}]=A/I[p^n]$ for some $n$. (All examples above are bounded.) In particular $A, A/I$ are classically $(p,I)$-complete. For a smooth $p$-adic formal scheme $X$ over $A/I$, Bhatt and Scholze defined the prismatic site $(X/A)_{\Prism}$ with the structure sheaf $\cO_{\Prism}$. An object of the prismatic site is a bounded prism $(B,IB)$ over $(A, I)$ with a map $\Spf (B/IB)\to X$ satisfying a compatibility condition. 
The \emph{prismatic cohomology} is the cohomology of the structure sheaf:
\[
R\Gamma_{\Prism}(X/A)\coloneqq R\Gamma ((X/A)_{\Prism}, \cO_{\Prism}), 
\]
which is a $(p,I)$-complete object of $D(A)$ equipped with a $\phi_A$-semilinear map. 

Regarding the examples above, Bhatt and Scholze established the following:
\begin{exam}
For simplicity, assume $X$ is proper over $A/I$. Then, $R\Gamma_{\Prism}(X/A)$ is perfect and, in each case of (1)-(3) above, it is
\begin{enumerate}
    \item a Frobenius descent of the crystalline cohomology $R\Gamma_{\crys}(X/W(k))$,
    \item a Frobenius descent of the $A_{\Inf}$-cohomology $R\Gamma_{A_{\Inf}}(X)$ \cite{BMS1}, 
    \item isomorphic to the Breuil-Kisin cohomology $R\Gamma_{\mathfrak{S}}(X)$ \cite{BMS2}; we prefer to write $R\Gamma_{\BK}(X)$ instead. 
\end{enumerate}
\end{exam}

General results for the prismatic cohomology are summarized in \cite{BS}*{Theorem 1.8, 1.16}. 
In particular, cohomology groups in (1)-(3) are related by base change along the maps explained above. 

We aim to introduce a logarithmic generalization, i.e., a version for formal schemes with \emph{log structures} in the sense of Fontaine-Illusie-Kato. 
The first approximation of log structures is the notion of \emph{prelog rings}. Recall that a prelog ring consists of a ring $A$, a monoid $M$, and a map $\alpha\colon M\to A$ of monoids, where $A$ is regarded as the multiplicative monoid. We can endow the example (1)-(3) with the following prelog structures:
\begin{exam}
We send the unit element to the unit element. 
\begin{enumerate}
    \item $\bN\to W(k); 1\mapsto 0$. This is used to define the so-called Hyodo-Kato cohomology. 
    \item $\cO_C^\flat\setminus\{0\}\to A_{\Inf}; x\mapsto [x]$. This appears, for instance, in \cite{CK}. 
    \item $\bN\to W(k)[\![u]\!]; 1\mapsto u$. 
\end{enumerate}
Note that the prelog structure of (3) is compatible with those of (1) and (2) via maps 
\begin{gather*}
W(k)[\![u]\!] \to W(k);u\mapsto 0, \quad \bN\to \bN;1\mapsto 1, \quad \textnormal{and} \\
W(k)[\![u]\!]\to A_{\Inf},\quad  \bN\mapsto \cO_C^\flat\setminus\{0\}; 1\mapsto [\pi^\flat]
\end{gather*}
respectively.
(The prelog structures of (1) and (2) are also related but less directly.) 
\end{exam}

In these examples, $\phi$ is simply the $p$-th power map on the image $\alpha(M)$, thus $\phi$ is compatible with $M\to M; m\mapsto m^p$. One possible approach to develop a logarithmic prismatic theory is to consider
\begin{center}
    a bounded prism $(A, I)$ with a prelog structure $\alpha\colon M\to A$ such that $\phi_A(\alpha(m))=\alpha(m^p)=\alpha(m)^p$ holds for every $m\in M$. 
\end{center}
It is indeed a reasonable object and we call it a \emph{prelog prism of rank $1$}.  
However, this approach is restrictive. In the language of log geometry, it means we work with \emph{charts} of log structures satisfying the special condition on $\phi_A$ as above. It seems impossible to define a well-behaved log prismatic site for a general log formal scheme in this way. (Nevertheless, it can be used to compute the right cohomology locally on $X$ at least in reasonable situations; cf. Proposition \ref{more coverings}.)

To go beyond that, let us now recall the definition of a log structure on a formal scheme $X$. 
A log structure is a map of \'etale sheaves of monoids $\alpha\colon \mathcal{M} \to \cO_X$ that induces an isomorphism $\alpha^{-1}(\cO_X^\times)\cong \cO_X^\times$. Locally, we have a \emph{log ring} rather than a prelog ring, i.e., $\alpha\colon M\to A$ induces an isomorphism $\alpha^{-1}(A^\times)\cong A^\times$. So, we will work with log rings. (We will always to pass the associated log structure in the main text, but we will ignore this point in this introduction.)
In this context, a naive version of a log prism is a triple $(A, I, M)$ such that $(A,I)$ is a bounded prism and $(A, M)$ is a log ring, equipped with an endomorphism $(\phi_A, \phi_M)$ of $(A, M)$ extending $\phi_A$, i.e., a sort of Frobenius lift. One may ask:
\begin{center}
    What is a Frobenius lift on $M$?
\end{center} 
Also, let us recall that a $\delta$-structure is finer than a Frobenius lift on a ring in general. With this in mind, here is our solution: $\phi_M$ should have the form of 
\begin{center}
$\phi_M (m)=m^p (1+ p\delta_{\log}(m))$ for a map\footnote{The notation $\delta_{\log}$ is borrowed from K. Kato.} $\delta_{\log}\colon M\to A$. 
\end{center}
Note that $1+ p\delta_{\log}(m)$ is invertible in $A$ as $p$ is topologically nilpotent, and thus can be regarded as an element of $M$ since $(A, M)$ is a log ring. The map $\delta_{\log}$ must satisfy some formulas analogous to those that $\delta$ satisfies. 
It is more convenient to have a prelog version of this notion, forgetting $I$. A \emph{$\delta_{\log}$-ring} is a tuple
\[
(A, M, \alpha\colon M\to A, \delta\colon A\to A, \delta_{\log}\colon M\to A),
\]
where $(A, \delta)$ is a $\delta$-ring, $(A, M, \alpha)$ is a prelog ring, and the map $\delta_{\log}$ satisfy
\begin{gather*}
    \delta_{\log}(e)=0, \quad \alpha(m)^p \delta_{\log}(m)=\delta (\alpha(m)), \\
    \delta_{\log}(mm')=\delta_{\log}(m)+\delta_{\log}(m')+p \delta_{\log}(m)\delta_{\log}(m'). 
\end{gather*}
Namely, $\delta_{\log}$ of the log prism should satisfy these formulas as well.  
The previous example (1)-(3) may be regarded as $\delta_{\log}$-rings by letting $\delta_{\log}=0$. 
Let us also note that there is a unique $\delta_{\log}$ on the trivial log structure $A^\times \subset A$, given by
\[
\delta_{\log}(a)=\frac{\delta (a)}{a^p}
\]
for $a\in A^\times$. 
By passing to the associated log ring (or more precisely, log structure in the main text), it is always possible to attach a log prism to any (bounded) prelog prism.

Once the notion of log prisms is defined, the log prismatic site is defined immediately; it is a mixture of the prismatic site and the (big) log crystalline site. We refer to Definition \ref{log prismatic site} for the precise definition. Fix a bounded prelog prism $(A, I, M_A)$. 
If $(X, M_X)$ is a smooth\footnote{In the main text, we require that it be integral over the base. In examples (1)-(3), this is automatic.} log $p$-adic scheme over a prelog ring $(A/I, M_A)$, we write $((X, M_X)/ (A, M_A))_{\Prism}$ for its log prismatic site, and define the \emph{log prismatic cohomology} of $(X, M_X)$ relative to $(A, M_A)$ by
\[
R\Gamma_{\Prism}((X, M_X)/ (A, M_A))\coloneqq R\Gamma (((X, M_X)/ (A, M_A))_{\Prism}, \cO_{\Prism})
\]
for the structure sheaf $\cO_{\Prism}$. 

Let us illustrate our results on the log prismatic cohomology:

\begin{exam}[A special case of the crystalline comparison]
This is the case (1) above. For a quasi-compact quasi-separated smooth log formal scheme $(X, M_X)$ over $(W(k), \bN)$, there is a canonical isomorphism
\[
(\phi_{W(k)}^* R\Gamma_{\Prism}((X, M_X)/(W(k), \bN)))^{\wedge}_p \cong
R\Gamma_{\crys}((X, M_X)/(W(k), \bN)),
\]
where the left hand side is the $p$-completed base change along the Frobenius lift $\phi_{W(k)}$ and the right hand side is the log crystalline cohomology. 
This (or a reformulated version) holds for more general log crystalline prisms, i.e., log prisms with $I=(p)$; see Theorem \ref{crystalline comparison:global}. 
\end{exam}

\begin{exam}[$A_{\Inf}$-cohomology]
Consider the case (2): $(A_{\Inf}, \ker\theta, \cO_C^\flat\setminus\{0\})$. 
Let $X$ be a quasi-compact quasi-separated semistable formal scheme over $\cO_C$ treated in \cite{CK}, and let $M_X$ denote its canonical log structure, which naturally admits a map $\cO_C\setminus\{0\} \to  M_X$. Then, there is a canonical isomorphism
\[
(\phi_{A_{\Inf}}^* R\Gamma_{\Prism}((X, M_X)/(A_{\Inf}, \cO_C^\flat\setminus\{0\})))^{\wedge}_{(p, \phi(\xi))} \cong
R\Gamma_{A_{\Inf}}(X), 
\]
where the left hand side is the $(p, \phi (\xi))$-completed base change along the Frobenius lift $\phi_{A_{\Inf}}$ and the right hand side is the $A_{\Inf}$-cohomology in \cite{CK}. 
This is established through an intermediate log $q$-crystalline/$q$-de Rham theory, combined with the Hodge-Tate comparison for the log prismatic cohomology. 
Using the crystalline comparison of the log prismatic cohomology, we can translate\footnote{We did not check if they have the same map or not, but see Theorem \ref{comparison with BdR^+-cohomology}.} the absolute crystalline comparison for the $A_{\Inf}$-cohomology to the following base change isomorphism: 
\begin{multline*}
A_{\crys}\widehat{\otimes}^L_{A_{\Inf}} R\Gamma_{\Prism}((X, M_X)/(A_{\Inf}, \cO_C^\flat \setminus\{0\})) \\ \cong 
R\Gamma_{\crys}((X, M_X)_{\cO_C/p}/(A_{\crys}, \cO_C^\flat \setminus\{0\})),
\end{multline*}
where $A_{\Inf}\to A_{\crys}$ is the composite of the natural embedding and the Frobenius $\phi$, and the completion is $p$-adic. 
\end{exam}

\begin{exam}[Breuil-Kisin cohomology]
Now the case (3). Let $X$ be a quasi-compact quasi-separated semistable formal scheme over $\cO_K$ with the canonical log structure $M_X$. 
There is the base change isomorphism
\begin{multline*}
(A_{\inf}\otimes^L_{W(k)[\![u]\!]}R\Gamma_{\Prism}(X, M_X)/(W(k)[\![u]\!], \bN)))^{\wedge}_{(p,\xi)} \\ \cong
R\Gamma_{\Prism}((X, M_X)_{\cO_C}/ (A_{\Inf}, \cO_C^\flat\setminus \{0\})). 
\end{multline*}
Combined with the previous example, we see that the left hand side gives a descent of the $A_{\Inf}$-cohomology. 
Thus, we define the \emph{Breuil-Kisin cohomology} of $X$ by
\[
R\Gamma_{\BK}(X):=R\Gamma_{\Prism}((X, M_X)/(W(k)[\![u]\!], \bN)). 
\]
If $X$ is proper over $\cO_K$, it follows from the Hodge-Tate comparison (see below) that $R\Gamma_{\BK}(X)$ is a perfect complex, and the base change holds without completion.  
\end{exam}

Main properties of the log prismatic cohomology established in this paper are

\begin{itemize}
    \item the Hodge-Tate comparison (Theorem \ref{Hodge-Tate}),
    \item base change (Corollary \ref{base change II}), and
    \item the crystalline comparison (Theorem \ref{crystalline comparison:global}). 
\end{itemize}

A very short summary of the proof of the crystalline comparison is that it is done by comparing the \emph{log PD envelopes} and \emph{log prismatic envelopes}, which are tools to compute the corresponding cohomology.  
For those familiar with log crystalline theory, let us point out that a key technical point in this paper is to show that $\delta_{\log}$-structures extend through the exactification (Proposition \ref{change of monoids}); this leads to the existence of log prismatic envelopes. 

We prove the Hodge-Tate comparison by reducing it to the non-log case treated in \cite{BS} by working with certain explicit complexes of log prismatic envelopes. So it is actually independent of the crystalline comparison in the log setting, while the original proof of the Hodge-Tate comparison in \cite{BS} uses the crystalline comparison\footnote{More direct proofs of the Hodge-Tate comparison are found later.}. 
The base change property is a simple corollary of the Hodge-Tate comparison as in the non-log case.

We study the Nygaard filtration, the isogeny property of Frobenius\footnote{See Remark \ref{Isogeny} for the case of log $q$-crystalline cohomology.}, the de Rham comparison\footnote{\label{de Rham}The de Rham comparison can be already proved in some cases by using the crystalline comparison and base change. Say $(A, I)=(W(k)[\![u]\!], E(u))$, then we take the ($\phi$-twisted) base change to the $p$-completed PD envelope $S$ of $(E(u))\subset W(k)[u]$, apply the crystalline comparison, and specialize the log crystalline cohomology from $S$ to $\cO_K$.}, and the \'etale comparison in a sequel \cite{Koshikawa-Yao} through the derived log prismatic cohomology. 
Let us also note that Diao-Yao has generalized the construction of $A_{\Inf}$-cohomology \cite{CK} to more general smooth log $p$-adic formal schemes. The comparison with the log prismatic cohomology would provide the \'etale comparison in that case too. 

\subsection*{Notation}
For many conventions and notation, we follow \cite{BS}. 
In particular, we refer the reader to \cite{BS} for the notion of $I$-completed flatness/smoothness/\'etaleness. 
For sites and topoi, we follow \cite{Stacks}. 
Base changes are denoted by adding subscripts. 

The completion means the derived completion unless otherwise specified, and it is denoted by $(-)^{\wedge}$. 
We sometimes write $(-)^{\wedge}_{\textnormal{cl}}$ for the classical completion. 
The topology for which the completion is taken depends on context and sometimes explicated by adding a subscript, e.g. $(-)^{\wedge}_p$. 
The completed polynomial ring is denoted by $\langle - \rangle$. 

A monoid $M$ is assumed to be commutative and its unit element is denoted by $e$. 
Its associated abelian group is denoted by $M^{\gp}$. We write $M^\times\subset M$ for the subgroup of invertible elements. For a formal scheme $X$ and a map of monoids $M \to \Gamma (X, \cO_X)$, we write $M^a_{X}$ for the associated log structure on the \'etale site of $X$. 
For several terminology of log geometry, we refer the reader to Appendix. 

\subsection*{Acknowledgements}
I would like to thank Kazuya Kato for constant encouragement which leads the author to the current generality of the theory presented in this paper. Takeshi Tsuji kindly pointed out a mistake in a previous version, and I want to thank him a lot.  
I am grateful to Bhargav Bhatt, Heng Du, Luc Illusie, Yutaro Mikami, and Zijian Yao for helpful remarks. 
I thank students of the University of Tokyo who shared a correction on \cite{BS}, and Ildar Gaisin and Koji Shimizu for their comments on drafts. 
This work was supported by JSPS KAKENHI Grant Number 20K14284.

\section{\texorpdfstring{$\delta_{\log}$}{delta-log}-rings}
Fix a prime number $p$. 
We first review the notion of $\delta$-rings introduced by Joyal \cite{Joyal}, following \cite{BS}. All rings are assumed to be $\bZ_{(p)}$-algebras. 

A \emph{$\delta$-ring} is a commutative ring $A$ with a map $\delta\colon A\to A$ satisfying
\begin{gather*}
    \delta(0)=\delta(1)=0, \quad \delta(xy)=\delta (x)y^p +x^p \delta(y) +p \delta(x)\delta(y), \\
    \delta(x+y)=\delta(x)+\delta(y) -\sum_{i=1}^{p-1} \frac{(p-1)!}{i!(p-i)!}x^i y^{p-i}. 
\end{gather*}
Given a $\delta$ as above, $\phi (x)=x^p +p\delta (x)$ is a ring endomorphism of $A$ lifting the Frobenius on $A/p$, and $\delta$ is determined by $\phi$ if $A$ is $p$-torsionfree. 

The category of $\delta$-rings is the obvious one. 
All limits and colimits exist, and forgetting $\delta$-structures preserves both limits and colimits \cite{BS}. In particular, the tensor product of $\delta$-rings admits a natural $\delta$-structure. 
In fact, the forgetful functor from the category of $\delta$-rings to the category of rings admits both left and right adjoint functors. It is a fundamental observation of Joyal \cite{Joyal} that the right adjoint functor is given by the ring of Witt vectors. 

An element $a$ of a $\delta$-ring $(A, \delta)$ is called \emph{distinguished} if $\delta (a)$ is a unit, and has \emph{rank $1$}
 if $\delta (a)=0$, which implies $\phi (a)=a^p$. 

Let us record the following criterion for an element to have rank $1$. 

\begin{lem}[\cite{BS}*{Lemma 2.32}]\label{rank 1 elements}
Let $(A, \delta)$ be a $p$-adically separated $\delta$-ring. An element $a$ of $A$ has rank $1$ if it admits a $p^n$-th root for any integer $n\geq 0$. 
\end{lem}


Recall that a \emph{prelog structure} $\alpha\colon M\to A$ of a ring $A$ is a map from a monoid $M$ to the multiplicative monoid $A$. 

\begin{defn}[$\delta_{\log}$-rings]
A \emph{$\delta_{\log}$-ring} is a tuple $(A, \delta, \alpha\colon M\to A, \delta_{\log}\colon M\to A)$ consisting of a $\delta$-ring $(A, \delta)$, a prelog structure $\alpha\colon M\to A$, and a map $\delta_{\log}\colon M \to A$ satisfying
\begin{enumerate}
    \item $\delta_{\log}(e)=0$, 
    \item $\alpha (m)^p \delta_{\log}(m)=\delta(\alpha(m))$, and
    \item $\delta_{\log}(mm')=\delta_{\log}(m)+\delta_{\log}(m')+p \delta_{\log}(m)\delta_{\log}(m')$
\end{enumerate}
for $m, m'\in M$. 
We simply write $(A, M)$ for a $\delta_{\log}$-ring if no confusion occurs. A morphism of $(A, M)\to (B,N)$ of $\delta_{\log}$-rings is a morphism of prelog rings compatible with $\delta$ and $\delta_{\log}$, and $\delta_{\log}$-rings form a category.  
A $\delta_{\log}$-ring is of \emph{rank $1$} if $\delta_{\log}=0$. 
\end{defn}

Let $(A, \alpha \colon M\to A)$ be a $\delta_{\log}$-ring. 
The property (2) of $\delta_{\log}$ relates the Frobenius lift $\phi$ of $A$ and the $p$-th power map on $M$:
\[
\phi (\alpha(m))= \alpha(m)^p \cdot (1 + p \delta_{\log}(m))
\]
for $m \in M$, and the map $m\mapsto 1 + p \delta_{\log}(m)$ is multiplicative on $m$ by the property (3). 
Thus, we see that $\phi^n (\alpha(m))\in \alpha(m)^{p^n}(1 + pA)$ for any integer $n \geq 0$. 
In particular, $\phi^n (\alpha(m))$ and $\alpha(m)^{p^n}$ differ by a unit if $A$ is classically $p$-complete; this has an effect that $\phi$ preserves the associated log structure as we discuss later. See also the remark below.  

Suppose that $\alpha(M)$ consists of nonzerodivisors. Then, $\delta_{\log}$ is uniquely determined by $\delta$ if it exists, and $\delta_{\log}$ exists if and only if $\delta (\alpha(m))$ is divisible by $\alpha(m)^p$ for all $m\in M$.  

\begin{rem}[Frobenius lifts on log rings]\label{Frobenius lifts on monoids}
If $(A, M)$ is a $\delta_{\log}$-ring of rank $1$, the $p$-th power map of $M$ lifts the Frobenius $\phi_A$ of $A$, and we obtain an endomorphism $(\phi_A, \phi_M)$ of $(A, M)$. A general $\delta_{\log}$-ring may not admit such an endomorphism. 

For log rings with $\delta_{\log}$-structure, we have a unique lift of $\phi_A$ under a mild assumption as follows.  
Let $(A, \alpha\colon M\to A)$ be a $\delta_{\log}$-ring, and assume that the Jacobson radical $\textnormal{rad}(A)$ contains $p$ and $(A, M)$ is a log ring, i.e., $\alpha$ induces an isomorphism $\alpha^{-1}(A^\times)\cong A^{\times}$. Then, 
\begin{center}
$1+p\delta_{\log}(m) \in A^\times$ for every $m \in M$,  
\end{center}
and $\alpha^{-1}(1+p\delta_{\log}(m))$ makes sense as an element of $M$. 
We have an endomorphism $\phi_M$ of $M$ defined by  
\[
\phi_M (m)=m^p \alpha^{-1}(1+p\delta_{\log}(m)),
\]
and we obtain an commutative diagram
\[
\begin{CD}
M @> \alpha >> A  \\
@V \phi_M VV @V \phi_A VV \\
M @>\alpha >> A, 
\end{CD}
\]
i.e., $(\phi_A, \phi_M)$ is an endomorphism of the log ring $(A, M)$. 
\end{rem}

\begin{exam}
\begin{enumerate}
    \item Let $(A, \delta)$ be a $\delta$-ring, and $A^\times \overset{\subset}{\longrightarrow} A$ the trivial log structure. There exists a unique $\delta_{\log}$-structure $\delta_{\log}\colon A^\times \to A$ in this setting, determined by
\[
\delta_{\log} (x)=\frac{\delta (x)}{x^p}
\]
    for $x\in A^\times$. 
    \item Let $R$ be a perfect ring of characteristic $p$. The ring of Witt vectors $W(R)$ with the canonical Frobenius lift may be regarded as a $\delta$-ring. The Teichm\"uller lifting 
    \[
    [-]\colon R \to W(R);\quad x \to [x]
    \]
    determines a $\delta_{\log}$-ring $(W(R), R)$ of rank $1$.  
    \item Let $M$ be a monoid, and consider the associated ring $\bZ_{(p)}[M]$. We endow $\bZ_{(p)}[M]$ with a $\delta$-structure such that every element $m$ of $M$ has rank $1$; there is a ring endomorphism
    \[
    \bZ_{(p)}[M]\to \bZ_{(p)}[M]; \quad m \mapsto m^p
    \]
    induced from a corresponding endomorphism of $M$, and it lifts the Frobenius of $\bF_p [M]$. It determines the $\delta$-structure as $\bZ_{(p)}[M]$ is $p$-torsion free. 
    
    Using the natural map $M\to \bZ_{(p)}[M]$, we obtain a $\delta_{\log}$-ring $(\bZ_{(p)}[M], M)$ of rank $1$. 
    \item Let $(A, \alpha\colon M\to A)$ be a $\delta_{\log}$-ring and $A\to B$ a map of $\delta$-rings. The composites
    \[
    M\overset{\alpha}{\longrightarrow} A\to B, \quad M\overset{\delta_{\log}}{\longrightarrow}A\to B
    \]
    determine a $\delta_{\log}$-ring $(B, M)$. 
\end{enumerate}
\end{exam}

\begin{rem}\label{W2}
\cite{BS}*{Remark 2.4} says that a $\delta$-structure on a ring $A$ corresponds to a section $w\colon A\to W_2 (A)$ of rings given by $w(x)=(x, \delta (x))$. Given a $\delta$-ring $(A, \delta)$ and a prelog structure $\alpha\colon M\to A$, a $\delta_{\log}$-structure corresponds to a map of (multiplicative) monoids 
\[
w_{\log}\colon M\to W_2(R);\quad m \mapsto (1, \delta_{\log}(m))
\]
satisfying $w(\alpha (m))=(\alpha(m), 0)\cdot w_{\log}(m)$.  
\end{rem}

\begin{rem}[Limits and colimits]
As in the category of $\delta$-rings, all limits and colimits of $\delta_{\log}$-rings exist. They can be computed at the level of underlying rings and monoids, i.e., the forgetful functor from $\delta_{\log}$-rings to prelog rings (or pairs of a ring and a monoid) commutes with limits and colimits. (For colimits, use Remark \ref{W2} as in \cite{BS}*{Remark 2.7}.)

In fact, using this, one can deduce that the forgetful functors admit left adjoint functors. 
Moreover, the left adjoint functors are the identity on the monoid part. 
Also, the forgetful functor from $\delta_{\log}$-rings to prelog rings admits a right adjoint functor. See also Remark \ref{right adjoint} below.  
\end{rem}

\begin{notation}
For $\delta$-rings, we use $\{-\}^{\delta}$ for adjoining elements; this notation is different from \cite{BS}. 
For a monoid $M$, we write $\{M\}_{\log}^{\delta}$ for adjoining $M$ to prelog structures in the theory of $\delta_{\log}$-rings. More precisely, for a $\delta_{\log}$-ring $(A, M_A)$, we obtain $(A\{ M\}_{\log}^{\delta}, M_A\oplus M)$ by applying the left adjoint functor to a prelog ring $(A[M], M_A\oplus M)$.
If $M$ is a free monoid $\bN^I$ with the unique basis $\{x_i \}_{i\in I}$, then we also write $\{x_i\}_{\log}^{\delta}$ instead. 
\end{notation}

\begin{rem}
If $(A, M_A)$ is a $\delta_{\log}$-ring, a natural map $\bZ_{(p)}[M_A]\to A$ induces a map of $\delta_{\log}$-rings:
\[
(\bZ_{(p)}\{M_A\}_{\log}^{\delta}, M_A)\to (A, M_A).
\]
\end{rem}

Let $(\bZ_{(p)}\{ x\}_{\log}^{\delta}, x^{\bN})$ denote the $\delta_{\log}$-ring freely generated by one element $x$ of the prelog structure, and $(\bZ_{(p)}\{ x^{\pm 1}\}_{\log}^{\delta}, x^{\bZ})$ denote the $\delta_{\log}$-ring with $x$ inverted in the prelog structure, i.e., associated with the monoid $x^{\bZ}$; these are the subject of Lemma \ref{one generator} below. 

\begin{lem}[Completions]\label{completions}
Let $(A, M)$ be a $\delta_{\log}$-ring, and $I\subset A$ a finitely generated ideal containing $p$. 
If $A^{\wedge}_{\textnormal{cl}}$ denotes the classical $I$-completion of $A$, then $(A^{\wedge}_{\textnormal{cl}}, M)$ admits a natural $\delta_{\log}$-structure. 
\end{lem}

\begin{proof}
This holds as $A^{\wedge}_{\textnormal{cl}}$ admits a natural $\delta$-structure \cite{BS}*{Lemma 2.17}. 
\end{proof}

\begin{exam}
Let $M$ be a $p$-divisible monoid. 
If $(A, \alpha\colon M\to A)$ is a classically $p$-complete $\delta_{\log}$-ring, then $\delta (\alpha(m))=0$ for $m\in M$ by Lemma \ref{rank 1 elements}. This implies that the natural map
\[
((\bZ_p \{ M \}^{\delta}_{\log})^{\wedge}_{\textnormal{cl}}, M) \to 
(\bZ_p \langle M \rangle, M)
\]
is an isomorphism of $\delta_{\log}$-rings of rank $1$. 
\end{exam}

\begin{lem}\label{one generator}
The ring $\bZ_{(p)}\{ x\}_{\log}^{\delta}$ is a polynomial ring on the set
\[
\{x, \delta_{\log} (x), \delta(\delta_{\log}(x)), \delta^2 (\delta_{\log}(x)), \dots \}
\]
and its Frobenius endomorphism $\phi$ is faithfully flat. 
The ring 
\[
\bQ_{(p)}\{ x\}_{\log}^{\delta}=\bZ_{(p)}\{ x\}_{\log}^{\delta}[1/p]
\]
is also a polynomial ring on the set $\{x, \phi (x)/ x^p, \phi (\phi (x)/ x^p), \phi^2 (\phi (x)/ x^p), \dots \}$. 

A natural map
\[
(\bZ_{(p)}\{ x\}_{\log}^{\delta}[x^{-1}], x^{\bZ}) \to
(\bZ_{(p)}\{ x^{\pm 1}\}_{\log}^{\delta}, x^{\bZ})
\]
induces an isomorphism of prelog rings after the classical $p$-completion.  
\end{lem}

\begin{proof}
We modify the proof of \cite{BS}*{Lemma 2.11} except for the final point. 
Consider the polynomial ring $A=\bZ_{(p)}[x, y_0, y_1, y_2, \dots]$ on countably many generators with a prelog structure $x^{\bN}\subset A$. 
The assignment $x\mapsto x^p(1+ p y_0), y_i \mapsto y_i^p +p y_{i+1}$ define an endomorphism $\phi$ of $A$ that lifts the Frobenius on $A/p$. As $A$ is $p$-torsionfree, there is a unique $\delta$-structure such that $\delta (x)=x^p y_0$, $\delta (y_i)=y_{i+1}$. As $x$ is a nonzerodivisor of $A$ and $\delta (x^n)$ is divisible by $x^{pn}$ for all integers $n\geq 0$, there is a unique $\delta_{\log}$-structure such that $\delta_{\log}(x)=y_0$. 
It is clear that the $\delta_{\log}$-ring $(A, x^{\bN})$ satisfies the same universality of $(\bZ_{(p)}\{ x\}_{\log}^{\delta}, x^ \bN)$, and we have $(\bZ_{(p)}\{ x\}_{\log}^{\delta}, x^ \bN)=(A, x^{\bN})$. The statement for $\bQ_{(p)}\{ x\}_{\log}^{\delta}$ follows immediately by inverting $p$. 

To check the faithfully flatness, we note that
\[
\bZ_{(p)}[1/p][x^p y_0] \subset \bZ_{(p)}[1/p][x, y_0]
\]
is faithfully flat. The argument in \cite{BS} works using this faithfully flatness. 

For the final point, we first note that $\bZ_{p}\langle x, y_0, y_1, \dots\rangle[x^{-1}]$ is a $\delta$-ring by \cite{BS}*{Lemma 2.15} as $\varphi^{n}(x)$ is invertible in this ring for all $n\geq 0$. It is easy to check that there is a unique $\delta_{\log}$-structure with $\delta_{\log}(x^{-1})=-y_0/(1+ py_0)$. Therefore, by passing to the completion, $(\bZ_{p}\langle x^{\pm 1}, y_0, y_1, \dots\rangle, x^{\bZ})$ is also a $\delta_{\log}$-ring. It satisfies the same universality of the completion of $\delta_{\log}$-ring $(\bZ_{(p)}\{ x^{\pm 1}\}_{\log}^{\delta}, x^{\bZ})$ among classically $p$-complete $\delta_{\log}$-rings.  
\end{proof}

\begin{rem}[The right adjoint of the forgetful functor]\label{right adjoint}
Let us describe the right adjoint of the forgetful functor from $\delta_{\log}$-rings to prelog rings, following the original description of \cite{Joyal} in the case of $\delta$-rings. 

Let $(A, M, \alpha\colon M\to A)$ be a prelog ring. 
Consider sets
\[
A'\coloneqq \textnormal{Hom}_{\textnormal{rings}}(\bZ_{(p)}\{ x\}^\delta, A), \quad
M'\coloneqq\textnormal{Hom}_{\textnormal{prelogrings}}((\bZ_{(p)}\{ y\}^\delta_{\log}, y^{\bN}), (A, M)) 
\]
with natural projections $A'\to A, M'\to M$ induced by 
\[
\bZ_{(p)}[x]\to \bZ_{(p)}\{ x\}^\delta, \quad (\bZ_{(p)}[ y], y^{\bN})\to (\bZ_{(p)}\{ y\}^\delta_{\log}, y^{\bN})
\]
respectively. 
We note that $A' \cong A^{\bN}$ and $M' \cong M\times A^{\bN}$ as sets using free generators $\{\delta^n (x)\}_{n \geq 0}$ and $y, \{\delta^n (\delta_{\log} (y))\}_{n \geq 0}$. 

The $\delta$-structure of $\bZ_{(p)}\{ x\}^\delta$ gives rise to a self map $\delta\colon A'\to A'$ induced by $x\mapsto \delta (x)$. 
Using this map $\delta$, we can make $A'$ a $\delta$-ring in a unique way such that the projection $A'\to A$ is a ring map. 
Similarly, the $\delta_{\log}$-structure of $(\bZ_{(p)}\{ y\}^{\delta}_{\log}, y^{\bN})$ gives rise to a map $\delta_{\log}\colon M' \to A'$ induced by sending $\delta^ n (x)$ to $\delta^n (\delta_{\log}(y))$ for integers $n\geq 0$. Now, the $\delta$-ring $A'$ may be upgraded uniquely to a $\delta_{\log}$-ring $(A', M', \alpha'\colon M'\to A')$ with this map $\delta_{\log}$ such that $(A', M')\to (A, M)$ is a map of prelog rings.  
Note that the projection $M\times A^{\bN}\to A^{\bN}$ corresponds to $\delta_{\log}\colon M' \to A'$. While $\alpha'$ has a more complicated expression in terms of these coordinates,  it is simply induced by sending $\delta^n (x)$ to $\delta^n (y)$ for integers $n\geq 0$. 

Joyal proved that $A'$ is isomorphic to the ring of Witt vectors $W(A)$ as $\delta$-rings \cite{Joyal}. 
\end{rem}

\begin{lem}[Etale strict maps]\label{\'etale maps}
Let $(A, M)$ be a $\delta_{\log}$-ring, and $I\subset A$ a finitely generated ideal containing $p$. Assume that $A\to B$ is $I$-completely \'etale. Then, $(B, M)$ admits a natural $\delta_{\log}$-structure compatible with $(A, M)$.  
\end{lem}

\begin{proof}
The only nontrivial part is the existence of $\delta$-structure, which is proved in \cite{BS}*{Lemma 2.18}. 
\end{proof}

$\delta_{\log}$-structures pass to the associated log structures:

\begin{prop}\label{adding units}
Let $(A, M)$ be a $\delta_{\log}$-ring. Consider the following pushout diagram of monoids:
\[
\begin{CD}
\alpha^{-1}(A^\times) @>>> A^\times \\
@VVV @VVV \\
M @>>>  N. 
\end{CD}
\]
There is a unique $\delta_{\log}$-structure on $(A, N)$ compatible with that of $(A, M)$. 
\end{prop}

\begin{proof}
There are unique $\delta_{\log}$-structures on prelog structures
\[
A^\times \to A, \quad \alpha\colon \alpha^{-1}(A^\times)\to A,
\]
and the above diagram may be upgraded to a pushout diagram of $\delta_{\log}$-rings with the common underlying ring $A$. So, the assertion holds.  
\end{proof}

\begin{cor}[Log structures]\label{sheafification}
Let $(A, M)$ be a $\delta_{\log}$-ring. 
Assume that $A$ is classically $I$-complete for some finitely generated ideal $I$ containing $p$. Consider the associated log $I$-adic formal scheme
\[
(X, M_X)=(\Spf (A), \underline{M}^a)
\]
of $M\to A$, and set
\[
M^a_U\coloneqq \Gamma (U, M_X), \quad (B, M)^a\coloneqq (B, M^a_U), 
\]
for any affine $U=\Spf (B)$ \'etale over $\Spf (A)$. 
Then, $(B, M)^a$ admits a unique $\delta_{\log}$-structure compatible with $(A, M)$ and its \'etale localizations. 
\end{cor}

Therefore, we can talk about a $\delta_{\log}$-structure on a log structure. 

\begin{proof}
Let $A \to B$ be an $I$-completely \'etale map. 
Recall that the $\delta$-structure on $A$ passes to $B$ by Lemma \ref{\'etale maps}. 
Proposition \ref{adding units} implies that the $\delta_{\log}$-structure extends to the log ring associated with $(B, M)$. Taking the sheafification and global sections, we get the $\delta_{\log}$-structure on $(A, M^a)$. 
(To see that it is indeed a $\delta_{\log}$-ring, look at stalks, for instance.)
\end{proof}

We can also treat a special case of formally \'etale maps:

\begin{prop}\label{change of monoids}
Let $(A, M, \alpha\colon M\to A)$ be a $\delta_{\log}$-ring with $M$ integral. 
If $\textnormal{rad}(A)$ contains $p$, then there exists a unique map $\delta_{\log}\colon M^{\gp}\to A$ satisfying
\[
\delta_{\log}(mm')=\delta_{\log}(m)+\delta_{\log}(m')+p \delta_{\log}(m)\delta_{\log}(m') \tag{$*_{\gp}$}
\]
for $m, m' \in M^{\gp}$. 
Moreover, for any submonoid $N\subset M^{\gp}$ containing $M$, there exists a unique $\delta$-structure on $A\otimes_{\bZ_{(p)}[M]}\bZ_{(p)}[N]$ making $(A\otimes_{\bZ_{(p)}[M]}\bZ_{(p)}[N], N)$ a $\delta_{\log}$-ring over $(A, M)$ using $\delta_{\log}$ above. 
The map 
\[
(A, M)\to (A\otimes_{\bZ_{(p)}[M]}\bZ_{(p)}[N], N)
\]
is universal among maps of $\delta_{\log}$-rings $(A, M)\to (B, N)$ compatible with the inclusion $M \subset N$. The construction of this universal map commutes with base changes $A\to A'$.  
\end{prop}

\begin{proof}
For an element $m'/m \in M^{\gp}$, set
\[
\delta_{\log}(m'/m)=\frac{\delta_{\log}(m')-\delta_{\log}(m)}{1+ p \delta_{\log}(m)} \in A. 
\]
It is straightforward to check that this is well-defined and satisfies ($*_{\gp}$). 
The uniqueness assertions are clear; the $\delta_{\log}$ above determines $\delta$ on (the image of) $N$, hence on the whole $A\otimes_{\bZ_{(p)}[M]}\bZ_{(p)}[N]$. 
Once the desired $\delta$-structure is shown to exist, the rest of the statement follows easily. 

To discuss the existence of $\delta$-structures, we first assume that $M$ is free. 
The prelog structure $M\to A$ determines a natural map of prelog rings 
\[
(A[M], M\hookrightarrow A[M])\to (A, M)
\]
and in turn induces a surjection of $\delta_{\log}$-rings
\[
(A\{M\}_{\log}^{\delta}, M)\to (A, M),
\]
where $(A\{M\}_{\log}^{\delta}, M)$ is obtained by adjoining $M$ to a $\delta_{\log}$-ring $(A, \{e\})$. 

To show that $A\otimes_{\bZ_{(p)}[M]}\bZ_{(p)}[N]$ admits a $\delta$-structure, it suffices to prove the same statement for the $p$-localization $((1+(p))^{-1} A\{M\}_{\log}^{\delta}, M)$ and $N$; the original case is obtained by base change along $(1+(p))^{-1} A\{M\}_{\log}^{\delta} \to A$. 
(The $\delta$-structure of $A\{M\}_{\log}^{\delta}$ extends to the localization $(1+(p))^{-1} A\{M\}_{\log}^{\delta}$ as in \cite{BS}*{Remark 2.16}.)
In particular, using Lemma \ref{one generator}, we may and do assume that $\alpha$ is injective and $\alpha(M)$ consists of nonzerodivisor after replacing $(A, M)$. 
Further, we may assume that $\alpha(M)$ are still nonzerodivisors in $A\otimes_{\bZ_{(p)}[M]}\bZ_{(p)}[N]$. 

Now, $A\otimes_{\bZ_{(p)}[M]}\bZ_{(p)}[M^{\gp}]$ is identified with the localization $\alpha(M)^{-1}A$ with the inclusion $A\subset \alpha(M)^{-1}A$. 
Since
\[
\phi^n (\alpha(m))\in \alpha(m)^{p^n} (1+pA)
\]
for any integer $n\geq 0$, we have $\alpha(M)^{-1}A=S^{-1}A$ for $S=\bigcup_n \phi^n (\alpha(M))\subset A$. Thus, this localization admits a $\delta$-structure by \cite{BS}*{Lemma 2.15}. 
It remains to show that $A\otimes_{\bZ_{(p)}[M]}\bZ_{(p)}[N]\subset \alpha(M)^{-1}A$ is $\delta$-stable. This is clear since 
\[
\delta(n)=n^p \delta_{\log}(n)\in A\otimes_{\bZ_{(p)[M]}}\bZ_{(p)}[N].
\]

Assume again that $M$ is free. 
We can also deduce from the previous paragraphs that, similar to Lemma \ref{one generator}, there is a natural map of $\delta_{\log}$-rings
\begin{equation}
(1+ p A\{M\}_{\log}^{\delta})^{-1 }A\{M\}_{\log}^{\delta} \otimes_{\bZ_{(p)}[M]}\bZ_{(p)}[N] \to
A\{N\}_{\log}^{\delta}
\end{equation}
that becomes an isomorphism after the $p$-localization. 

Now consider a general $M$. Then, we have a surjection of prelog rings
\[
(A[\bN^{M}], \bN^M) \to (A, M); \quad 1_{m} \mapsto m. 
\]
It induces a surjection of $\delta_{\log}$-rings $(A\{\bN^M\}_{\log}^{\delta}, \bN^M)\to (A, M)$, where $\bN^M$ is adjoined to a $\delta_{\log}$-ring $(A, \{e\})$. 
Write $\widetilde{M}, \widetilde{N}$ for the inverse images of $M, N\subset M^{\gp}$ in $\bZ^M$ respectively. Note that $\widetilde{M}$ contains $\bN^M$. 
To show that $A\otimes_{\bZ_{(p)}[M]}\bZ_{(p)}[N]$ admits a $\delta$-structure, 
it suffices to prove the same statement for 
\begin{center}
$((1+p A\{\widetilde{M}\}_{\log}^{\delta})^{-1} A\{\widetilde{M}\}_{\log}^{\delta}, \widetilde{M})$ and $\widetilde{N}$. 
\end{center}
By the previous case we have discussed, we know that
\begin{multline*}
(1 + p A\{\bN^M\}_{\log}^{\delta})^{-1} A\{\bN^M\}_{\log}^{\delta} \otimes_{\bZ_{(p)}[\bN^M]}\bZ_{(p)}[\widetilde{N}]\\
= 
((1 + p A\{\bN^M\}_{\log}^{\delta})^{-1} A\{\bN^M\}_{\log}^{\delta} \otimes_{\bZ_{(p)}[\bN^M]}\bZ_{(p)}[\widetilde{M}])\otimes_{\bZ_{(p)}[\widetilde{M}]}\bZ_{(p)}[\widetilde{N}]
\end{multline*}
admits a $\delta$-structure. 
Using the map (1) for $\bN^M$ and $\widetilde{M}$, which is a map of $\delta$-rings, we see that
\[
(1 + p A\{\widetilde{M} \}_{\log}^{\delta})^{-1} A\{\widetilde{M} \}_{\log}^{\delta}
\otimes_{\bZ_{(p)}[\widetilde{M}]}\bZ_{(p)}[\widetilde{N}]
\]
is also a $\delta$-ring. This finishes the existence of $\delta$-structure in the general case, and we complete the proof. 
\end{proof}

Recall that a \emph{$\delta$-pair} is a pair $(A, I)$ of a $\delta$-ring $A$ and an ideal $I$ of $A$. 
A \emph{$\delta_{\log}$-triple} is a triple $(A, I, M)$, where $(A, M)$ is a $\delta_{\log}$-ring and $I$ is an ideal of $A$. 
The category of $\delta$-triples is defined in an obvious way. 

We say that a map $(A, M)\to (B, N)$ of prelog rings (with integral monoids) is \emph{exact surjective} if it is surjective, i.e., $A\to B$, $M\to N$ are both surjective, and $M/M^\times \cong N/N^\times$; these induce exact closed immersions of associated (integral) log schemes, and one can treat them as if they are closed immersion of schemes. It it a useful technique in log geometry to pass to exact surjective maps. 

\begin{const}[Exactification of $\delta_{\log}$-triples]\label{Exactification}
Let $(A, I, M)$ be a $\delta_{\log}$-triple, $(A/I, N)$ a prelog ring and $(A, M)\to (A/I, N)$ a surjective map of prelog rings.  
Assume that $M$ and $N$ are integral, and write $h\colon M \to N$, $\overline{h}\colon M \to N/ N^\times$ with the associated maps $h^{\gp}\colon M^{\gp}\to N^{\gp}$, $\overline{h}^{\gp}\colon M^{\gp}\to (N/ N^\times)^{\gp}$ of abelian groups. 
Set $M'=(h^{\gp})^{-1}(N)=(\overline{h}^{\gp})^{-1}(N/N^\times)$. By the surjectivity of $h$, $M'$ is generated by $M$ and $(h^{\gp})^{-1}(\{e\})$. 

Assume that $A$ is classically $p$-complete. 
Applying Proposition \ref{change of monoids} to $(A, M)$ and $M\subset M'$, we obtain a $\delta_{\log}$-ring $(A', M')$ over $(A, M)$. By the same proposition, the natural map $M'\to N$ of monoids lifts the map $(A, M)\to (A/I, N)$ to an \emph{exact} surjection $(A', M')\to (A/I, N)$ of prelog rings. Let us write $I'$ for the kernel of the surjection $A'\to A/I$. Then, $(A', I', M')$ is a $\delta_{\log}$-triple with an exact surjection $(A', M')\to (A'/I', N)=(A/I, N)$ of prelog rings. We call $(A', I', M')$ the \emph{exactification} of $(A, I, M)$ along $(A, M)\to (A/I, N)$. 
This construction is functorial, and the formation of $(A', M')$ commutes with base changes on $A$. 
\end{const}

\begin{rem}[Exactification and integral maps]\label{Exactification and integral maps}
Recall that a map of integral monoids $M\to N$ is \emph{integral} if every pushout 
\[
\begin{CD}
M @>>> N \\
@VVV @VVV \\
M' @>>> N'
\end{CD}
\]
along a map $M\to M'$ of integral monoids
gives an integral monoid $N'$. It is clear that integral maps of integral monoids are preserved by base change.

In the setup of Construction \ref{Exactification}, assume that $(A, M)\to (A/I, N)$ lives over some base $(B, M_B)$, and assume that $M_B\to N$ is an integral map of integral monoids. Then, $N\to M'$ is integral as $M'\to N$ is exact; see \cite{Ogus}*{I.4.6.3.1}. 
\end{rem}

\section{Logarithmic prisms}
Now we recall the notion of prisms defined in \cite{BS}, and introduce two logarithmic variants: prelog prisms and log prisms. 

A \emph{prism} is a $\delta$-pair $(A, I)$ such that $I$ defines a Cartier divisor of $\Spec (A)$, $A$ is $(p,I)$-complete, and $p\in I +\phi(I)A$. Prisms form a full subcategory of the category of all $\delta$-pairs. 
Note that $I$ is finitely generated by Zariski descent. 

A prism is \emph{bounded} if $A/I$ has bounded $p^{\infty}$-torsion, in which case $A$ is classically $(p, I)$-complete \cite{BS}*{Lemma 3.7}.  
A prism is \emph{orientable} if $I$ is principal, and \emph{oriented} if it is orientable and a generator of $I$ is fixed. 

One of the most important properties of prisms is the \emph{rigidity}: if $(A, I)\to (B,J)$ is a map of prisms, then $I\otimes_A B\cong J$ canonically. This leads to the notion of \emph{prismatic envelopes}:

\begin{prop}[Existence of prismatic envelopes]\label{Bhatt:lecture}
Fix a prism $(A, I)$. Let $(B, J)$ be a $\delta$-pair over $(A,I)$, i.e., equipped with a map $(A, I)\to (B, J)$. If $(A, I)$ is orientable, then there exists a universal map $(B,J)\to (B', IB')$ of $\delta$-pairs over $(A, I)$ to a prism $(B', IB')$.  
\end{prop}

\begin{proof}
This is \cite{Bhatt:lecture}*{V. Lemma 5.1}. 
\end{proof}

In general, there is no control of the prismatic envelope $(B', IB')$. If $J$ is generated by a regular sequence, it behaves well. 

\begin{prop}[Prismatic envelopes for regular sequences]\label{prismatic envelopes for regular sequenes}
Using the notation of Proposition \ref{Bhatt:lecture}, assume that $(A, I)$ is bounded and $B$ is $(p,I)$-completely flat over $A$ and 
\[
J=(I, x_1, \dots, x_r)\subset B
\]
for a $(p, I)$-completely regular sequence $x_1, \dots, x_r$ relative to $A$ in the sense of \cite{BS}*{Definition 2.41}. 
Then, the prismatic envelope $(B', IB')$ is $(p,I)$-completely flat over $(A, I)$, and its formation commutes with base changes $(A, I)\to (A',I')$ and $(p,I)$-completely flat base changes on $B$. 
\end{prop}

In particular, we see that $(B', IB')$ is a bounded prism and $B'$ is classically $(p,I)$-complete by \cite{BS}*{Lemma 3.7}. 

\begin{proof}
This is \cite{BS}*{Proposition 3.13}\footnote{This is stated for not necessarily orientable prisms $(A, I)$. As pointed out by Mikami, the descent argument there is nontrivial if one considers maps from $(B, J)$ to unbounded prisms. The notion of animated prism in \cite{Bhatt-Lurie} would suffice to justify the argument in such generality, but let us remark that we essentially only use the universality for maps to bounded prisms.}. 
\end{proof}

We are now in position to introduce logarithmic variants. 

\begin{defn}[Prelog prisms and log prisms]
A (bounded) \emph{prelog prism} is a $\delta_{\log}$-triple $(A, I, M)$, where $(A, I)$ is a (bounded) prism. Prelog prisms form a full subcategory of the category of all $\delta_{\log}$-triples. 
A prelog prism is of \emph{rank $1$} if $\delta_{\log}=0$. 
If the underlying prism is bounded, any prelog prism gives rise to a log structure $M^a_{\Spf (A)}$ on the $(p,I)$-adic formal scheme $\Spf (A)$ with a $\delta_{\log}$-structure. 
(Recall that $A$ is classically $(p,I)$-complete if $(A, I)$ is bounded.)
A \emph{log prism} $(A, I, M_{\Spf (A)})$ is a bounded prism $(A, I)$ with a log structure $M_{\Spf (A)}$ and its $\delta_{\log}$-structure that comes from some prelog prism $(A, I, M)$. 
A map of log prisms is a map of the underlying log formal schemes that induces a map of the underlying prisms and preserves $\delta_{\log}$-structures.  
\end{defn}

If $(A, I, M)$ is a bounded prelog prism, we have the associated log prism
\[
(A, I, M)^a\coloneqq (A, I, M^a_{\Spf(A)})
\]
by Corollary \ref{sheafification}, and this construction is functorial. 
Conversely, if $(A, I, M_{\Spf (A)})$ is a log prism, then $(A, I, \Gamma (\Spf (A), M_{\Spf (A)}))$ is a prelog prism. 

\begin{exam}
\begin{enumerate}
\item For any bounded prism $(A,I)$, the trivial log structure defines a log prism $(A, I, \cO_{\Spf (A)}^\times)$. 
\item For any prism $(A,I)$, a prelog structure $\bN \to A; 1\mapsto 0$ defines a prelog prism $(A, I, \bN)$ of rank $1$. 
\item Let $(A, I)$ be a \emph{perfect prism}, equivalently, $R= A/I$ is a perfectoid ring (in the sense of \cite{BMS1}) and $A=W(R^\flat)$ \cite{BS}*{Theorem 3.10}. Using Teichm\"uller lifts $[-]\colon R^\flat\to A$ as a prelog structure, $(A, I, R^\flat)$ is a prelog prism of rank $1$.
We will be mainly interested in the case where $R[1/p]$ is a perfectoid field $C$, and consider
\[
(A, I, R^\flat \setminus\{0\})=(A_{\Inf}, (\xi), \cO^\flat_C\setminus\{0\})
\]
instead, which makes sense more generally if $R^\flat$, equivalently $R$, is an integral domain. 
\item Let $\cO_K$ be a totally ramified finite extension of $W(k)$ for a perfect field $k$ of characteristic $p$ with uniformizer $\pi$. Let $E(u)\in W(k)[u]$ be an Eisenstein polynomial with $E(\pi)=0$. 
The \emph{Breuil-Kisin prelog prism} is
\[
(W(k)[\![ u]\!], (E(u)), \mathbb{N}\to W(k)[\![ u]\!] ; n \mapsto u^n)
\]
with $\delta_{\log}=0$. 
\end{enumerate}
\end{exam}

A natural analogue of the rigidity of prisms \cite{BS}*{Lemma 3.5} holds for prelog prisms. In particular, we see that if $(A, I, M_A)$ is a prelog prism and $A\to B$ is a map of $\delta$-rings with $B$ being $(p,I)$-complete, then $(B, IB, M_A)$ is a prelog prism exactly when $B[I]=0$. 
This condition is satisfied if $(A, I)$ is bounded and $B$ is $(p,I)$-completely flat over $A$ \cite{BS}*{Lemma 3.7 (3)}. 

\begin{rem}[Maps from prelog prisms]\label{Maps from prelog prisms}
Let $(A, I, M_A)$ be a bounded prelog prism and let $(B, J, M_{\Spf (B)})$ be a log prism. 
Any map of prelog prisms $(A, I, M_A) \to (B, J, \Gamma (\Spf (B), M_{\Spf (B)}))$ induces a map of log prisms $(A, I, M_A)^a \to (B, J, M_{\Spf (B)})$ (and this is clearly a one-to-one correspondence by taking global sections). Indeed, $J\cong IB$ and we obtain a map from the constant prelog structure on $\Spf (B)$ determined by $M_A$ to $M_{\Spf (B)}$ compatibly with $\delta_{\log}$-structures. It then induces a map from the associated log structure to $M_{\Spf (B)}$ compatibly with $\delta_{\log}$-structures. 

Note however that the canonical map $(B, J, \Gamma (\Spf (B), M_{\Spf (B)}))^a \to (B, J, M_{\Spf (B)})$ may not be an isomorphism.
\end{rem}

\begin{prop}[Existence of prelog prismatic envelopes]\label{Existence of prelog prismatic envelopes}
Fix an orientable prelog prism $(A, I, M_A)$ with $M_A$ integral. Let $(B, J, M_B)$ be a $\delta_{\log}$-triple over $(A, I, M_A)$, i.e., with a map $(A, I, M_A)\to (B, J, M_B)$, and let $(B, M_B)\to (B/J, N)$ be a surjection of prelog rings with integral monoids. Then,
\begin{center}
there is a universal map $(B,J, M_B)\to (B', IB', M_{B'})$
\end{center}
of $\delta_{\log}$-triples over $(A, I, M_A)$ to a prelog prism $(B', IB', M_{B'})$ with an \emph{exact} surjection $(B', M_{B'})\to (B'/IB', N)$. 
Moreover, the monoid $M_{B'}$ is integral. 
\end{prop}

\begin{proof}
Applying Construction \ref{Exactification}, we may assume that $(B, M_B)\to (B/J, N)$ is an exact surjection. Then, the statement follows from Proposition \ref{Bhatt:lecture}. 
\end{proof}

The universal object $(B', IB', M_{B'})$ in the previous proposition is called the \emph{prelog prismatic envelope}. (We later use the same terminology for such an object for non-orientable $(A, I, M_A)$.)
If $(B', IB', M_{B'})$ is bounded, then we have the associated log prism $(B', IB', M_{B'})^a$ and the surjection $(B', M_{B'})\to (B'/IB', N)$ induces an exact closed immersion of associated log $(p, I)$-adic formal schemes. 
It gives the \emph{log prismatic envelope}:

\begin{prop}[Existence of log prismatic envelopes]\label{Existence of log prismatic envelopes}
We use the notation of Proposition \ref{Existence of prelog prismatic envelopes}. 
Assume that $(B', IB', M_{B'})$ is bounded. 
The construction of the log prism $(B', IB', M_{B'})^a$ is universal in the sense that
the following commutative diagram
\[
\begin{CD}
(\Spf (B'/IB'), N^a_{\Spf (B'/IB')}) @>>> (\Spf (B'), M^a_{B', \Spf (B')}) \\
@VVV @VVV \\
(\Spf (B/J), N_{\Spf (B/J)}^a) @>>> (\Spf (B), M_{B, \Spf (B)}^a), 
\end{CD}
\]
is final in the category of commutative diagrams
\[
\begin{CD}
(\Spf (C/IC), N^a_{\Spf (C/IC)}) @>>> (\Spf (C), M_{\Spf (C)}) \\
@VVV @VVV \\
(\Spf (B/J), N_{\Spf (B/J)}^a) @>>> (\Spf (B), M_{B, \Spf (B)}^a), 
\end{CD}
\]
where $(C, IC, M_{\Spf (C)})$ is a log prism with an integral log structure, and the top arrow is an exact closed immersion of log $(p,I)$-adic formal schemes. 
\end{prop}

\begin{lem}\label{exact closed immersion}
Put
\[
N^a_{C/I}\coloneqq \Gamma (\Spf (C/IC), N^a_{\Spf (C/IC)}).
\]
The map $(C, \Gamma (\Spf (C), M_{\Spf (C)}))\to (C/IC, N^a_{C/I})$ is exact surjective.  
\end{lem}

\begin{proof}
\cite{Beilinson}*{p.3, Exercises (iii)} implies that, for any integers $m, n\geq 1$, the exact closed immersion 
\[
(\Spf (C/(p^m, I)), M_{\Spf (C/(p^m, I))}) \hookrightarrow (\Spf (C/(p^m, I^{n})), M_{\Spf (C/(p^m, I^{n}))}),
\]
where $M_{?}$ denotes the restriction of $M_{\Spf (C)}$ (as a log structure),  gives rise to an exact surjection
\begin{multline*}
(C/(p^m, I^{n}), \Gamma (\Spf (C/(p^m, I^n), M_{\Spf (C/(p^m, I^{n}))})) \\
\to (C/(p^m, I), \Gamma (\Spf (C/(p^m, I), M_{\Spf (C/(p^m, I))})), 
\end{multline*}
which is a $(1+ IC/(p^m, I))$-torsor on monoids. By passing to the limits, we obtain an exact surjection
\begin{multline*}
(C, \Gamma (\Spf (C), M_{\Spf (C)}))\cong \varprojlim_{m,n} (C/(p^m, I^{n}), \Gamma (\Spf (C/(p^m, I^n)), M_{\Spf (C/(p^m, I^{n}))})) \\
\to 
\varprojlim_m (C/(p^m, I), \Gamma (\Spf (C/(p^m, I)), M_{\Spf (C/(p^m, I))})) \\
\cong (C/I, \Gamma (\Spf (C/I), M_{\Spf (C/I)})) \cong
(C/I, N^a_{C/I}), 
\end{multline*}
which is a $(1+IC)$-torsor on monoids. 
\end{proof}

\begin{proof}[Proof of Proposition \ref{Existence of log prismatic envelopes}]
Take the fiber product $\Gamma (\Spf (C), M_{\Spf (C)})\times_{N^a_{C/I}} N$. The natural map $\Gamma (\Spf (C), M_{\Spf (C)})\times_{N^a_{C/I}} N\to C$ is a chart of $M_{\Spf (C)}$ on $\Spf (C)$ since $(\Spf (C/IC), N^a_{\Spf (C/IC)}) \hookrightarrow (\Spf (C), M_{\Spf (C)})$ is an exact closed immersion and the underlying affine $(p, I)$-adic formal schemes are the same modulo topologically nilpotent $I$; cf. \cite{Beilinson}*{footnote 6} and the proof of Lemma \ref{exact closed immersion} above.
By Lemma \ref{exact closed immersion}, $(C, \Gamma (\Spf (C), M_{\Spf (C)}))\to (C/IC, N^a_{C/I})$ is exact surjective. Thus, we see that 
\[
(C, \Gamma (\Spf (C), M_{\Spf (C)})\times_{N^a_{C/I}} N)\to (C/IC, N)
\]
is also exact surjective as base change preserves exact surjective maps of monoids. 

We may regard $(C, IC, \Gamma (\Spf (C), M_{\Spf (C)})\times_{N^a_{C/I}} N)$ as a bounded prelog prism using the $\delta_{\log}$-structure induced by $\delta_{\log}\colon M_C\to C$ so that 
\[
(C, IC, \Gamma (\Spf (C), M_{\Spf (C)})\times_{N^a_{C/I}} N)^a\cong (C, IC, M_{\Spf (C)}).
\]
Applying Proposition \ref{Existence of prelog prismatic envelopes} to $(C, IC, \Gamma (\Spf (C), M_{\Spf (C)})\times_{N^a_{C/I}} N)$, we obtain a map of bounded prelog prisms \[
(B', IB', M_{B'})\to (C, IC, \Gamma (\Spf (C), M_{\Spf (C)})\times_{N^a_{C/I}} N).
\]
Passing to the associated log prisms, we obtain $(B', IB', M_{B'})^a \to (C, IC, M_{\Spf (C)})$. 
This means $(B', IB', M_{B'})^a$ satisfies the universality we want. 
\end{proof}

It is important to control prelog/log prismatic envelopes under certain smoothness assumption. In particular, we need to control their boundedness and flatness. The proposition below is a technical core to compute log prismatic cohomology. 
Before stating the proposition, we need to introduce several notions in log geometry. These are summarized in Appendix, and we briefly explain them here for reader's convenience.  

We say that a map of prelog rings $(A, M)\to (B, N)$ with $A, B$ being classically $p$-complete is \emph{$p$-completely smooth} if the associated log $p$-adic formal scheme $(\Spf (B), N)^a$ is \emph{smooth} over $(A, M_A)$ in the sense of Definition \ref{smooth}; our class of smooth log (formal) schemes is different from smooth fine log (formal) schemes in the literature: 1. we do not impose any finiteness condition on the base chart $M_A \to A$ (although this makes the smoothness possibly dependent on the choice of the chart), 2. we require a smooth morphism be integral; see Appendix for the precise definition and its properties.

Some constructions later involve possibly huge monoids, and we will put them a mild finiteness condition.  
A map of monoids $h\colon M\to N$ is \emph{finitely generated} (resp. \emph{weakly finitely generated}) if there exist finitely many elements of $N$ that, together with the image $h(M)$ (resp. $h(M)$ and $N^\times$), generate $N$. We say that $N$ is (weakly) finitely generated over $M$ if $h$ is so. Being weakly finitely generated is actually a mild condition, and it holds locally for global sections of log structures in practice; see Lemma \ref{small neighborhood}. 

Also recall the following property of an integral map of monoids $M\to N$: the induced map $\bZ[M]\to \bZ[N]$ of the associated rings is flat if and only if $M\to N$ is integral and injective \cite{Kato:log}*{4.1}. In particular, integral injective maps are stable under pushouts. 

Here is the key proposition:

\begin{prop}[Prelog prismatic envelopes for smooth log algebras]\label{Prelog prismatice envelops for smooth log algebras}
Fix a bounded prelog prism $(A,I, M_A)$ with integral $M_A$. 
Let $(B_0, M_B)$ be a prelog ring over $(A, M_A)$ with $M_B$ integral and let $(B_0, M_B)\to (B_0/J, N)$ be a surjection of prelog rings with kernel $J\subset B_0$ to a $p$-completely smooth prelog ring $(B_0/J, N)$ over $(A/I, M_A)$. Assume that $M_A\to N$ is \emph{integral} and the following condition $(*)$ holds
\begin{center}
  $(*)$: $M_A\to M_B$ is injective and integral, the cokernel $M_B^{\gp}/M_A^{\gp}$ is a free abelian group, and $B_0$ is the $(p,I)$-completely free over the $(p, I)$-completion of $A\otimes_{\bZ_{(p)}[M_A]}\bZ_{(p)}[M_B]$. 
\end{center}
Let $(B, M_B)$ denote the $(p,I)$-completed universal $\delta_{\log}$-ring over $(A, M_A)$ generated by $(B_0, M_B)$. 
If $N$ is \emph{weakly finitely generated} over $M_A$, then the prelog prismatic envelope $(B', IB', M_{B'})$ of $(B, (JB)^{\wedge}_{(p,I)}, M_B)$ exists and it is \emph{$(p,I)$-completely flat} over $(A, I, M_A)$. Moreover, its formation commutes with base changes on $(A, I, M_A)$ and $(p, I)$-completely flat base changes on $B_0$. 
In particular, the prelog prism $(B', IB', M_{B'})$ is bounded. 
\end{prop}

\begin{rem}
All the assumptions are preserved by base changes $(A, I, M_A)\to (A', IA', M_{A'})$. Note that the relevant pushouts of monoids are automatically integral as $M_A\to M_B$, $M_A\to N$ are integral.
\end{rem}

A typical example is
\[
B=(A\{(X_s)_{s\in S}\}^\delta \{\bN^T\}_{\log}^{\delta})^{\wedge}_{(p,I)}, \quad
B_0=A\langle (X_s)_{s\in S}, \bN^T\rangle, \quad
M_B=M_A\oplus \bN^T
\]
for some sets $S$, $T$. (Note that $B$ is $(p,I)$-completely free over $B_0$ by Lemma \ref{one generator}.)

\begin{proof}
First note that $\bZ_{(p)}[M_B]$ is flat over $\bZ_{(p)}[M_A]$ as $M_A\to M_B$ is injective and integral. Hence the $(p,I)$-completion appearing in $(*)$ agrees with the classical $(p, I)$-completion. Similarly, we see that $B_0/J$ is $p$-completely flat over $A/I$, using the $p$-completely smoothness of $(B_0/J, N)$ over $(A/I, M_A)$.  

Let $M'_A\subset M_A^{\gp}$ denote the inverse image of $N \subset N^{\gp}$ under $M_A^{\gp} \to N^{\gp}$. Then, one can check that the natural maps $M_A \hookrightarrow M'_A$, $M'_A \to N$ are integral as $M_A\to N$ is integral by using \cite{Kato:log}*{4.1}. 
In particular, the natural map
\[
A\to A'\coloneqq (A\otimes_{A[M_A]}A[M'_A])^{\wedge}_{(p,I)} 
\]
is $(p, I)$-completely flat. It may be regard as a map of $\delta_{\log}$-rings by Proposition \ref{change of monoids}, hence a map of bounded prelog prisms. 
Also note that $(B_0/J, M_B)$ lives over $(A'/ IA', M'_A)$ and it is smooth. 
After replacing $(A, M_A)$ (resp. $(B_0, M_B)$) by $(A', M'_A)$ (resp. its base change along $(A, M_A) \to (A', M'_A)$) but keeping $(B_0/J, N)$, the prelog prismatic envelope $(B', IB', M_{B'})$ does not change because of its universality. 
So, we may assume that the inverse image of $N$ is $M_A$, i.e., $M_A\to N$ is exact. (For the definition of exactness, see \S \ref{crystalline comparison}.) 
This also implies that the inverse image of $N^\times$ in $M_A^{\gp}$ is $M_A^\times$.  

Write $h$ for $M_B\to N$ and set $M'_B=(h^{\gp})^{-1}(N)$. 
This gives the exactification $(B_0, M_B)\to (B''_0, M'_B)$ of $(B_0, M_B)\to (B_0/J, N)$, and the condition $(*)$ continues to hold by Remark \ref{Exactification and integral maps}. Let $J''$ denote the kernel of $B''_0\to B_0/J$ and $B''$ denote the $(p,I)$-completed base change of $B$ along $B_0\to B''_0$. 
Then, the exactification of $(B, (JB)^{\wedge}_{(p,I)}, M_B)$ along $(B, M_B)\to (B/(JB)^{\wedge}_{(p,I)}, N)$ can be written as $(B'', (J''B'')^{\wedge}_{(p,I)}, M'_B)$.  
The prelog prismatic envelope of $(B, (JB)^{\wedge}_{(p,I)}, M_B)$ is the same as the prelog prismatic envelope of $(B'', (J''B'')^{\wedge}_{(p,I)}, M'_B)$. Since the formation of $(B''_0, M'_B)$ and $J''$ commutes with base changes on $(A, I, M_A)$ and $(p, I)$-completely flat base changes on $B_0$, we may and do assume that $(B_0, M_B)\to (B_0/J, N)$ is exact surjective. 
(For the claim about base changes on $(A, I, M_A)$, use that the exact morphisms of integral monoids are stable under pushouts in the category of integral monoids.)
In the rest of the proof, prelog prismatic envelopes involved are just prismatic envelopes. 

Now we use the assumption that $M_A\to N$ is weakly finitely generated to write $(B_0, M_{B})$ as the $(p, I)$-completed colimit of a filtered diagram ${(B_s, M_{s})}_{s\in S}$ of $(p,I)$-completely smooth $(A,M_A)$-algebras with integral monoids and $(p,I)$-completely faithfully flat transition maps. 
This is indeed possible because $M_B=(h^{\gp})^{-1}(N)$ is generated by $M_A$ and $(h^{\gp})^{-1}(N^{\times})$ together with finitely many elements; $(h^{\gp})^{-1}(N^{\times})$ is a group, and its quotient $(h^{\gp})^{-1}(N^{\times})/M_A^{\times}$ is embedded into $M_B^{\gp}/M_A^{\gp}$ and isomorphic to $\bZ^T$ for some $T$. 
(Any subgroup of a free abelian group is free.)
Thus, the quotient can be written as a union of finitely generated abelian free subgroups $\bZ^{T'}$, $T'\subset T$. 
One may assume that $(B_s, M_{s})\to (B_0/J, N)$ is exact surjective for all $s\in S$. 

Using that the kernel of $B_s\to B_0/J$ is Zariski locally generated by a $(p,I)$-completely regular sequence relative to $A$, one deduces from Proposition \ref{prismatic envelopes for regular sequenes} by Zariski descent that the statement of the proposition holds with $(B_0, J, M_B)$ replaced by $(B_s, J_s, M_s)$, where $J_s$ denotes the kernel of $B_s\to B_0/J$. We finish by taking the $(p,I)$-completed filtered colimit with respect to $s$. 
\end{proof}

The following (perhaps more convenient) variant holds:

\begin{prop}\label{Prelog prismatice envelops for smooth log algebras:variant}
Fix a bounded prelog prism $(A,I, M_A)$ with integral $M_A$. 
Let $(B, M_B)$ be a $(p,I)$-completely smooth $\delta_{\log}$-ring over $(A, M_A)$. Let $(B, M_B)\to (R, P)$ be a surjection of prelog rings to a $p$-completely smooth prelog ring $(R, P)$ over $(A/I, M_A)$. 
Assume that $M_A\to M_B$ is a smooth chart and $M_A\to P$ is integral. 
Then the prelog prismatic envelope $(B', IB', M_{B'})$ of $(B, M_B)\to (R, P)$ exists and it is \emph{$(p,I)$-completely flat} over $(A, I, M_A)$. Moreover, its formation commutes with base changes on $(A, I, M_A)$. 
In particular, the prelog prism $(B', IB', M_{B'})$ is bounded. 
\end{prop}

\begin{proof}
We may assume that $M_B\to P$ is exact surjective by Lemma \ref{change of monoids} and Remark \ref{Exactification and integral maps}.
Then, the kernel of $B\to R$ is Zariski locally generated by a $(p,I)$-completely regular sequence relative to $A$, so the prismatic envelope is $(p,I)$-completely flat over $A$. 
\end{proof}

\section{The logarithmic prismatic site}
We define a logarithmic variant of the prismatic site of \cite{BS}. 
Fix a bounded prelog prism $(A, I, M_A)$ with the associated log prism $(A, I, M_A)^a$. All log formal schemes are assumed to be $(p, I)$-adic. 

\subsection{Definition of the logarithmic prismatic site}
Suppose $M_A$ is integral and fix a log formal scheme $(X, M_X)$ smooth over $(A/I, M_A)$. By our convention on smoothness, $M_X$ is an integral log structure. 

\begin{defn}[Logarithmic prismatic site]\label{log prismatic site}
Let $((X,M_X)/ (A, M_A))_{\Prism}$ be the opposite of the category with objects described as follows: an object consists of
\begin{itemize}
 \item a log prism $(B, IB, M_{\Spf (B)})=(B, IB, M_B)^a$ with integral log structure $M_{\Spf (B)}$ and a map of log prisms $(A, I, M_A)^a\to (B, IB, M_{\Spf (B)})$, 
 \item a map of formal schemes $f\colon  \Spf (B/IB) \to X$ over $A/I$, and 
 \item an exact closed immersion of log formal schemes 
 \[
 (\Spf (B/IB), f^*M_X)\to (\Spf (B), M_{\Spf (B)})
 \]
 over $(A, M_A)$.
\end{itemize}    
Morphisms are obvious ones. 
We simply write $(B, IB, M_{\Spf(B)})$ for an object when there is no confusion. 

A morphism $(B, IB, M_{\Spf (B)})\to (C, IC, M_{\Spf (C)})$ in $((X,M_X)/ (A, M_A))_{\Prism}$ is an \'etale cover if $B\to C$ is $(p,I)$-completely \'etale and faithfully flat and the induced map of log formal schemes $(\Spf (C), M_{\Spf (C)})\to \Spf (B), M_{\Spf (B)})$  is \'etale and strict\footnote{This means the pullback induces an isomorphism of log structures.}. 
The log prismatic site is the category $((X,M_X)/ (A, M_A))_{\Prism}$ with the \'etale (pre)topology.  

The structure sheaf $\cO_{\Prism}$ is defined by $(B, IB, M_{\Spf (B)})\mapsto B$. Similarly, $\overline{\cO}_{\Prism}$ is defined by $(B, IB, M_{\Spf (B)})\mapsto B/IB$, and satisfies $\cO_{\Prism}\otimes^L_A A/I\cong \overline{\cO}_{\Prism}$. 
\end{defn}

\begin{rem}[The category of \'etale coverings]\label{slice category}
Let $(B, IB, M_{\Spf (B)})$ be an object of $((X,M_X)/ (A, M_A))_{\Prism}$, and $B/IB \to \overline{C}$ a $p$-completely \'etale map. Then, by Lemma \ref{\'etale maps} and Proposition \ref{sheafification}, it lifts uniquely to a $(p, I)$-completely \'etale map $(B, IB, M_{\Spf (B)})\to (C, IC, M_{\Spf (C)})$ in $((X,M_X)/ (A, M_A))_{\Prism}$ with $C/IC \cong \overline{C}$. 
Thus, the category of objects \'etale over $(B, IB, M_{\Spf (B)})$ is equivalent to the category of $p$-completely \'etale $B/IB$-algebras. This guarantees that the log prismatic site is indeed a site and the structure sheaves are sheaves on it. 
\end{rem}

\begin{rem}[Flat topology]\label{flat topology}
One can also use $(p, I)$-completely faithfully flat maps $B\to C$ as coverings. 
One can show that this indeed defines a site and the cohomology of the structure sheaf does not change. 
\end{rem}

\begin{rem}
We fixed a chart $M_A\to A$ of the log formal scheme $(\Spf (A), M_A)^a$. While the assumption that $(X, M_X)$ is smooth over $(A/I, M_A)$ in our sense possibly depends on the chart, the log prismatic site itself only depends on $(\Spf (A), M_A)^a$ and $(X, M_X)$. 

If $M_X$ is the trivial log structure, our log prismatic site recovers the prismatic site $(X/ A)_{\Prism}$ of Bhatt-Scholze with \'etale topology in place of flat topology \cite{BS}*{Definition 4.1}. 
\end{rem}

\begin{rem}[Relation to the \'etale site]\label{prismatic to etale}
With Remark \ref{slice category}, we use the same construction as in \cite{BS}*{Remark 4.4} to define a morphism of topoi:
\[
\nu \colon \textnormal{Shv}(((X,M_X)/ (A, M_A))_{\Prism}) \to
\textnormal{Shv}(X_{\et}), 
\]
and, for any \'etale map $U\to X$, we have a canonical isomorphism
\[
(\nu_* F)(U\to X)\cong H^0 (((U,M_U)/ (A, M_A))_{\Prism}, F|_{((U,M_U)/ (A, M_A))_{\Prism}}),
\]
where $M_U$ is the pullback of $M_X$. Write
\begin{align*}
\Prism_{(X, M_X)/(A, M_A)}&\coloneqq R\nu_* \cO_{\Prism}\in D(X_{\et}, A), \\
\overline{\Prism}_{(X, M_X)/(A, M_A)}&\coloneqq R\nu_* \overline{\cO}_{\Prism}\in D(X_{\et}, A/I).
\end{align*}
These are commutative algebra objects in the corresponding derived categories and satisfy
\[
\overline{\Prism}_{(X, M_X)/(A, M_A)}
\cong 
\Prism_{(X, M_X)/(A, M_A)}\otimes_{A}^L A/I. 
\]
\end{rem}

\begin{rem}[Absolute log prismatic site]\label{absotelu site}
Let $(X, M_X)$ be an integral log $p$-adic formal scheme. We define the absolute log prismatic site $(X, M_X)_{\Prism}$ as the category of diagrams
\[
(\Spf (B), M_{\Spf (B)}) \hookleftarrow (\Spf (B/J), M_{\Spf (B/J)}) \to (X, M_X), 
\]
where $(B, J, M_{\Spf (B)})$ is a log prism, $M_{\Spf (B)}$ is integral, and $M_{\Spf (B/J)}$ is the restriction of $M_{\Spf (B)}$ as a log structure. We also assume that the right map $(\Spf (B/J), M_{\Spf (B/J)}) \to (X, M_X)$ admits an integral chart \'etale locally. 
We equip it with the flat topology. We have a variant with ``integral'' replaced by ``saturated'' and we use this variant in \cite{Koshikawa-Yao}. 

We can also consider a variant which requires $(\Spf (B/J), M_{\Spf (B/J)}) \to (X, M_X)$ to be strict. 
The relative version defined in Definition \ref{log prismatic site} admits a natural forgetful functor to this variant.
\end{rem}

\subsection{Computing log prismatic cohomology}\label{computing log prismatic cohomology}
Suppose that $X$ is an affine formal scheme $\Spf (R)$ and there is an integral chart $P\to \Gamma (X, M_X)$ over $M_A$. 
In this situation, we also write $((R, P)/ (A, M_A))_{\Prism}$ for $((X, M_X)/ (A, M_A))_{\Prism}$ instead. 
A standard argument shows that the cohomology of the structure sheaf does not change by replacing the \'etale topology by the indiscrete topology.  
We assume that 
\begin{center}
$M_A\to P$ is integral and weakly finitely generated 
\end{center}
to have good control of log prismatic envelopes. 

Set
\[
\Prism_{(R, P)/(A, M_A)}\coloneqq R\Gamma (((R, P)/(A, M_A)_{\Prism}, \cO_{\Prism}). 
\]
This is a $(p,I)$-complete commutative algebra object in $D(A)$ equipped with a $\phi$-semilinear map $\Prism_{(R, P)/(A, M_A)}\to \Prism_{(R, P)/(A, M_A)}$ induced by that of $\cO_{\Prism}$ determined by $\delta$-structures. 
Similarly, we define 
\[
\overline{\Prism}_{(R, P)/(A, M_A)}\coloneqq R\Gamma (((R, P)/(A, M_A)_{\Prism}, \overline{\cO}_{\Prism}).
\]
This is a $p$-complete commutative algebra object in $D(A/I)$, and there is a natural isomorphism
\[
\Prism_{(R, P)/(A, M_A)}\otimes^L_{A} A/I \cong \overline{\Prism}_{(R, P)/(A, M_A)}.
\]
These objects only depend on $(X, M_X)=(\Spf(R), P)^a$ and $(A, I, M_A)^a$. 

\begin{const}[\v{C}ech-Alexander complexes for log prismatic cohomology]\label{\v{C}ech}
Take a surjection $M_{B}=M_A\oplus M_{B}^A\to P$ from a monoid $M_B$ free over $M_A$, i.e., $M_B^A$ is isomorphic to $\bN^T$ for some set $T$.  
Take a surjection from a $(p,I)$-completed polynomial ring
\[
B_0\coloneqq A\langle (X_{s})_{s\in S}, \bN^T\rangle\to R
\]
for some set $S$, compatible with the map $M_B\to P$. 

Let $(B, M_{B})$ be the $(p,I)$-completed $\delta_{\log}$-ring over $(A, M_A)$ generated by $B_0$:
\[
B=(A\{ (X_{s})_{s\in S}\}^{\delta}\{\bN^T\}_{\log}^{\delta})^{\wedge}_{(p,I)}.  
\]
Let $(B_0^{\bullet}, M_{B}^{\bullet})$ (resp. $(B^{\bullet}, M_{B}^{\bullet})$) be the $(p,I)$-completed \v{C}ech nerve of $(A, M_A)\to (B_0, M_B)$ (resp. $(A, M_A)\to (B, M_B)$), with a natural inclusion
\[
(B_0^{\bullet}, M_{B}^{\bullet}) \hookrightarrow (B^{\bullet}, M_{B}^{\bullet})
\]
identifying monoids. There is a surjection of prelog rings
\[
(B_0^{\bullet}, M_{B}^{\bullet})\to (R, P). 
\]
Denote by $J^{\bullet}\subset B_0^{\bullet}$ the kernel of $B_0^{\bullet}\to R$. 

Each $(B_0^{\bullet}, M_B^{\bullet})$ satisfies the condition $(*)$ of Proposition \ref{Prelog prismatice envelops for smooth log algebras}, thus we may apply the proposition to 
\[
(B^{\bullet}, (J^{\bullet}B^{\bullet})^{\wedge}_{(p,I)}, M_{B}^{\bullet})
\]
as all other assumptions are satisfied as well, and we obtain a cosimplicial prelog prism and the associated cosimplicial object in $((R, P)/(A, M_A))_{\Prism}$:
\[
(C^{\bullet}, IC^{\bullet}, M_C^{\bullet}), \quad (C^{\bullet}, IC^{\bullet}, M_C^{\bullet})^a
\]
with exact surjections $(C^{\bullet}, M_C^{\bullet})\to (C^{\bullet}/ I C^{\bullet}, P)$.  
Each $C^{\bullet}$ is $(p,I)$-completely flat over $A$, and this cosimplicial object in $((R, P)/(A, M_A))_{\Prism}$ agrees with the \v{C}ech nerve of $(C^{0}, IC^{0}, M_{C^{0}})^a$ by Proposition \ref{Existence of log prismatic envelopes}. 

Let $(B', IB', M_{\Spf (B')})$ be an object of $((R, P)/(A, M_A))_{\Prism}$. Note that 
\begin{equation}\label{exact surjectivity of modulo I}
(B', \Gamma (\Spf (B'), M_{\Spf (B')}))\to 
\Gamma (\Spf (B'/IB'), P^a_{\Spf (B'/IB')})
\end{equation}
is exact surjective, cf. Lemma \ref{exact closed immersion}. 
By the freeness of $(B_0, M_B)$ as a prelog ring, there exists at least one map of prelog rings
\begin{center}
$(B_0, M_B)\to (B', \Gamma (\Spf (B'), M_{\Spf (B')}))$ 
\end{center}
lifting the composite $(B_0, M_B)\to (R,P)\to (B'/IB', \Gamma (\Spf (B'/IB'), P^a_{\Spf (B'/IB')}))$, which maps $J$ to $IB'$. 
The $\delta$- and $\delta_{\log}$-structures on $(B', \Gamma (\Spf (B'/IB'), M_{\Spf (B'/IB')}))$ and the freeness of $(B, M_B)$ as a $\delta_{\log}$-ring enable us to extend it to a map of $\delta_{\log}$-triples
\[
(B, (JB)^{\wedge}_{(p,I)}, M_B) \to (B', IB', \Gamma (\Spf (B'), M_{\Spf (B')}))
\]
yielding, by the exact surjectivity of (\ref{exact surjectivity of modulo I}), a map of bounded prelog prisms
\[
(C^{0}, IC^{0}, M_{C^{0}}) \to (B', IB', \Gamma (\Spf (B'), M_{\Spf (B')})). 
\]
By Remark \ref{Maps from prelog prisms}, it further induces a map of log prisms
\[
(C^{0}, IC^{0}, M_{C^{0}})^a \to (B', IB', M_{\Spf (B')}). 
\]
Therefore, the object $(C^{0}, IC^{0}, M_{C^{0}})^a$ in $((R, P)/(A, M_A))_{\Prism}$ represents a covering of the final object of the topos $\Shv (((R, P)/(A, M_A))_{\Prism})$. Thus, $\Prism_{(R, P)/(A, M_A)}$ is computed by the cosimplicial $\delta$-$A$-algebra $C^{\bullet}$. (Use \cite{Stacks}*{\href{https://stacks.math.columbia.edu/tag/079Z}{Tag 079Z}} and the affine vanishing, or work with the indiscrete topology.) This construction commutes with any base change $(A, I, M_A)\to (A', IA', M_{A'})$. 
\end{const}

\begin{rem}[Strictly functorial complexes computing log prismatic cohomology]\label{functorial complex}
Continuing from the above construction, we further assume that
\begin{center}
    $P=\Gamma (X, M_X)$, i.e., we look at the identity $\Gamma (X, M_X)\to \Gamma (X, M_X)$.   
\end{center}
Recall that we also require that $M_A\to P=\Gamma (X, M_X)$ be integral and weakly finitely generated. For instance, this is the case for any ``small'' neighborhood of a geometric point $\overline{x}$ of $X$ under some mild assumption by Lemmas \ref{small neighborhood}, \ref{chart and global section II}, and smoothness of $(X, M_X)$. 

In Construction \ref{\v{C}ech}, we can take a surjection $(B_0, M_B)\to (R, P)$ strictly functorial on $(X, M_X)$ under these assumptions. Explicitly, we can use a unique map
\[
(B_0, M_B)=(A\langle \bN^R\oplus \bN^P\rangle, M_A\oplus\bN^P)\to (R, P)
\]
sending $1_r\mapsto r$, $1_p\mapsto p$ for $r\in R$, $p\in P$.  

The resulting complex $C^{\bullet}((R, P)/ (A, M_A), \cO_{\Prism})$ is a strictly functorial complex computing $\Prism_{(R, P)/ (A, M_A)}$. (This complex does not commute with base changes on $(A, M_A)$, but still functorial on base changes of $(A, M_A)$.) 
\end{rem}

\begin{rem}[The log-affine line]\label{log-affine line}
Let $(R, P)=(A/I\langle X\rangle, M_A\oplus \bN)$. In Construction \ref{\v{C}ech}, we can take 
\begin{center}
$(B_0, M_B)=(A\langle X\rangle, M_A\oplus \bN)$.
\end{center}
The resulting cosimplicial $\delta$-$A$-algebra $C^{\bullet}$ commutes with base change on the prelog prism $(A, I, M_A)$. 
\end{rem}

\begin{rem}[Monoid algebras]\label{monoid algebras}
In Construction \ref{\v{C}ech}, assume that $R=A/I \langle P\rangle$. Then, we can take
\begin{center}
$(B_0, M_B)=(A\langle \bN^P \rangle, M_A\oplus \bN^P)$.
\end{center}
The resulting cosimplicial $\delta$-$A$-algebra $C^{\bullet}$ commutes with base change on the prism $(A, I)$ without changing $M_A$. 
\end{rem}

The above construction has immediate consequences. 

\begin{lem}[weak Base change]\label{base change I}
Let $(R, P)$ be as above. Suppose $(A, I, M_A)\to (A', IA', M_{A'})$ is a map of bounded prelog prisms such that $M_{A'}$ is integral and $A\to A'$ has finite $(p,I)$-complete Tor amplitude. If we write $(R', P')$ for the $p$-completed base change of $(R, P)$ as a prelog ring, then the natural map induces an isomorphism
\[
\Prism_{(R, P)/(A, M_A)}\widehat{\otimes}^L_A A' \overset{\cong}{\longrightarrow}
\Prism_{(R', P')/(A', M_{A'})}. 
\]
Similar statement holds for $\overline{\Prism}_{(R, P)/(A, M_A)}$. 
\end{lem}

\begin{proof}
Note that $(R', P')$ satisfies the assumption in \S \ref{computing log prismatic cohomology}. Given Construction \ref{\v{C}ech}, we can argue as in \cite{BS}*{Lemma 4.20}. 
\end{proof}

\begin{lem}[Strict \'etale localization]\label{localization}
Let $R\to S$ be a $p$-completely \'etale map of $A/I$-algebras; in particular, $(S, P)$ is $p$-completely smooth over $(A, M_A)$. Then the natural map 
\[
\overline{\Prism}_{(R, P)/(A, M_A)}\widehat{\otimes}^L_R S \to \overline{\Prism}_{(S, P)/(A, M_A)}
\]
is an isomorphism. 
\end{lem}

\begin{proof}
Note that the pair $(S, P)$ satisfies the assumption in \S \ref{computing log prismatic cohomology}. 
We proceed as in the proof of \cite{BS}*{Lemma 4.21}. 
The restriction functor 
\[
((\Spf (S), P)^a /(A, M_A))_{\Prism}\to ((\Spf (R), P)^a /(A, M_A))_{\Prism}
\]
sending $(B, IB, M_{\Spf (B)})$ to itself regarded as an object of the target admits a right adjoint
\[
F\colon ((\Spf (R), P)^a /(A, M_A))_{\Prism}\to ((\Spf (S), P)^a /(A, M_A))_{\Prism}.
\]
At the level of $R$-algebras, $F$ is the $p$-completed base change along $R\to S$ and $F$ is its unique lifting determined by Remark \ref{slice category}. The log structure (with $\delta_{\log}$-structure) is simply the pullback to the base change. 
Thus, the proof of \cite{BS}*{Lemma 4.21} works in our setting. 
\end{proof}

\subsection{More coverings of the final object}
We now work with the \emph{flat} topology; see Remark \ref{flat topology}. 
Under a mild assumption that is always satisfied \'etale locally on $X$, we show that there are coverings of the final object in the topos that are simpler and smaller than the ones we considered above. 
So, we can also use these coverings to compute log prismatic cohomology. 
In the non-log setting, similar results were considered in several papers including \cite{Mao}. 

Let $(R, P)$ be $p$-completely smooth over $(A/I, M_A)$, and assume that $M_A \to P$ is a smooth chart. 
Let $(B, IB, M_B)$ be a prelog prism over $(A, I, M_A)$ that is $(p,I)$-completely smooth over $(A, M_A)$ and assume $M_A\to M_B$ is smooth chart. 
Let $(B, M_B) \to (R, P)$ a surjection and let $(B', IB', M_{B'})$ denote the prelog prismatic envelope of $(B, IB, M_B)$, which is $(p, I)$-completely flat over $A$ by Proposition \ref{Prelog prismatice envelops for smooth log algebras:variant}. 

\begin{prop}\label{more coverings}
Let $(C, IC, M_{C})^a$ be an object of $((R, P)/ (A, M_A))_{\Prism}$. 
The product of $(B', IB', M_{B'})^a$ and $(C, IC, M_{C})^a$ in $((R, P)/ (A, M_A))_{\Prism}$ exists and it is $(p,I)$-completely faithfully flat over $C$.
In particular, the sheaf represented by $(B', IB', M_{B'})^a$ is a covering of the final object of the topos. 
\end{prop}

For instance, if $(A, M_A)$ is of rank 1, then a natural smooth lift $(\widetilde{R}, P)$ over $(A, M_A)$ of $(R, P)$ admits the structure of a $\delta_{\log}$-ring of rank $1$, and it gives rise to a covering. 

\begin{proof}
As in the proof of Proposition \ref{Existence of log prismatic envelopes}, the map
\[
P \times_{\Gamma (\Spf (C/IC), M_{\Spf (C/IC)})} \Gamma (\Spf (C), M_{\Spf (C)}) \to C
\]
is a chart of $M_{\Spf (C)}$ with $\delta_{\log}$-structure. We may assume this is the chart $M_C \to C$. By Lemma \ref{exact closed immersion} and base change, $M_C\to P$ is exact surjective.
Using the composite $(B, M_B)\to (R, P) \to (C/IC, P)$, we obtain a surjection
\[
(B\widehat{\otimes}_A C, M_B \times_{M_A} M_C) \to (C/IC, P),  
\]
where the source is a $(p, I)$-completely smooth $\delta_{\log}$-ring over $(C, M_C)$. 
Let $(D, M_{C'})$ denote its exactification, which is $(p,I)$-completely smooth over $(C, M_C)$. 
The prelog prismatic envelope $(C', IC', M_{C'})$ of the displayed map is $(p,I)$-completely flat over $C$ by Proposition \ref{Prelog prismatice envelops for smooth log algebras:variant}. In fact, it is $(p,I)$-completely faithfully\footnote{This can be seen somewhat more directly using an elementary formula in \cite{Mao}*{5.34} (see also \cite{Du-Liu}*{2.2.4}), together with the proof of \cite{BS}*{Proposition 3.13}.} flat over $C$: we may assume $(A, I)$ is orientable, and we can apply \cite{Mao}*{5.51} since the surjection $D/ID\to C/IC$ is $p$-completely quasiregular. 
By its construction, there is a also map $(B', IB', M_{B'})^a \to (C', IC', M_{C'})^a$ and one sees that $(C', IC', M_{C'})^a$ is the product of $(B', IB', M_{B'})^a$ and $(C, IC, M_C)^a$. 
\end{proof}

Let us also record the following lemma in the absolute case (see also \cite{Du-Liu}):

\begin{lem}
The Breuil-Kisin log prism $(W(k)[\![ u]\!], (E(u)), u^{\bN})^a$ represents a covering of the final object of the category of sheaves on $(\cO_K, \cO_K\setminus\{0\})_{\Prism}$.   
\end{lem}

In the non-log case, see \cite{Bhatt-Scholze:F-crystal}*{Example 2.6(1)}. 

\begin{proof}
Let $(A, I, M_{\Spf (A)})$ be an object of $(\cO_K, \cO_K\setminus\{0\})_{\Prism}$. 
Note that $A$ is canonically a $W(k)$-algebra. 
As before, take the fiber product 
\[
N\coloneqq u^\bN \times_{\Gamma (\Spf (A/I), M_{\Spf (A/I)})} \Gamma (\Spf (A), M_{\Spf (A)}), 
\]
which comes with an exact surjection $N\to \bN$. 
Take a lift $n\in N$ of the generator of $\bN$. 
(Note that $A$ is classically complete with respect the ideal of $A$ generated by the image of $n$.)
We consider $(A[\![ u ]\!], n^{\bN}u^{\bN})$ equipped with a surjection
\[
(A[\![ u ]\!], n^{\bN}u^{\bN}) \to
(A/I, n^{\bN}) 
\]
sending $u^{\bN}$ to $n^{\bN}$. 
The exactification is 
\[
(A [\![ u]\!] [ (u/n)^{\pm 1}], u^{\bN}n^{\bN}(u/n)^{\bZ}). 
\]
The map $A\to A [\![ u]\!] [ (u/n)^{\pm 1}]$ is $(p,I)$-completely flat, and the kernel of the extended map $A [\![ u ]\!] [ (u/n)^{\pm 1}]\to A; u/n\mapsto 1$ is generated by a $(p,I)$-completely regular sequence relative to $A$. 
So, its prismatic envelope $A'$ is $(p,I)$-completely flat over $A$. 
It is also $(p,I)$-completely \emph{faithfully} flat by \cite{Mao}*{5.51} (See also \cite{Du-Liu}). 
Moreover, there is clearly a map from $(W(k) [\![ u ]\!], (E(u)), u^{\bN})$ to $(A', IA', u^{\bN}n^{\bN}(u/n)^{\bZ})$ that is compatible with $(W(k)[\![ u]\!], u^{\bN})\to (\cO_K, \bN) \to (A/I, n^\bN)$. 
\end{proof}

\subsection{A remark on globalization}
Any site-theoretic construction will be tautologically global as the log prismatic site is globally defined.  
However, in some situation, we would have to work purely locally and use strictly functorial complexes like the one in Remark \ref{functorial complex}, and glue them to globalize; cf. Section \ref{section:comparison with AOmega}. To do so, we need a basis of the \'etale site of $X$ for which charts satisfying the assumptions in Remark \ref{functorial complex} exist.  

More concretely, we want a basis $(U_i\to X)_{i\in I}$ such that $(U_i, (M_X)|_{U_i})$ is log-affine, and $M_A\to \Gamma (U_i, M_X)$ is integral and weakly finitely generated. 
Lemmas \ref{small neighborhood}, \ref{chart and global section II}, and smoothness of $(X, M_X)$ imply that this is the case if 
\begin{itemize}
    \item the underlying topological space of $X$ is locally noetherian, and
\end{itemize}
one of the following holds:
\begin{itemize}
    \item the monoid $M_A$ has finitely many prime ideals \cite{Ogus}*{I.1.4.1}, or
    \item the image of $M_A\setminus M_A^{\times}\to A$ is contained in $\sqrt{I}$. 
\end{itemize}

\section{The Hodge-Tate comparison}
After recalling Gabber's cotangent complexes for maps of prelog ringed structures on topoi, we formulate and prove the Hodge-Tate comparison. 

\subsection{Gabber's logarithmic cotangent complex}
There are two notions of logarithmic cotangent complexes in the literature; one is defined by Olsson, and the other one is due to Gabber \cite{Olsson}. We shall use (the completed version of) Gabber's cotangent complex as it is defined without any finiteness conditions.

We mainly work locally. 
Let $f\colon (A, M_A)\to (B, M_B)$ be a map of prelog rings and assume that $A$ and $B$ are classically $I$-complete for some finitely generated ideal $I\subset A$. Using the canonical free resolution of prelog ringed structure on a topos (with enough points), Gabber(-Olsson) defined the cotangent complex \cite{Olsson}*{8.5}. We apply it in the \'etale topos of the $I$-adic formal scheme $\Spf (B)$ to
\[
(f^{-1}\cO_{\Spf (A)}, f^{-1} \underline{M_{A}})\to (\cO_{\Spf (B)}, \underline{M_B}), 
\]
where $\cO_{\Spf (A)}, \cO_{\Spf (B)}$ are structure sheaves on corresponding formal schemes and $\underline{M_A}, \underline{M_B}$ are constant \'etale sheaves associated with $M_A, M_B$. (Compare with \cite{Olsson}*{8.29}.)
Its $I$-completion is denoted by $\underline{L}_{(B, M_B)/ (A, M_A)}$, often regarded as an $I$-complete object of the derived category of $\cO_{\Spf (B)}$-modules. 

We list some properties:
\begin{enumerate}
    \item If $M_A\to M_B$ is an isomorphism, $\underline{L}_{(B, M_B)/ (A, M_A)}$ is canonically isomorphic to the $I$-completed (nonlog) cotangent complex $L_{B/A}$ \cite{Olsson}*{8.22}.
    \item For maps $(A, M_A)\to (B, M_B)\to (C, M_C)$ of classically $I$-complete prelog rings, there is a functorial distinguished triangle
    \[
    (\underline{L}_{(B,M_B)/(A,M_A)}\otimes^L_{\cO_{\Spf (B)}} \cO_{\Spf (C)})^{\wedge} \to \underline{L}_{(C, M_C)/(A,M_A)}\to \underline{L}_{(C, M_C)/(B,M_B)}\overset{+1}{\longrightarrow}
    \]
    \cite{Olsson}*{8.18}, where the completion is $I$-adic. 
    \item A natural map
\[
\underline{L}_{(B, M_B)/ (A, M_A)}\to \underline{L}_{(B, M_B^a)/ (A, M_A^a)}
\]
induces an isomorphism \cite{Olsson}*{8.20}. 
\end{enumerate}

The properties (1) and (2) imply that
\[
\underline{L}_{(C, M_B)/(A, M_A)}\cong (\underline{L}_{(B, M_B)/(A, M_A)}\otimes^L_{\cO_{\Spf (B)}} \cO_{\Spf (C)})^{\wedge}
\]
for an $I$-completely \'etale map $B\to C$ as $L_{C/B}=0$. 

Under our definition of smooth log formal schemes, the cotangent complex may be identified with the sheaf of log-differentials in the smooth case as shown below.
(For a smooth morphism of log schemes in the sense of Kato, the cotangent complex may not be discrete; a key additional assumption in our convention is the integrality.)

Let us write $\underline{\Omega}^1_{(B,M_B)/(A,M_A)}$ for the ($I$-completed) sheaf of log-differentials of associated log formal schemes, and $\Omega^1_{(B, M_B)/(A, M_A)}$ for its global sections, which only depends on the associated log formal schemes. 

\begin{prop}
Suppose that $A$ has bounded $p^\infty$-torsion and $(B,M_B)$ is $I$-completely smooth over $(A, M_A)$ with a smooth chart $P\to M_B$ over $M_A$. Then, $\underline{L}_{(B, M_B)/(A, M_A)}$ is discrete and there is a canonical isomorphism
\[
\underline{H}^0(\underline{L}_{(B, M_B)/(A, M_A)})\cong \underline{\Omega}^1_{(B,M_B)/(A,M_A)}. 
\]
\end{prop}

\begin{proof}
By \'etale localization, it reduces to the case where $B$ is the $I$-completion of $A\otimes_{\bZ[M_A]}\bZ[P]$. Let $\cO_{A'}=\cO_{\Spf (A)}\otimes_{\bZ[\underline{M_A}]}\bZ[\underline{P}]$ denote the corresponding \'etale sheaf on $\Spf (A)$. 
In this case, we prove an analogous claim for the cotangent complex $\underline{L}_{(\cO_{A'}, \underline{P})/(\cO_{\Spf (A)}, \underline{M_A})}$ itself (without $I$-completion) on the \'etale topos of $\Spf(A)$ defined in \cite{Olsson}; this in turn implies the result for our $I$-completed cotangent complex of $\cO_{\Spf (B)}$-modules as $\underline{\Omega}^1_{(A',P)/(A,M_A)}$ (defined in an obvious way) is free of finite rank over $\cO_{A'}$, and ($I$-completed) $\underline{L}_{B/(A\otimes_{\bZ[M_A]}\bZ[P])}$ is $0$. 

In the rest of the  proof, every cotangent complex lives in the \'etale topos of $\Spf (A)$, and we do not take $I$-completions. 
At least in the case of fine log schemes, we are done by \cite{Olsson}*{8.34} through the comparison with Olsson's cotangent complex. 
Actually, the argument in {\it loc.cit.} can be modified so that Olsson's cotangent complex does not appear and the finiteness condition on $M_A$ is not used. Indeed, the argument in \cite{Olsson}*{8.32} shows that $\underline{L}_{(\cO_{A'}, \underline{P})/(\cO_{\Spf (A)}, \underline{M_A})}$ is isomorphic to the homotopy pushout of the diagram
\[
\begin{CD}
\cO_{A'}\otimes_{\bZ[\underline{P}]}L_{\bZ[\underline{P}]/\bZ[\underline{M_A}]} @>>> \cO_{A'}\otimes_{\bZ}(\underline{P}^{\gp}/\underline{M_{A}}^{\gp})\\
@VVV \\
L_{\cO_{A'}/\cO_{\Spf (A)}}. 
\end{CD}
\]
As $\cO_{A'}=\cO_{\Spf (A)}\otimes_{\bZ[\underline{M_A}]}\bZ[\underline{P}]$ and $\bZ[\underline{P}]$ is flat over $\bZ[\underline{M_A}]$, the vertical arrow is an isomorphism by \cite{Illusie:cotangent complex I}*{II.2.2.1}.  
\end{proof}

\subsection{Formulation of the Hodge-Tate comparison}
Fix a bounded prelog prism $(A, I, M_A)$ with integral $M_A$. Let $(X, M_X)$ be a smooth log $p$-adic formal scheme over $(A/I, M_A)$. 
Recall that
\begin{gather*}
\Prism_{(X, M_X)/ (A/I, M_A)} = R\nu_* \cO_{\Prism}, \quad \overline{\Prism}_{(X, M_X)/ (A/I, M_A)}= R\nu_* \overline{\cO}_{\Prism}, \\
\overline{\Prism}_{(X, M_X)/ (A/I, M_A)}\cong \Prism_{(X, M_X)/ (A/I, M_A)}\otimes^L_A A/I.
\end{gather*}
Consider the Bockstein differential:
\begin{multline*}
\beta_I\colon H^i (\overline{\Prism}_{(X,M_X)/(A, M_A)})\{i\}\coloneqq H^i (\overline{\Prism}_{(X,M_X)/(A, M_A)}) \otimes_A^L I^i / I^{i+1}  \\
\to 
H^{i+1} (\overline{\Prism}_{(X,M_X)/(A, M_A)})\{i+1\}.
\end{multline*}
(Twists $\{ *\}$ may disappear if $(A, I)$ is orientable, i.e., $I$ is principal, and we fix its orientation(=generator).)

We have the structure map 
\[
\eta_{(X,M_X)}^0\colon \cO_X \to H^0 (\overline{\Prism}_{(X,M_X)/(A, M_A)}), 
\]
and this extends to 
\[
\eta_{(X,M_X)}^1\colon \Omega^1_{X/A} \to H^1 (\overline{\Prism}_{(X,M_X)/(A, M_A)})\{1\}
\]
by the universal property of the K\"ahler differentials. 
We want to extend this map to log-differentials. 

We shall use another description of $\eta_{(X,M_X)}^1$ in terms of cotangent complexes following \cite{BS}*{Proposition 4.15}. 
Assume that $X=\Spf (R)$ and $M_X$ has a chart $P\to \Gamma (X, M_X)\to R$. 
For any object $(B, IB, M_B)^a$ of the log prismatic site $((R, P)/(A, M_A))_{\Prism}$ with an exact surjection $M_B \to P$ as in the proof of Proposition \ref{more coverings}, there are natural maps
\begin{align*}
R\Gamma (\Spf (R)_{\et}, \underline{L}_{(R, P)/(A, M_A)}) &\longrightarrow R\Gamma (\Spf (B/IB)_{\et}, \underline{L}_{(B/IB, P)/(B, M_B)}) \\
&\overset{\cong}{\longleftarrow}R\Gamma(\Spf (B/IB)_{\et}, \underline{L}_{(B/IB, M_B)/(B, M_B)}) \\
&\overset{\cong}{\longrightarrow}{IB/I^2 B}[1] \\
&\overset{\cong}{\longleftarrow}  I/I^2\otimes^L_{A/I} B/IB[1];
\end{align*}
the second map is an isomorphism as $P$ and $M_B$ have the same associated log structure on $\Spf (B/IB)$, and the third is an isomorphism by \cite{Illusie:cotangent complex I}*{III.3.2.4}. (These maps can be defined more naturally using the cotangent complexes of log formal schemes.) 
Taking derived global sections on $((R, P)/(A, M_A))_{\Prism}$, we get a map
\[
R\Gamma (\Spf (R)_{\et}, \underline{L}_{(R, P)/(A, M_A)}) \to \overline{\Prism}_{(R,P)/(A, M_A)}\{1\}[1]. 
\]
Passing to $H^0$, we obtain
\[
\eta_{(R,P)}^1\colon \Omega^1_{(R, P)/ (A/I, M_A)}\to H^1(\overline{\Prism}_{(R,P)/(A, M_A)})\{1\}. 
\]
This recovers the previous $\eta_{(X, M_X)}^1$ after restricted to $\Omega^1_{R/P}$: it reduces by functoriality to the case of the affine line with the trivial log structure, and this case is checked in  \cite{AB}*{Proposition 3.2.1}. 
The above map globalizes to 
\[
\eta_{(X,M_X)}^1\colon \Omega^1_{(X, M_X)/ (A/I, M_A)}\to H^1(\overline{\Prism}_{(X,M_X)/(A, M_A)})\{1\}. 
\]

We want to extend the above map to higher degrees. 
We know that the differential graded $A/I$-algebra $H^* (\overline{\Prism}_{(R,P)/(A, M_A)})\{*\}$ is graded commutative. It follows from \cite{BS}*{Proposition 6.2} and functoriality that \cite{BS}*{Lemma 4.10} continues to holds in our setting. Namely, for any local section $f\in \cO_X(U)$, we have $\beta_I (f)^2=0$. 
We need an analog for log-differentials:

\begin{lem}\label{lemma for log-differentials}
For any local section $m\in M^{\gp}_X(U)$, the following holds:
\[
\eta_{(X,M_X)}^1 (d\log (m))^2=0 \quad \textnormal{and} \quad \beta_I (\eta_{(X,M_X)}^1 (d\log (m)))=0. 
\]
\end{lem}

We prove it later, and assume Lemma \ref{lemma for log-differentials} holds true for the moment. Since $\Omega^1_{(X, M_X)/(A, M_A)}$ has locally a basis that consists of elements of the form of $d\log (m) \, (m \in M^{\gp}_X(U))$, we obtain, from the universality of the log de Rham complex, a map of commutative differential graded algebras
\[
\eta_{(X,M_X)}^*\colon \Omega^*_{(X, M_X)/ (A/I, M_A)}\to H^*(\overline{\Prism}_{(X,M_X)/(A, M_A)})\{*\}
\]
that compatible with $\cO_X$-module structure on the terms. 

\begin{thm}[The Hodge-Tate comparison]\label{Hodge-Tate}
The map
\[
\eta_{(X,M_X)}^*\colon \Omega^*_{(X, M_X)/ (A/I, M_A)}\to H^*(\overline{\Prism}_{(X,M_X)/(A, M_A)})\{*\}
\]
constructed above is an isomorphism of differential graded $A/I$-algebras. 
In particular, $\overline{\Prism}_{(X, M_X)/(A, M_A)}$ is a perfect complex. 
\end{thm}

\begin{rem}
By the construction of $\eta_{(R, P)}$, the Hodge-Tate comparison implies that there is an isomorphism
\[
R\Gamma (\Spf (R)_{\et}, \underline{L}_{(R,P)/(A, M_A)})\cong (\tau^{\leq 1}\overline{\Prism}_{(R,P)/(A,M_A)})\{1\}[1]. 
\]
\end{rem}

\begin{cor}[Base change]\label{base change II}
Let $(A, I, M_A)\to (A', IA', M_{A'})$ be a map of bounded prelog prisms with integral monoids. The formation of $\Prism_{(X, M_X)/(A,M_A)}$ commutes with the completed base change along this map. 
\end{cor}

\begin{proof}
This reduces to the corresponding statement for $\overline{\Prism}_{(X, M_X)/(A,M_A)}$, which follows from the Hodge-Tate comparison as $\Omega^*_{(X, M_X)/(A/I, M_A)}$ commutes with base change.  
\end{proof}

\begin{rem}[The de Rham comparison]
The de Rham comparison is an isomorphism of the form of
\begin{align*}
R\Gamma_{\Prism}((X, M_X)/(A, M_A))\widehat{\otimes}^L_{A, \phi} \phi_* A/I 
&\cong\phi_* R\Gamma_{\textnormal{dR}}((X, M_X)/(A/I, M_A)) \\
&\cong \phi_* R\Gamma_{\crys}((X, M_X)_{A/(p,I)}/(A/I, M_A)), 
\end{align*}
where the completion is $p$-adic. 
The de Rham comparison can be deduced in some cases by rewriting the log prismatic cohomology using base change along some map to a crystalline prism and applying the crystalline comparison discussed in the next section. 
However, the characterization of the de Rham comparison isomorphism seems to be best explained in terms of Nygaard filtrations, and we will only mention the de Rham comparison briefly in this paper. 
\end{rem}

\subsection{A reformulation of the Hodge-Tate comparison for smooth rings}
We rewrite the Hodge-Tate comparison for smooth rings without log structure. 
Let $(A, I)$ be a bounded prism. 
Let $G^0 \to G$ be a surjection of finitely generated abelian groups whose $p$-torsion parts are trivial, with kernel $H$. 
It induces a surjection
\[
A \langle G^0  \rangle \to R\coloneqq A/I \langle G \rangle. 
\]
Note that $R$ (resp. $A \langle G^0  \rangle$) is $p$-completely (resp. $(p,I)$-completely) smooth over $A/I$ (resp. $A$). 
Its prismatic envelope $B^0$ gives rise to a covering of the final object of the topos of $(R/A)_{\Prism}$ by Proposition \ref{more coverings}. 
Let $B^{\bullet}$ denote the \v{C}ech nerve of $B^0$ in $(R/A)_{\Prism}$. 
Concretely, again by Proposition \ref{more coverings}, $B^{\bullet}$ is the prismatic envelope of a surjection
\[
A\langle (G^0)^{\oplus(\bullet +1)}\rangle \to A/I \langle G\rangle. 
\]
induced by the multiplication map $(G^0)^{\oplus (\bullet +1)} \to G^0$ and $G^0 \to G$. 
Let $H_n$ denote the kernel of $(G^0)^{\oplus (n+1)} \to G^0$. Regarding $B^n$ as a $B^0$-algebra via the face map $\delta_0^n$, we rewrite $B^n$ as the $(p,I)$-adic completion of
\[
B^0 \otimes_{A} A \langle H_n  \rangle \{ (I, (h-1)_{h \in H_n})/I\}^\delta. 
\]
So, the reduction $B^n \otimes_A A/I$ is the $p$-completed tensor product
\begin{multline*}
A\langle G^0 \rangle \widehat{\otimes}_{A\langle H \rangle} A \langle H\oplus H_n  \rangle  \{ (I, (h-1)_{h \in H\oplus H_n})/I\}^\delta \otimes_A A/I \\
\cong 
R\widehat{\otimes}_{A/I} (A \langle H\oplus H_n  \rangle  \{ (I, (h-1)_{h \in H\oplus H_n})/I\}^\delta \otimes_A A/I), 
\end{multline*}
where we use that $h=1$ in $A \langle H\oplus H_n  \rangle  \{ (I, (h-1)_{h \in H\oplus H_n})/I\}^\delta \otimes_A A/I$. 
We observe that $B^{\bullet} \otimes_A A/I$ is a cosimplicial $R$-algebra, and, in fact, is the $p$-completed base change of a cosimplicial $A/I$-algebra
\[
C^{\bullet}\coloneqq A \langle H\oplus H_n  \rangle  \{ (I, (h-1)_{h \in H\oplus H_n})/I\}^\delta \otimes_A A/I. 
\]
For instance, for $G^0=G=X_0^\bZ$, $R=A\langle X_0^{\pm 1}\rangle$, we have
\[
B^n= A\langle X_0^{\pm 1}, \dots X_n^{\pm 1} \rangle
\]
and $X_1/X_0-1=\dots=X_n/X_0 -1=0$ in $C^n$.

Note that the Hodge-Tate comparison isomorphism induces a map
\begin{align*}
A/I \otimes_{\bZ} G \xhookrightarrow{d\log} \Omega^1_{R/ (A/I)}  
&\xrightarrow{\cong} H^1 (\overline{\Prism}_{R/A}) \{1\} \\
&\cong H^1 (\textnormal{Tot} (B^{\bullet}\otimes_A A/I)) \{1\}, 
\end{align*}
and it factors over 
\[
H^1 (\textnormal{Tot} (C^{\bullet})) \{1\} \to H^1 (\textnormal{Tot} (B^{\bullet}\otimes_A A/I)) \{1\}.
\]
Indeed, by the Bockstein construction, we see that this map is, Zariski locally on $A$, induced by sending $1\otimes g, g\in G$ to $(\widetilde{g}-1)/d \otimes d$ for a lift $\widetilde{g} \in G^0$ of $g$ and an orientation $d\in I$. 
Also, this induces
\[
\bigwedge^i A/I \otimes_{\bZ} G \to H^i (\textnormal{Tot} (C^{\bullet})) \{i\} \hookrightarrow H^i (\textnormal{Tot} (B^{\bullet}\otimes_A A/I)) \{i\}. 
\]

Here is our reformulation of the Hodge-Tate comparison theorem in the non-log case:

\begin{lem}\label{baby Hodge-Tate}
Notation as above. The induced map
\[
\bigwedge^i A/I \otimes_{\bZ} G \to H^i (\textnormal{Tot} (C^{\bullet})) \{i\}
\]
is an isomorphism of $A/I$-modules for all $i$. 
\end{lem}

\begin{proof}
As in the proof of \cite{BS}*{Theorem 6.3}, we choose a map
\[
\eta\colon \bigoplus_{i=0}^{\infty} \bigwedge^i A/I \otimes_{\bZ} G \{-i\}[-i] \to \textnormal{Tot} (C^{\bullet})
\]
compatible with maps in the statement. 
(Note that $A/I \otimes_{\bZ} G$ is a free $A/I$-module of finite rank.)
It suffices to show $\eta$ is an isomorphism in the derived category $D(A/I)$. 
This can be checked after the base change $A/I \to R$ as this map is $p$-completely faithfully flat.
After the base change, the map induces 
\[
\bigoplus_{i=0}^{\infty} \Omega^i_{R/ (A/I)} \{-i\}[-i] \to \textnormal{Tot} ( B^{\bullet}\otimes_A A/I )
\]
as totalizations commute with base change $A/I\to R$ by the $p$-complete flatness. 
As this maps is compatible with the Hodge-Tate comparison map on $H^{\bullet}$, it is an isomorphism by the Hodge-Tate comparison theorem.
\end{proof}

\subsection{The case of log affine line}
Let $(A, I, M_A)$ be a bounded prelog prism and assume $M_A$ is integral. 
Consider the (pre)log affine line
\[
(R, P)=(A/I\langle \bN \rangle, M_A \oplus \bN)=(A/I \langle X_0 \rangle, M_A \oplus X_0^\bN).
\]
Its lift $(B^0, M_{B^{0}})\coloneqq (A\langle \bN \rangle, M_A \oplus \bN)$ give rise to an object $(B^0, IB^0, M_{B^0})^a$ of $((R, P)/(A, M_A))$ by setting $\delta_{\log}(\bN)=0$. 
By Proposition \ref{more coverings}, the \v{C}ech nerve $(B^{\bullet}, IB^{\bullet}, M_{B^{\bullet}})^a$ of $(B^0 IB^0, M_{B^0})^a$ computes the log prismatic cohomology $\Prism_{(R, P)/(A, M_A)}$. 
Concretely, $B^n$ is the $(p, I)$-adic completion of
\[
A\langle X_0, \dots , X_n \rangle \{ (I, X_1 /X_0 -1, \dots, X_n /X_0 -1 )/ I\}^{\delta}. 
\]
The reduction $B^{\bullet}\otimes_A A/I$ is a cosimplicial $R$-algebra, and it has the form of
\[
B^{\bullet}\otimes_A A/I \cong R\widehat{\otimes}_{A} C^{\bullet}, 
\]
where the completed tensor product is $p$-adic and $C^{\bullet}$ is the one in Lemma \ref{baby Hodge-Tate} for $G^0=G=X_0^\bZ$. 
Thus, by Lemma \ref{baby Hodge-Tate}, we obtain an isomorphism 
\[
R\otimes_{\bZ} X_0^{\bZ} 
\cong H^1 (\textnormal{Tot}(R\widehat{\otimes}_A C^{\bullet}) )\{1\} \cong 
H^1 (\overline{\Prism}_{(R, P)/(A, M_A)})\{1\}. 
\]
After inverting $X_0\in R$, this identification is compatible with the Hodge-Tate comparison map in the non-log case, and $1\otimes X_0$ corresponds to $d X_0 / X_0$. So, $X_0 \otimes X_0$ corresponds to $d X_0$ before inverting $X_0$. As $X_0$ is a nonzerodivisor of $R$, this implies that the above identification is compatible with 
\[
\eta^1_{(R, P)}\colon \Omega^1_{(R, P)/ (A/I, M_A)} \to H^1 (\overline{\Prism}_{(R, P)/(A, M_A)})\{1\}
\]
under the map $d\log$. 
Therefore, Lemma \ref{baby Hodge-Tate} implies that

\begin{prop}[The Hodge-Tate comparison for the log-affine line]\label{Hodge-Tate:log-affine line}
The Hodge-Tate comparison holds for $(A/I \langle \bN \rangle, M_A \oplus \bN)$.  
\end{prop}

\subsection{Proof of the Hodge-Tate comparison}
First note that Proposition \ref{Hodge-Tate:log-affine line} implies that Lemma \ref{lemma for log-differentials} holds true for any $(R, P)$ by the functoriality of the Hodge-Tate comparison maps. 

Fix a surjection $M_A \oplus \bN^S \to P$ for some finite set $S$. 
We have a $\delta_{\log}$-ring $(A\langle \bN^S \rangle, M_A \oplus \bN^S)$ over $(A, M_A)$ by letting $\delta_{\log}(\bN^S)=0$, equipped with a surjection
\[
(A\langle \bN^S \rangle , M_A \oplus \bN^S) \to (R, P). 
\]
The \v{C}ech nerve $(B^{\bullet}, IB^{\bullet}, M_{\bullet})^a$ of its log prismatic envelope computes the log prismatic cohomology $\Prism_{(R, P)/ (A, M_A)}$ by Proposition \ref{more coverings}. 
(Note that $B^n$ is $(p,I)$-completely flat over $A$ for every $n$.)
Set 
\[
G\coloneqq P^{\gp}/M_A^{\gp}, \quad G^0\coloneqq \bZ^S, \quad H \coloneqq \textnormal{Ker} (G^0\to G). 
\]
The reduction $B^{\bullet}\otimes_A A/I$ is a cosimplicial $R$-algebra, and it has the form of
\[
B^{\bullet}\otimes_A A/I \cong   R  \widehat{\otimes}_A C^{\bullet}
\]
for $C^{\bullet}$ in Lemma \ref{baby Hodge-Tate}. Therefore, by the same lemma, we have identifications
\[
\Omega^i_{(R, P)/(A/I, M_A)} \cong
\bigwedge^i R\otimes_{\bZ} G \cong H^i ( \textnormal{Tot} (B^{\bullet}\otimes_A A/I) )\{i\}
\cong H^i (\overline{\Prism}_{(R, P)/(A, M_A)})\{i\},
\]
as $R$ is $p$-completely flat over $A/I$. 
So, it remains to check that this identification is compatible with the Hodge-Tate comparison map for $i=1$. 
For this, we fix an element $m\in P$ and show that images of $d\log (m)$ are the same. 
The element $d\log (m)$ is the image of $d\log (X_0)$ under the map induced by
\[
(A\langle X_0 \rangle, M_A \oplus X_0^{\bN}) \to (R, P)
\]
sending $X_0$ to $m$. 
Consider another surjection
\[
M_A \oplus \bN^S \oplus X_0^{\bN} \to P,  
\]
which lives over the identity of $M_A \oplus X_0^{\bN}$. 
For this identity, the compatibility is already checked in the discussion of Proposition \ref{Hodge-Tate:log-affine line}. 
Note that the construction is functorial with respect to surjections $M_A \oplus \bN^S \to P$ and the Hodge-Tate comparison is functorial. So, we are allowed to replace the surjection $M_A\oplus \bN^S \to P$ to compare images of $d\log (m)$, and the desired compatibility follows from the case of the identity of $M_A \oplus X_0^{\bN}$. 

\section{The crystalline comparison}\label{Section:crystalline comparison}
In the case of crystalline prisms, i.e., $I=(p)$, we relate the log prismatic cohomology to the log crystalline cohomology. 

Let $(A, (p), M_A)$ be a bounded prelog prism with integral $M_A$, and assume it is of rank $1$ or $(A, M_A)$ is a log ring. In particular, $A$ is $p$-torsionfree. Let $I$ be a PD ideal of $A$ that contains $p$, i.e., $I$ admits divided powers. 
The Frobenius map $(A/ p, M_A)\to (A/p, M_A)$, where Frobenius acts on $M_A$ via $\phi_{M_A}$ from Remark \ref{Frobenius lifts on monoids}, factors as 
\[
(A/p, M_A) \to (A/I, M_A) \overset{\psi}{\longrightarrow} (\phi_* A/p, \phi_*M_A)
\]
with the first map being a natural surjection since $I$ maps to $(p)$ under the Frobenius map.   
(Here $\phi_*M_A$ denotes $M_A$ regarded as a monoid over $M_A$ via $\phi_{M_A}$.)

Let $(R, P)$ be a smooth prelog ring over $(A/I, M_A)$ with $P$ integral, and assume that 
$M_A\to P$ is integral and (weakly) finitely generated. Let $(R^{(1)}, P^{(1)})$ be the base change of $(R, P)$ along $\psi$; the monoid $P^{(1)}$ is given by the pushout
\[
\begin{CD}
P @<<< M_A \\
@VVV @VV m \mapsto \phi_{M_A}(m)V \\
P^{(1)} @<<< \phi_*M_A. 
\end{CD}
\]
Note that $P^{(1)}$ is integral as $M_A \to P$ is integral, $(R^{(1)}, P^{(1)})$ is smooth over $(\phi_* A/p, \phi_* M_A)$, and $\phi_* M_A\to P^{(1)}$ is (weakly) finitely generated.  

We prove the crystalline comparison under an additional assumption on charts. 
A map of monoids $h\colon M\to N$ is \emph{exact} if the following diagram 
\[
\begin{CD}
M@>h >> N \\
@VVV @VVV \\
M^{\gp}@> h^{\gp}>> N^{\gp}
\end{CD}
\]
is Cartesian. If $M, N$ are integral and $h$ is injective, this means $M^{\gp}\cap N=M$ inside $N^{\gp}$, and in particular $M^{\gp}\cap N^\times=M^{\times}$. 
Note that a map of integral monoids is exact surjective if and only if it is exact and surjective, so no confusion occurs.

We say that a map $M\to N$ from a sharp\footnote{A monoid $M$ is sharp if $M^\times=\{e\}$; this assumption is not essential, and put here only for avoiding conflict of notation $N^{(1)}$ in the main text.} monoid to a monoid is of \emph{Cartier type} if it is integral and the relative Frobenius $N^{(1)}\to N$ is exact, where $N^{(1)}$ is the pushout of $N$ along the $p$-th power map of $M$.

Back to the setting at the beginning.   
We will assume that $M_A/M_A^{\times} \to P/P^\times$ is of Cartier type. This is equivalent to that $P^{(1)}\to P$ is exact. 
We say that $M_A\to P$ is of Cartier type as well.  

\begin{thm}\label{crystalline comparison}
We assume that $M_A\to P$ is \emph{of Cartier type} and $(R, P)$ admits an exact surjection 
$(\widetilde{R}, \widetilde{P})\to (R, P)$ from a smooth lift $(\widetilde{R}, \widetilde{P})$ over $(A/p, M_A)$.  
Then, there exists a canonical isomorphism
\[
\Prism_{(R^{(1)}, P^{(1)})/ (\phi_* A, \phi_* M_A)} \cong \phi_* R\Gamma_{\crys} ((R, P)/ (A, M_A))
\]
of $E_{\infty}$-$\phi_* A$-algebras compatible with the Frobenius. 
\end{thm}

\begin{rem}
By Corollary \ref{base change II}, one can replace the left hand side by the $p$-completed base change of $\Prism_{(\widetilde{R}, \widetilde{P})/ (A, M_A)}$. 
\end{rem}

The required lifting $(\widetilde{R}, \widetilde{P})$ of $(R,P)$ perhaps always exists as in the case of schemes \cite{Stacks}*{\href{https://stacks.math.columbia.edu/tag/07M8}{Tag 07M8}}, but we leave it as an assumption.
It exists if $P$ itself is a smooth chart, i.e., the one appearing in the definition of smoothness (see Definition \ref{smooth}), of the associated log scheme $(X, M_X)$ over $(A/I, M_A)$. In particular, the lifting exists \'etale locally, and in fact we can globalize the crystalline comparison:

\begin{thm}[The crystalline comparison]\label{crystalline comparison:global}
Let $(X, M_X)$ be a smooth log scheme over $(A/I, M_A)$ of Cartier type in the sense of \cite{Kato:log}*{4.8}. 
There is a canonical isomorphism of $E_{\infty}$-$\phi_* A$-algebras on $X_{\et}$:
\[
\Prism_{(X^{(1)}, M_X^{(1)})/(\phi_* A, \phi_* M_A)}\cong \phi_* Ru^{\crys}_{*} \cO_{\crys}
\]
in $D(X_{\et}, \phi_* A)$ under the natural identification $D(X_{\et}^{(1)}, \phi_* A)\cong D(X_{\et}, \phi_* A)$, where $u^{\crys}$ denotes the projection from the crystalline topos and $(X^{(1)}, M_X^{(1)})$ is the base change of $(X, M_X)$ along $\psi$. 
\end{thm}

Under the above assumption, one can locally make a smooth chart of Cartier type. 
We essentially follow \cite{BS}*{Theorem 5.2} for the proof, but we formulate it in a more site-theoretic way.

\subsection{The \texorpdfstring{$\delta_{\log}$}{delta-log}-crystalline site}
The $\delta_{\log}$-crystalline site is a special case of the log $q$-crystalline site in \S \ref{log q-crystalline site} below. 
It is a modification of the big log crystalline site such that each object is equipped with a Frobenius lift, i.e., $\delta$- and $\delta_{\log}$-structure. 

We shall consider a general log $p$-adic formal scheme $(X, M_X)$ over $(A, M_A)$ instead of $(R, P)$.

\begin{defn}[The $\delta_{\log}$-crystalline site]\label{log delta-crystalline}
We define the (big) \emph{$\delta_{\log}$-crystalline site} $((X, M_X)/(A, M_A))_{\delta\CRYS}$ as the opposite of the category whose object consists of 
\begin{itemize}
    \item a log prism $(B, (p), M_{\Spf (B)})=(B, (p), M_B)^a$ associated to a bounded prelog prism $(B, (p), M_B)$ over $(A, (p), M_A)$ with integral log structure, and a $p$-completed PD ideal $J\subset B$ such that $B/J$ is classically $p$-complete (we call the triple $(B, J, M_B)^a\coloneqq (B, J, M_{\Spf (B)})$ a \emph{$\delta_{\log}$-PD triple} over $(A, M_A)$), 
    \item a map of $p$-adic formal schemes $f\colon \Spf (B/J)\to X$ over $A$, and
    \item an exact closed immersion of log $p$-adic formal schemes 
    \[
    (\Spf (B/J), f^*M_X)\hookrightarrow (\Spf (B), M_{\Spf (B)})
    \]
     over $(A, M_A)$.
\end{itemize}
Morphisms are obvious ones. We endow it with the \'etale topology as in the case of log prismatic site. 
(Any \'etale cover of $B/J$ lifts uniquely to an \'etale cover of $(B, J, M_B)^a$.)
The structure sheaf $\cO_{\delta\CRYS}$ is defined by sending $(B, J, M_B)^a$ to $B$, which is indeed a sheaf. 
We write $R\Gamma_{\delta\CRYS}((X,M_X)/(D, M_D))$ for the cohomology of the structure sheaf. 
\end{defn}

We only allowed $(B, (p), M_B)$ that admits a map from $(A, (p), M_A)$, and this is not natural as it depends on the chart $M_A\to A$ rather than the associated log structure. While it is also possible to drop this condition, this slightly simplifies the proof of the crystalline comparison. 

\begin{rem}
For a $\delta_{\log}$-PD triple $(B, J, M_B)^a$ over $(A, M_A)$, $J$ admits divided powers as $B$ is assumed to be $p$-torsionfree. Thus, for any PD-ideal $I$ of $A$, the divided power structure of $I$ extends to $IB$, and further to $J+IB$. In other words, the divided power structures of $I, J$ are always compatible. 
\end{rem}

\begin{rem}[Relation to a mixed-characteristic log crystalline site]
We define a version of the big log crystalline site\footnote{The notation here is traditionally used for crystalline topoi rather than crystalline sites.} $((X, M_X)/(A, M_A))_{\CRYS}$ with the \'etale topology by dropping $\delta$- and $\delta_{\log}$-structures on $B$ in the definition of $\delta_{\log}$-crystalline site. 
Its structure sheaf is denoted by $\cO_{\CRYS}$. 
There is a morphism of topoi
\[
u_X\colon \Shv(((X, M_X)/(A, M_A))_{\CRYS})\to \Shv(X_{\et})
\]
as usual. 

There is a natural cocontinuous functor of sites
\[
((X, M_X)/(A, M_A))_{\delta\CRYS} \to ((X, M_X)/(A, M_A))_{\CRYS}
\]
by forgetting $\delta$- and $\delta_{\log}$-structures, identifying \'etale coverings. 
It induces a morphism of topoi
\[
\Shv(((X, M_X)/(A, M_A))_{\delta\CRYS})\to \Shv(((X, M_X)/(A, M_A))_{\CRYS}), 
\]
and its composite with $u_X$:
\[
u_X^{\delta} \colon \Shv(((X, M_X)/(A, M_A))_{\delta\CRYS})\to \Shv(X_{\et}). 
\]
The pushforward along $u_X, u_X^{\delta}$ are computed as in the case of log prismatic site. 
Finally, as $\cO_{\CRYS}$ restricts to $\cO_{\delta\CRYS}$, we have a canonical map
\[
Ru_{X*}\cO_{\CRYS}\to Ru_{X*}^{\delta}\cO_{\delta\CRYS}. 
\]
\end{rem}

\begin{rem}[Relation to usual log crystalline sites]\label{small log crystalline sites}
For simplicity, suppose $X$ lives over $A/I$ for a $p$-completed PD ideal $I \subset A$ containing $p$. 
For any integer $m \geq 1$, the reduction mod $p^m$ defines a cocontinuous functor
\[
((X, M_X)/(A, M_A))_{\CRYS} \to ((X, M_X)/ (A/p^m, M_A))_{\CRYS}
\]
to the log version of the big crystalline site with affine objects and the \'etale topology, and the local computation of $Ru_{X*}\cO_{\CRYS}$ given in the proof of Proposition \ref{delta-crys vs crys} below shows that it induces the following isomorphisms
\begin{align*}
Ru_*^{\crys} \cO_{(X, M_X)/ (A/p^m, M_A)}&\overset{\cong}{\longleftarrow} Ru_*^{\CRYS} \cO_{(X, M_X)/ (A/p^m, M_A)} \\
&\overset{\cong}{\longrightarrow} Ru_{X*}(\cO_{\CRYS}\otimes_{A}^L A/p^m) \\
&\overset{\cong}{\longleftarrow}(Ru_{X*}\cO_{\CRYS})\otimes_{A}^L A/p^m
\end{align*}
compatibly with $m$, where $u_*^{\CRYS}$ denote the projection from the big crystalline site. 
The last identification is the projection formula \cite{Stacks}*{\href{https://stacks.math.columbia.edu/tag/0943}{Tag 0943}}.

The isomorphisms lift to an isomorphism in $D(X_{\et} \times \bN, A)$. 
Therefore, taking the derived limit, we obtain a canonical isomorphism
\[
Ru_*^{\crys} \cO_{(X, M_X)/ (A, M_A)} \cong Ru_{X*}\cO_{\CRYS}
\]
since $Ru_{X*}\cO_{\CRYS}$ is $p$-complete \cite{Stacks}*{\href{https://stacks.math.columbia.edu/tag/099J}{Tag 099J}}. 
(Use also \cite{Beilinson}*{1.12}.)
\end{rem}

\begin{prop}[$\delta_{\log}$-crystalline cohomology and log crystalline cohomology]\label{delta-crys vs crys}
Let $I\subset A$ be a $p$-completed PD ideal such that $A/I$ is classically $p$-complete. 
If $(X, M_X)$ is smooth over $(A/I, M_A)$, the natural map induces an isomorphism of  $E_{\infty}$-$A$-algebras on $X_{\et}$:
\[
Ru_{X*}\cO_{\CRYS}\overset{\cong}{\longrightarrow} Ru_{X*}^{\delta}\cO_{\delta\CRYS}. 
\]
\end{prop}

\begin{proof}
It suffices to show that the restriction map
\begin{multline*}
R\Gamma (((X, M_X)/(A, M_A))_{\CRYS}, \cO_{\CRYS}) \\
\to 
R\Gamma (((X, M_X)/(A, M_A))_{\delta\CRYS}, \cO_{\delta\CRYS})
\end{multline*}
is an isomorphism if $(X, M_X)$ comes from a prelog ring $(R, P)$ with a lift $(\widetilde{R}, \widetilde{P})$ as in Theorem \ref{crystalline comparison} but we require neither that $I$ contain $p$ nor $M_A\to \widetilde{P}$ be exact. 
We can even assume that $P$ itself is a smooth chart.

Take a surjection
\[
(B_0, M_B)=(A\langle (X_s)_{s\in S}, \bN^T\rangle, M_A\oplus \bN^T)\to (\widetilde{R}, \widetilde{P})
\]
over $(A, M_A)$ for some sets $S$, $T$ whose kernel is Zariski locally generated by $p$ and a $p$-completely regular sequence relative to $A$. 
Denote the $p$-completed \v{C}ech nerve of $(A, M_A)\to (B_0, M_B)$ by $(B_0^{\bullet}, M_B^{\bullet})$ with surjections $(B_0^{\bullet}, M_B^{\bullet})\to (R, P)$, 
and their $p$-completed log PD envelopes by $(C_0^{\bullet}, M_{C}^{\bullet})$ regarded as $p$-adically log-affine prelog rings. 
By an argument analogous to Proposition \ref{Prelog prismatice envelops for smooth log algebras}, we see that each $C_0^{\bullet}$ is $p$-completely flat over $A$ (cf. \cite{BS}*{Lemma 2.42}), hence $p$-torsionfree as $A$ is so. 
Therefore, it is the \v{C}ech nerve of $(C_0^{\bullet}, M_{C}^{\bullet})$ in $((X, M_X)/(A, M_A))_{\CRYS}$.
Write $J$ for the kernel of $C_0 \to R$. 
As $(C_0, J, M_C)^a$ is a weakly final object of $((X, M_X)/(A, M_A))_{\CRYS}$, the cosimplicial $A$-algebra $C_0^{\bullet}$ computes the cohomology $R\Gamma (((X, M_X)/(A, M_A))_{\CRYS}, \cO_{\CRYS})$; compare with Construction \ref{\v{C}ech}. 

We have a similar computation for the $\delta_{\log}$-crystalline cohomology, parallel to Construction \ref{\v{C}ech}. 
Namely, let
\[
B=(A\{ (X_s)_{s\in S}\}^{\delta}\{\bN^T\}_{\log}^{\delta})^{\wedge}_{p}, 
\]
and write $C$ for the $p$-completion of $B \otimes_{B_0}C_0$, which is $p$-torsionfree. 
Since $B_0\to B$ is $p$-completely flat, $(JC)^{\wedge}_p$ is a PD ideal of $C$ and $C$ is identified with a $p$-completed log PD envelope of $(B, M_B)$ with a surjection $M_B\to \widetilde{P}$ and an ideal generated, Zariski locally, by $p$ and a $p$-completely regular sequence relative to $A$. Therefore, $C$ is a $\delta$-ring over $C_0$ by \cite{BS}*{Corollary 2.38} and Zariski descent.  
It is easy to see that the $\delta_{\log}$-PD triple associated to $(C, (JC)^{\wedge}_p, M_C)$ is weakly final in the category $((X, M_X)/(A, M_A))_{\delta\CRYS}$, and its \v{C}ech nerve $C^{\bullet}$ computes the cohomology $R\Gamma (((X, M_X)/(A, M_A))_{\delta\CRYS}, \cO_{\delta\CRYS})$. Moreover, the \v{C}ech nerve has the following description:
\[
C^{\bullet}\cong (B^{\bullet}\otimes_{B_0^{\bullet}}C_0^{\bullet})^{\wedge}_p, 
\]
where $B^{\bullet}$ comes from the $p$-completed \v{C}ech nerve of $(A, M_A)\to (B, M_B)$.  
(To see that the right hand side comes from a simplicial object of $((X, M_X)/(A, M_A))_{\delta\CRYS}$, we use again \cite{BS}*{Corollary 2.38} and Zariski descent.)

To conclude, it is enough to observe that the natural map $C_0^{\bullet}\to C^{\bullet}$ is a cosimplicial homotopy equivalence because $B_0^{\bullet}\to B^{\bullet}$ is so by the $p$-completely freeness of $B$ over $B_0$. 
(Another way to conclude is to regard $C_0\to C$ as a covering in the big crystalline site with the ind-smooth topology and argue that $C^{\bullet}$ computes the same cohomology.)
\end{proof}

\begin{rem}[Smooth $\delta_{\log}$-algebras compute the $\delta_{\log}$-crystalline cohomology]\label{small smooth-delta-algebra computes delta-crystalline cohomology}
For simplicity assume $I$ contains $p$. 
Suppose $X=\Spec (R)$ is affine with a chart $M_A\to P$ that is integral and weakly finitely generated, and let $(B, M_B)\to (R, P)$ be a surjective map from a $p$-completely smooth $\delta_{\log}$-ring $(B, M_B)$ of topologically of finite presentation over $(A, M_A)$. 
Then, its log PD envelope modulo $p^m$, for every integer $m\geq 1$, is PD-smooth over $(A/p^m, M_A)$ in the sense of \cite{Beilinson}*{1.4} by Remark \ref{strong lifting property}. 
(Compare with \cite{Beilinson}*{1.4 Remarks (ii)}.)

Therefore, by \cite{Beilinson}*{1.6}, the ($p$-completed) \v{C}ech nerve $C^{\bullet}$ of the $p$-completed log PD envelope $(C, M_C)$ either in the $\delta_{\log}$-crystalline site or the log crystalline site computes the log crystalline cohomology of $(X, M_X)$ over $(A, M_A)$. 
Combined with Remark \ref{small log crystalline sites} and Proposition \ref{delta-crys vs crys}, $C^{\bullet}$ also computes the $\delta_{\log}$-crystalline cohomology. Moreover, the log de Rham complex with coefficients in $C$ computes these cohomology as well \cite{Beilinson}*{1.8, 1.7 Exercises (i)}. 
\end{rem}

\subsection{Proof of the crystalline comparison}
First we define a cocontinuous functor
\[
((X, M_X)/(A, M_A))_{\delta\CRYS} \to ((X^{(1)}, M_X^{(1)})/(\phi_* A, \phi_* M_A))_{\Prism}. 
\]
Let $(B, J, M_B)^a$ be an object of $((X, M_X)/(A, M_A))_{\delta\CRYS}$. We may assume $(B, M_B)$ is a log ring. Take the pushout
\[
\begin{CD}
M_A @>\phi_{M_A} >> \phi_* M_A \\
@VVV @VVV \\
M_{B} @>>> M_{B}^{(1)} 
\end{CD}
\]
with a map $M_{B}^{(1)}\to \phi_* M_{B}$ induced by $\phi_{M_B}$ from Remark \ref{Frobenius lifts on monoids}. 
Taking the composite of the last map with $M_B\to B$, we obtain a prelog ring $(\phi_* B, M_B^{(1)})$ over $(\phi_* A, \phi_* M_A)$. 
Then, the Frobenius on $(B/p, M_B)$ factors, where $M_B$ is acted by $\phi_{M_B}$, as
\[
(B/p, M_B)\to (B/J, M_B)\overset{\psi_B}{\longrightarrow} (\phi_* B/p, M_B^{(1)}) \to (\phi_*B/p, \phi_* M_B).   
\]
The composite 
\[
\Spf (\phi_* B/p) \overset{\psi_B}{\longrightarrow} \Spf (B/J) \to X
\]
factors through $X^{(1)}$ over $\phi_{*}A/p$.  
One can check that 
\begin{center}
$(\phi_* B, (p), M_B^{(1)})^a$ with $\Spf(\phi_* B/p) \to X^{(1)}$
\end{center}
comes naturally from an object of $((X^{(1)}, M_X^{(1)})/(\phi_* A, \phi_* M_A))_{\Prism}$. 
This defines the desired functor. It is cocontinuous because $\psi_B$ identifies the \'etale covers of $B/J$ and $B/p$. 
Moreover, it gives rise to a morphism of ringed topoi
\begin{multline*}
(\Shv(((X, M_X)/(A, M_A))_{\delta\CRYS}), \phi_* \cO_{\delta\CRYS}) \\
\to
(\Shv((X^{(1)}, M_X^{(1)})/(\phi_* A, \phi_* M_A))_{\Prism}, \cO_{\Prism}). 
\end{multline*}
Thus, we obtain a comparison map as in Theorem \ref{crystalline comparison:global}. 

To prove the comparison map is an isomorphism, we may work locally and it suffices to prove Theorem \ref{crystalline comparison}. 
We recall the computation of $\delta_{\log}$-crystalline cohomology from the proof of Proposition \ref{delta-crys vs crys}. 
Take $(\widetilde{R}, \widetilde{P})$ as in the assumption of Theorem \ref{crystalline comparison}. 
Choose a surjection
\[
(B_0, M_B^0)=(A\langle (X_s)_{s\in S}, \bN^T\rangle, M_A\oplus\bN^T)\to (\widetilde{R}, \widetilde{P})
\]
over $(A, M_A)$ for some sets $S$, $T$ whose kernel is Zariski locally generated by $p$ and a $p$-completely regular sequence relative to $A$. 
Let
\[
B=(A\{ (X_s)_{s\in S}\}^{\delta}\{\bN^T\}_{\log}^{\delta})^{\wedge}_{p}, 
\]
and $(B^{\bullet}, M_B^{\bullet})$ the \v{C}ech nerve of $(A, M_A)\to (B, M_B^0)$. 
A certain $p$-completed log PD envelope $(C^{\bullet}, M_C^{\bullet})$ computes the $\delta_{\log}$-crystalline cohomology as in the proof of \ref{delta-crys vs crys}. 
More precisely, it is given as follows. Let $(B_0^{\bullet}, M_B^{\bullet})$ denotes the \v{C}ech nerve of $(A, M_A)\to (B_0, M_B^0)$. We have a surjection $(B_0^{\bullet}, M_B^{\bullet})\to (\widetilde{R}, \widetilde{P})$, and denote its $p$-completed exactification by $((B_0^{\bullet})', (M_B^{\bullet})')$ and the $p$-completed log PD envelope by $(C_0^{\bullet}, M_C^{\bullet})$. 
If $J^{\bullet}$ denotes the kernel of $(B_0^{\bullet})'\to \widetilde{R}$, $C_0^{\bullet}$ is obtained as  
the $p$-completed PD envelope of $J^{\bullet} \subset (B_0^{\bullet})'$. 
Finally, $C^{\bullet}$ is the $p$-completion of $B^{\bullet}\otimes_{B_0^{\bullet}}C_0^{\bullet}$, 
which is the $p$-completed PD envelope of 
\[
(B^{\bullet})'\coloneqq (B^{\bullet}\otimes_{B_0^{\bullet}}(B_0^{\bullet})')^{\wedge}_{p}
\]
with the $p$-completed ideal generated by $J^{\bullet}$. 

On the other hand, starting from the following surjection
\[
((\phi_*\phi^* B_0)^{\wedge}_p, M_B^{0(1)})\to (\widetilde{R}^{(1)}, \widetilde{P}^{(1)}) 
\]
over $(\phi_* A, \phi_* M_A)$, we can compute the log prismatic cohomology as in Construction \ref{\v{C}ech}. 
Using the above notation, consider 
\begin{center}
$\phi_*\phi^* J^{\bullet}\subset \phi_*\phi^* (B_0^{\bullet})'$ and its extension $\phi_*(\phi^* J^{\bullet}(B^{\bullet})')^{\wedge}_p\subset \phi_*(\phi^* (B^{\bullet})')^{\wedge}_p$. 
\end{center}
The prismatic envelope $D^{\bullet}$ over $(\phi_* A, (p))$ of the latter computes the log prismatic cohomology. 

The Frobenius $\phi_{(B^{\bullet})'}\colon (B^{\bullet})'\to \phi_*(B^{\bullet})'$ extends to
\[
1\otimes \phi_{(B^{\bullet})'}\colon \phi_*\phi^* (B^{\bullet})'\to \phi_* (B^{\bullet})', 
\]
and it further extends to
\[
D^{\bullet}\to \phi_* C^{\bullet}
\]
by the description of the PD envelope \cite{BS}*{Corollary 2.38}. 
These maps come from a map of $\delta_{\log}$-rings
\[
(\phi_*\phi^* B, M_B^{0(1)})\overset{1\otimes \phi_B}{\longrightarrow} (\phi_* B, M_B^{0(1)})
\]
identifying $M_B^{0(1)}$. 
Therefore, the extended map is compatible with the morphism of ringed topoi and the calculation of the cohomology of structure sheaves, and it remains to show that it is an isomorphism. 
We will apply Proposition \ref{Relative Frobenius is a quasi-isomorphism} and \cite{BS}*{Lemma 5.4}. 
Indeed, we have the following description. Recall that 
\begin{center}
$(B_0^0)'$ is $p$-completely free over the $p$-completion of $A\otimes_{\bZ_{(p)}[ M_A]} \bZ_{(p)}[(M_{B}^{0})']$.
\end{center} 
Set
\[
Q_1=M_A, \quad Q_2=(M_{B}^{0})', \quad G=Q_2^{\gp}/Q_1^{\gp},
\]
$Q_1\to Q_2$ being of Cartier type by our assumption. 
Recall also that $G$ is a free abelian group. 
So, $(M_{B}^{n})'$ is isomorphic to $Q_2 \oplus G^n$, and we have the following cosimplicial $A$-algebra
\[
A \otimes_{\bZ_{(p)} [Q_1]} \bZ_{(p)}[Q_2]\to A\otimes_{\bZ_{(p)} [Q_1]} \bZ_{(p)}[Q_2\oplus G] \to
A \otimes_{\bZ_{(p)} [Q_1]} \bZ_{(p)}[Q_2\oplus G^2] \to \cdots 
\]
as in Appendix B. 
Up to the contribution from the ``free part'' and $p$-completion, this is the description of $(B^{\bullet})'$, and the map $D^{\bullet}\to \phi_* C^{\bullet}$ is induced from the relative Frobenius of $(B^{\bullet})'$ by \cite{BS}*{Corollary 2.38}. 
Therefore, the mod $p$ reduction of $D^{\bullet}\to \phi_* C^{\bullet}$ decomposes into two maps to which we can apply either Proposition \ref{Relative Frobenius is a quasi-isomorphism} or \cite{BS}*{Lemma 5.4} (this is for ``relative Frobenius'' on the free part). 
Thus, $D^{\bullet}\to \phi_* C^{\bullet}$ induces a quasi-isomorphism on the associated complexes. 

\section{Logarithmic \texorpdfstring{$q$}{q}-crystalline cohomology and logarithmic \texorpdfstring{$q$}{q}-de Rham complexes}
We introduce a logarithmic variant of (a big version of) the $q$-crystalline site of \cite{BS}. 
This is a $q$-deformation of the $\delta_{\log}$-crystalline site we introduced before. 

Let $A=\bZ_p [\![ q-1 ]\!]$ with $\delta$-structure given by $\delta(q)=0$. 
The $q$-analogue of $p$ is defined by 
\[
[p]_q=\frac{q^p-1}{q-1}=1+q+\cdots q^{p-1}
\]
and $[p]_q$ specializes to $p$ under $q\mapsto 1$.  
We have a prism $(A, ([p]_q))$, called the \emph{$q$-de Rham prism} or \emph{$q$-crystalline prism}. 

If $x$ is an element of a $[p]_q$-torsionfree $\delta$-ring $D$ over $A$ such that $\phi(x)\in [p]_q D$, we define
\[
\gamma (x)=\frac{\phi (x)}{[p]_q} - \delta (x). 
\]

\subsection{Log \texorpdfstring{$q$}{q}-divided power thickenings}
First we recall from \cite{BS} that a \emph{$q$-PD pair} is a $(p, [p]_q)$-complete $\delta$-pair $(D,I)$ over $(A, (q-1))$ satisfying
\begin{enumerate}
    \item $(D, ([p]_q))$ is a bounded prism over $(A, ([p]_q))$, 
    \item $\phi (I) \subset [p]_q D$ and $\gamma (I)\subset I$, 
    \item $D /(q-1)$ is $p$-torsionfree with finite $(p, [p]_q)$-complete Tor-amplitude over $D$, 
    \item $D/I$ is classically $p$-complete. 
\end{enumerate}
The final condition is not required in \cite{BS}*{Definition 16.2}, but it makes possible to consider the affine $p$-adic formal scheme $\Spf (D/I)$, which will be used to define a global version the $q$-crystalline site. 

We have a logarithmic variant:

\begin{defn}[log $q$-PD triples]\label{q-PD triples}
A \emph{prelog $q$-PD triple} is a triple $(D, I, M_D)$, where $(D, I)$ is a $q$-PD pair and $(D, I, M_D)$ is a $\delta_{\log}$-triple.  
A \emph{log $q$-PD triple} is a triple $(D, I, M_{\Spf (D)})$, where $(D, I)$ is a $q$-PD pair and $(D, [p]_q, M_{\Spf (D)})=(D, [p]_q, M_D)^a$ for some prelog $q$-PD triple $(D, I, M_D)$.  
\end{defn}

In particular, we recover $\delta_{\log}$-PD triples (over $\bZ_p$ with the trivial prelog structure) if $q=1$. 
Let us call a prelog $q$-PD triple a \emph{pre-$\delta_{\log}$-PD triple} if $q=1$. 

Prelog $q$-PD triples form a category, and pre-$\delta_{\log}$-PD triples form its full subcategory. 
If $(D, I, M_D)$ is a prelog $q$-PD triple, we have the associated log $q$-PD triple $(D, I, M_D)^a$. 

If $(D, I, M_D)$ is a prelog $q$-PD triple, then $(D/(q-1), I/(q-1), M_D)$ is a pre-$\delta_{\log}$-PD triple. 
Suppose $(D, I, M_D)$ is a $\delta_{\log}$-triple, where $D$ is regarded as an $A$-algebra with $q=1$.  Then, $(D, I, M_D)$ is a pre-$\delta_{\log}$-PD triple if and only if
\begin{enumerate}
    \item $D$ is $p$-torsionfree and $p$-adically complete, 
    \item $D/I$ is classically $p$-complete, and
    \item $I$ admits divided powers. 
\end{enumerate}

In particular, $I$ is locally nilpotent modulo $p^m D$ for any positive integer $m$, and we will freely use Remark \ref{strong lifting property} for exact closed immersions defined by $I$. 

\begin{exam}
Let $C$ be an algebraically closed perfectoid field of characteristic $0$, and let $\epsilon$ be a usual element of $C^\flat$ for a fixed choice of compatible $p$-power roots of unity. 
Set $q=[\epsilon]$ and $\xi=\phi^{-1}([p]_q)$. 
The triple 
\[
(A_{\Inf}(\cO_C), (\xi), \cO_C^\flat \setminus\{0\})
\]
is a prelog $q$-PD triple. 
\end{exam}

\begin{lem}\label{\'etale maps and q-PD}
Let $(D, I, M_D)$ be a prelog $q$-PD triple. If $D/I\to \overline{E}$ is $p$-completely \'etale, there exists a unique prelog $q$-PD triple $(E, J, M_D)$ over $(D, I, M_D)$ lifting $D/I\to \overline{E}$. 
\end{lem}

\begin{proof}
There is a unique $(p,[p]_q)$-completely \'etale lift $D\to E$ of $D/I\to \overline{E}$, which is a $\delta_{\log}$-ring over $D$ in a unique way by Lemma \ref{\'etale maps}. In particular, $E$ is $(p, [p]_q)$-completely flat over $D$. This lemma is essentially \cite{BS}*{Lemma 16.5.(5)} applied to $D\to E$. 
The kernel of $E\to \overline{E}$ is given by $J\coloneqq (IE)^{\wedge}_{(p,[p]_q)}$. 
As $E$ is classically $(p, [p]_q)$-complete, $J$ agrees with the classical $(p, [p]_q)$-completion of $IE$. 
To show that $(E, J, M_D)$ is a prelog $q$-PD triple, we only check the second half of the condition (2) of Definition \ref{q-PD triples}; this is basically \cite{BS}*{Lemma 16.5.(4)}. 
As $(D, I)$ is a $q$-PD pair, $\gamma (I)\subset I$. Recall the following general formulas from \cite{BS}*{Remark 16.6}:
\begin{align*}
\gamma (x+y)&=\gamma (x)+\gamma(y) +\sum_{i=1}^{p-1}\frac{(p-1)!}{i!(p-i)!} x^i y^{p-i},  \\
\gamma (fx)&=\phi (f)\gamma (x)-x^p \delta(f). 
\end{align*}
In particular, we see that $\gamma (IE)\subset IE$. It is easy to check that
\[
\delta (p^m) \equiv 0 \mod p^{m-1}, \quad \phi((q-1)^n), \delta((q-1)^n) \equiv 0 \mod (q-1)^n
\]
for integers $m, n\geq 1$. 
Again using the formulas, we see that 
\[
\gamma(p^m IE + (q-1)^n IE) \subset p^{m-1} IE + (q-1)^n IE. 
\]
Therefore, for any $j\in J$, 
\[
\gamma (j+p^m IE + (q-1)^n IE) \in \gamma (j) + (p^{m-1}, (q-1)^n)J. 
\]
As $J$ is with the classical $(p, [p]_q)$-completion of $IE$, we conclude that $\gamma (j) \in J$. 
(Note that $(p,[p]_q)$-completion and $(p, q-1)$-completion agree.)
\end{proof}

We shall prove the existence of log $q$-PD envelopes. 

\begin{lem}[Existence of log $q$-PD envelopes for smooth log algebras]\label{log q-PD envelope}
Fix a prelog $q$-PD triple $(D_1, I_1, M_{D_1})$ with integral $M_{D_1}$, and let $(R, P)$ be a $p$-completely smooth prelog ring over $(D_1/I_1, M_{D_1})$ such that $M_{D_1}\to P$ is \emph{integral} and \emph{weakly finitely generated}. Assume that $(R, P)^a$ admits a \emph{smooth lift} over $(D_1, M_{D_1})$. Let $(D_{2,0}, M_{D_2})\to (R, P)$ be a surjection from a prelog ring over $(D_1, M_{D_1})$ satisfying either 
\begin{itemize}
    \item $(D_{2,0}, M_{D_2})$ is a $(p,[p]_q)$-completely smooth $\delta_{\log}$-ring of topologically finite presentation over $(D_1, M_{D_1})$, or
    \item $(*)_{q}$: $M_{D_1}\to M_{D_2}$ is injective and integral, $M_{D_2}^{\gp}/M_{D_1}^{\gp}$ is a free abelian group, and $D_{2,0}$ is $(p, [p]_q)$-completely free over the $(p, [p]_q)$-completion of 
$D_1 \otimes_{\bZ_{(p)}[M_{D_1}]}\bZ_{(p)}[M_{D_2}]$.  
\end{itemize}
In the second case, let $(D_2, M_{D_2})$ denote the $(p, [p]_q)$-completed universal $\delta_{\log}$-ring over $(D_1, M_{D_1})$ generated by $(D_{2,0}, M_{D_2})$.  
Put $D_2= D_{2,0}$ in the first case. 
Write $I_{2,0}$ for the kernel of $D_{2,0} \to R$, and $I_2$ for the $(p,[p]_q)$-completion of $I_{2,0}D_2$. 

\begin{enumerate}
    \item There exists a universal map $(D_2, I_2, M_{D_2})\to (D_3, I_3, M_{D_3})$ to a prelog $q$-PD triple with an exact surjection $M_{D_3} \to P$ that fits into a commutative diagram
    \[
    \begin{CD}
    M_{D_3} @>>> P \\
    @VVV @VVV \\
    D_3/I_3 @<<< D_2/ I_2=R. 
    \end{CD}
    \]
    \item $D_3$ is $(p, [p]_q)$-completely flat over $D_1$. 
    \item The bottom arrow in the diagram is an \emph{isomorphism} $D_2 / I_2 \cong D_3 /I_3$. 
    \item The map of log $q$-PD triples $(D_2, I_2, M_{D_2})^a\to (D_3, I_3, M_{D_3})^a$ is universal among factorizations of $(\Spf (R), P^a_{\Spf (R)})\to (\Spf (D_2), M^a_{D_2, \Spf (D_2)})$ into an exact closed immersion
    \[
    (\Spf (R), P_{\Spf (R)}^a)\cong (\Spf (D/I), P^a_{\Spf (D/I)}) \hookrightarrow (\Spf (D), M_{\Spf (D)})
    \]
    for log $q$-PD triples $(D, I, M_{\Spf (D)})$ over $(D_2, I_2, M_{D_2})^a$ with integral $M_{\Spf (D)}$. 
    \item\label{log PD envelope} The association $(D_2, I_2, M_{D_2})\mapsto (D_3, I_3, M_{D_3})$ commutes with $(p, [p]_q)$-completed base change along $(D_1, I_1, M_{D_1})\to (D'_1, I'_1, M_{D_1})$ of prelog $q$-PD triples. In particular, $D_3 \widehat{\otimes}_{D_1} D_1 / (q-1)$ is the $p$-completed log PD envelope of $I_2 /(q-1) \subset D_2 / (q-1)$. 
\end{enumerate}
\end{lem}

In light of (4), $(D_3, I_3, M_{D_3})^a$ is called the \emph{log $q$-PD envelope} of $(\Spf (R), P)^a\hookrightarrow (\Spf (D_2), M_{D_2})^a$. 

\begin{proof}
To simplify the notation, we only discuss the first case. The case $(*)_q$ is similar, following the proof of Proposition \ref{Prelog prismatice envelops for smooth log algebras}. 

Assume that $(D_2, M_{D_2})$ is $(p, [p]_q)$-completely smooth and topologically of finite presentaiton over $(D_1, M_{D_1})$. 
We first prove (1)-(3) and (5). 
Applying Construction \ref{Exactification} to $(D_2, M_{D_2})\to (R, P)$, we may assume that $(D_2, M_{D_2})\to (R, P)$ is exact surjective by Remark \ref{Exactification and integral maps}. 

By the assumption, $(R, P)^a$ admits a $(p, [p]_q)$-completely smooth lift $(\widetilde{R}, \widetilde{P})$ over $(D_1, M_{D_1})$, and $(D_2, M_{D_2})^a\to (R, P)^a$ lifts to $(D_2, M_{D_2})^a\to (\widetilde{R}, \widetilde{P})$ modulo $(p^m, [p]_q^n)$ for each $m, n\geq 1$ by Remark \ref{strong lifting property}. This implies that $I_2$ is generated, Zariski locally, by $I_1$ and a $(p, [p]_q)$-completely regular sequence relative to $D_1$. Thus, we may apply \cite{BS}*{Lemma 16.10}. 

To check (4), one may proceed as in the proof of Proposition \ref{Existence of log prismatic envelopes} using (1) and (3). 
\end{proof}

\subsection{The log \texorpdfstring{$q$}{q}-crystalline site}\label{log q-crystalline site}
Fix a prelog $q$-PD triple $(D, I, M_{D})$ with integral $M_D$. 
Let $(X, M_X)$ be a smooth log $p$-adic formal scheme over $(D/I, M_D)$. 

\begin{defn}[The log $q$-crystalline site]
We define the log $q$-crystalline site $((X, M_X)/ (D, M_D))_{q\CRYS}$ as the opposite of the following category: an object is a log $q$-PD triple $(E, J, M_{\Spf (E)})$ associated with a prelog $q$-PD triple $(E, J, M_E)$ over $(D, I, M_D)$ and $M_E$ being integral, equipped with a morphism of $p$-adic formal schemes $f\colon \Spf (E/J)\to X$ over $D/I$ and an exact closed immersion 
\[
(\Spf (E/J), f^* M_X) \hookrightarrow (\Spf (E), M_{\Spf (E)})
\]
over $(D, M_D)$. 
Morphisms are obvious ones. 
We endow it with the \'etale topology; see Lemma \ref{\'etale maps and q-PD}. 
\end{defn}

Let $\cO_{q\CRYS}$ denote the structure sheaf $(E, J, M_E)^a\mapsto E$. 
We write 
\[
R\Gamma_{q\CRYS}((X, M_X)/(D, M_D))
\]
for the cohomology of the structure sheaf. 
It is a $(p, [p]_q)$-complete $E_{\infty}$-$D$-algebra equipped with a $\phi_D$-semilinear endomorphism. 

As usual, we have a morphism of topoi
\[
u^q_X\colon \Shv (((X, M_X)/ (D, M_D))_{q\CRYS})\to \Shv (X_{\et}). 
\]
We write $q\Omega_{(X, M_X)/ (D, M_D)}$ for $Ru^q_{X*}\cO_{q\CRYS}$,  
which is an $E_{\infty}$-$D$-algebra on $X_{\et}$ equipped with a $\phi_D$-semilinear endomorphism. 

If $X=\Spf(R)$ is affine and $P$ is a chart of $M_X$ over $M_D$, we also write 
\begin{center}
$((R, P)/ (D, M_D))_{q\CRYS}$ for $((X, M_X)/ (D, M_D))_{q\CRYS}$, 
\end{center}
and
\begin{center}
$q\Omega_{(R, M_R)/ (D, M_D)}$ for $R\Gamma (((R, P)/ (D, M_D))_{q\CRYS}, \cO_{q\CRYS})$. 

\end{center}

\begin{rem}
If $M_X$ is trivial, our $q$-crystalline site is close to that of \cite{BS}*{Remark 16.15 (2)}. 
\end{rem}

\begin{rem}[Relation to $\delta_{\log}$-crystalline sites]\label{log q- vs log crys:map}
If $q=1$ in $D$, the log $q$-crystalline site is exactly the $\delta_{\log}$-crystalline site. 
In general, there is a natural functor
\[
((X, M_X)/ (D/(q-1), M_D))_{\delta\CRYS} \to
((X, M_X)/ (D, M_D))_{q\CRYS}, 
\]
inducing
\[
q\Omega_{(X, M_X)/ (D, M_D)}\widehat{\otimes}^L_D D/(q-1) \to
Ru^{\delta}_{X*}\cO_{\delta\CRYS}. 
\]
We will show that this is an isomorphism in Theorem \ref{log q vs log crys}. 
\end{rem}

\begin{const}[Computing log $q$-crystalline cohomology]\label{compuing log q-crystalline cohomology}
Assume that $X=\Spf (R)$ is affine with a chart $P\to \Gamma (X, M_X)$ over $M_D$, and 
\begin{center}
    $M_{D}\to P$ is integral and weakly finitely generated. 
\end{center}
We also assume that $(X, M_X)$ admits a smooth lift over $(D, M_D)$. 

Take a surjection $M_{E}=M_D\oplus \bN^T\to P$ for some set $T$.  
Take a surjection from a $(p,[p]_q)$-completed polynomial ring
\[
E_0\coloneqq D\langle (X_{s})_{s\in S}, \bN^T\rangle\to R
\]
for some set $S$, compatible with the map $M_E\to P$. 

Let $(E, M_{E})$ be the $(p,I)$-completed $\delta_{\log}$-ring over $(D, M_D)$ generated by $E_0$:
\[
E=(D\{ (X_{s})_{s\in S}\}^{\delta}\{\bN^T\}_{\log}^{\delta})^{\wedge}_{(p,I)}.  
\]
Let $(E_0^{\bullet}, M_{E}^{\bullet})$ (resp. $(E^{\bullet}, M_{E}^{\bullet})$) be the $(p,[p]_q)$-completed \v{C}ech nerve of $(D, M_D)\to (E_0, M_E)$ (resp. $(D, M_D)\to (E_0, M_E)$), with a natural inclusion
\[
(E_0^{\bullet}, M_{E}^{\bullet}) \hookrightarrow (E^{\bullet}, M_{E}^{\bullet})
\]
identifying monoids. There is a surjection of prelog rings
\[
(E_0^{\bullet}, M_{E}^{\bullet})\to (R, P). 
\]

Each $(E_0^{\bullet}, M_E^{\bullet})$ satisfies the condition $(*)_q$ of Lemma \ref{log q-PD envelope}, thus we obtain a cosimplicial object $(F^{\bullet}, J^{\bullet}, M_F^{\bullet})^a$ in $((R, P)/(D, M_D))_{q\CRYS}$ with exact surjections $(F^{\bullet}, M_F^{\bullet})\to (F^{\bullet}/ J^{\bullet}, P)$.  
Each $F^{\bullet}$ is $(p,[p]_q)$-completely flat over $D$, and this cosimplicial object in $((R, P)/(D, M_D))_{q\CRYS}$ agrees with the \v{C}ech nerve of $(F^{0}, J^{0}, M_{F^{0}})^a$. 

By the freeness of $(E_0, M_E)$ as a prelog ring, for any object $(F', J', M_{\Spf (F')})$ of $((R, P)/(D, M_D))_{q\CRYS}$, there exists at least one map of prelog rings
\begin{center}
$(E_0, M_E)\to (F', P\times_{\Gamma (\Spf (F'/J'), M_{\Spf (F'/J')})} \Gamma (\Spf (F'), M_{\Spf (F')}))$
\end{center}
lifting the composite $(E_0, M_E)\to (R,P)\to (F'/J', P)^a$. The $\delta$- and $\delta_{\log}$-structures on $(F', M_{\Spf (F')})$ and the freeness of $(E, M_E)$ as a $\delta_{\log}$-ring enable us to extend it to a map of $\delta_{\log}$-rings
\[
(E, M_E) \to (F', P\times_{\Gamma (\Spf (F'/J'), M_{\Spf (F'/J')})} \Gamma (\Spf (F'), M_{\Spf (F')}))
\]
yielding a map of log $q$-PD triples
\[
(F^{0}, J^{0}, M_{J^{0}})^a \to (F', J', M_{\Spf (F')}). 
\]
Therefore, the object $(F^{0}, J^{0}, M_{F^{0}})^a$ in $((R, P)/(D, M_D))_{q\CRYS}$ represents a covering of the final object of the topos $\Shv (((R, P)/(A, M_A))_{q\CRYS})$. Thus, the cosimplicial $\delta$-$D$-algebra $F^{\bullet}$ computes $q\Omega_{(R, P)/(D, M_D)}$.  
\end{const}

\begin{rem}[Strictly functorial complexes computing log $q$-crystalline cohomology]\label{functorial:q-crys}
In Construction \ref{compuing log q-crystalline cohomology}, one may take $(E_0, M_E)\to (R, P)$ strictly functorial on $(R, P)$. 
Indeed, we set
\[
(E_0, M_E) := (D\langle \bN^R, \bN^P\rangle , M_D\oplus \bN^P). 
\]
This construction does not commute with base changes on $(D, M_D)$, but still functorial on $(D, M_D)$. 
\end{rem}

\begin{thm}[Log $q$-crystalline cohomology and log crystalline cohomology]\label{log q vs log crys}
The canonical map from Remark \ref{log q- vs log crys:map} induces an isomorphism
\[
q\Omega_{(X,M_X)/(D, M_D)}\widehat{\otimes}^L_D D/(q-1)\overset{\cong}{\longrightarrow} 
Ru^{\delta}_{X*}\cO_{\delta\CRYS}.
\]
Therefore, 
\[
q\Omega_{(R,P)/(D, M_D)}\widehat{\otimes}^L_D D/(q-1)\cong
Ru^{\crys}_*\cO_{(X, M_X)/(D/(q-1), M_D)}, 
\]
where the target is defined using the projection from the small crystalline site. 
\end{thm}

\begin{proof}
We may assume that $X=\Spf (R)$ is affine with a smooth chart $M_D\to P$. 
Let $F^{\bullet}$ a cosimplicial $\delta$-$D$-algebra computing $q\Omega_{(R,P)/(D, M_D)}$ as in Construction \ref{compuing log q-crystalline cohomology}. 
It also follows from Construction \ref{compuing log q-crystalline cohomology} and Lemma \ref{log q-PD envelope} that the term-wise $(p,[p]_q)$-completion $F^{\bullet}\widehat{\otimes}^L_D D/(q-1)$ computes the $\delta_{\log}$-crystalline cohomology. 
Since the totalization commutes in this setting \cite{BS}*{Lemma 16.5 (3)}, we see that
\[
q\Omega_{(R,P)/(D, M_D)}\widehat{\otimes}^L_D D/(q-1)\cong
R\Gamma_{\delta\CRYS}((X,M_X)/(D, M_D)). 
\]
The second point is a combination with Proposition \ref{delta-crys vs crys} and Remark \ref{small log crystalline sites}. 
\end{proof}

\begin{rem}[Smooth $\delta_{\log}$-algebras compute the log $q$-crystalline cohomology]\label{computing log q-crystalline cohomology via smooth lift}
Assume that $X=\Spf (R)$ with a chart $M_D\to P$ that is integral and weakly finitely generated, and $(X, M_X)$ admits a smooth lift over $(D,  M_D)$. 
Let $(E, M_E)\to (R, P)$ be a surjection from $(p, [p]_q)$-completely smooth $\delta_{\log}$-ring of topologically of finite presentation over $(D, M_D)$. 
By a similar procedure as in Construction \ref{compuing log q-crystalline cohomology}, we obtain $(F^{\bullet}, M_F^{\bullet})$. By Theorem \ref{log q vs log crys}, the cosimplicial object $(F^{\bullet}, M_F^{\bullet})$ also computes the log $q$-crystalline cohomology as it computes the log crystalline cohomology modulo $(q-1)$ by Remark \ref{small smooth-delta-algebra computes delta-crystalline cohomology}. (Also use \cite{BS}*{Lemma 16.5 (2)}.)
\end{rem}

\begin{lem}\label{Invariance under q-PD thickenings of the base}
Let $(D, I, M_D)\to (D,I', M_D)$ be a strict map of prelog $q$-PD triples. 
Then, there is a canonical isomorphism
\[
q\Omega_{(X, M_X) /(D, M_D)} \overset{\cong}{\longrightarrow}
q\Omega_{(X, M_X)_{D/I'}/(D, M_D)}. 
\]
\end{lem}

\begin{proof}
The canonical embedding $X_{D/I'}\hookrightarrow X$ induces a functor
\[
((X, M_X)_{D/I'} /(D, M_D))_{q\CRYS}\to ((X, M_X) /(D, M_D))_{q\CRYS}. 
\]
This determines the map, and we shall show that it is an isomorphism. 

We may assume $X=\Spf (R)$ with a smooth chart $M_D\to P$. 
We apply Construction \ref{compuing log q-crystalline cohomology} to $((R, P) /(D, M_D))_{q\CRYS}$ and obtain $F^{\bullet}$ computing $q\Omega_{(R, P) /(D, M_D)}$. It is clear from Construction \ref{compuing log q-crystalline cohomology} that $F^{\bullet}$ also computes $q\Omega_{(R, P)_{D/I'} /(D, M_D)}$
\end{proof}

Let $(D, I, M_D)$ be a prelog $q$-PD-triple and assume it is of rank $1$ or $(D, M_D)$ is a log ring. In particular, we have $\phi_{M_D}\colon M_D\to M_D$ and the target will be denoted by $\phi_* M_D$ instead. 
Observe that $(\phi_D, \phi_{M_D})$ induces $\psi_D\colon (D/I, M_D)\to (\phi_* D/[p]_q, \phi_* M_D)$.
Write $(X^{(1)}, M_X^{(1)})$ for the base change of $(X, M_X)$ along $\psi_D$. 
We have $\Prism_{(X^{(1)}, M_X^{(1)})/ (\phi_* D, \phi_* M_D)}$ relative to the log prism $(\phi_* D, ([p]_q ), \phi_* M_D)$. 

\begin{thm}[Log $q$-crystalline cohomology and log prismatic cohomology]\label{log q-crystalline vs log prismatic}
Assume that the mod $p$ fiber of $(X, M_X)$ is of Cartier type over $(D/(p,I), M_D)$. 
There exists a canonical isomorphism
\[
\Prism_{(X^{(1)}, M_X^{(1)})/ (\phi_* D, \phi_* M_D)}\overset{\cong}{\longrightarrow}
\phi_* q\Omega_{(X,M_X)/ (D, M_D)}
\]
of $E_{\infty}$-$\phi_* D$-algebras on $X_{\et}$. 
\end{thm}

\begin{rem}
By the base change theorem, $\Prism_{(X^{(1)}, M_X^{(1)})/ (\phi_* D, \phi_* M_D)}$ is isomorphic to the $(p, [p]_q)$-completed base change of $\Prism_{(\widetilde{X}, M_{\widetilde{X}})/ (D, M_D)}$ along $\phi_D$. 
\end{rem}

\begin{proof}
First we construct a functor
\[
((X, M_X)/(D, M_D))_{q\CRYS}\to ((X^{(1)}, M_X^{(1)})/(D, M_D))_{\Prism}. 
\]
Let $(E, J, M_E)^a$ be an object of $((X, M_X)/ (D, M_D))_{q\CRYS}$. There is a commutative diagram
\[
\begin{CD}
E/J @>>> \phi_* E/ [p]_q  \\
@AAA @AAA \\
D/I @>\psi_D >> \phi_* D/[p]_q , 
\end{CD}
\]
where the top map is induced by $\phi_E$, and this diagram induces a map of formal schemes $\Spf (\phi_* E/[p]_q)\to X^{(1)}$. Moreover, $\phi_{M_E}$ factors as
\[
M_E \to M_E^{(1)}\to M_E,
\]
where the first map is induced by $\phi_{M_D}$. 
The log prism $(\phi_* E, ([p]_q), M_E^{(1)})^a$ over $(\phi_* D, ([p]_q ), \phi_* M_D)$ can be naturally upgraded to an object of the prismatic site $((X^{(1)}, M_X^{(1)})/(\phi_* D, \phi_* M_D))_{\Prism}$.
This functor is cocontinuous and induces the desired comparison map, say $\alpha$. 

We shall show that $\alpha$ is an isomorphism. We may work locally and assume that $X=\Spf(R)$ with a smooth chart $M_D\to P$ of Cartier type.
In particular, $(R, P)$ has a smooth lift over $(D/(q-1), M_D)$ (even over $(D, M_D)$). The comparison map $\alpha$ is compatible with the change of $I$ under the isomorphism of Lemma \ref{Invariance under q-PD thickenings of the base} (the prismatic side does not change), so we may further assume that $I=(q-1)$, the smallest $q$-PD ideal. 

By \cite{BS}*{Lemma 16.5 (2)}, it suffices to show that $\alpha$ is an isomorphism after $(p, [p]_q)$-completed base change to $D/(q-1)$.  
Applying Lemma \ref{base change I} to 
\[
(\phi_* D, ([p]_q), \phi_* M_D)\to (\phi_* D/ (q-1), (p), \phi_* M_D),
\]
we obtain an isomorphism
\[
\Prism_{(R^{(1)}, P^{(1)})/ (\phi_* D, \phi_* M_D)}\widehat{\otimes}^L_{\phi_* D} \phi_* D/(q-1) \overset{\cong}{\longrightarrow} 
\Prism_{(\overline{R}^{(1)}, P^{(1)})/ (\phi_* D/(q-1), \phi_* M_D))}, 
\]
where $\overline{R}^{(1)}$ is the base change of $R^{(1)}$ to $\phi_* D/(p, q-1)$. 
On the other hand, (the proof of) Theorem \ref{crystalline comparison} and Theorem \ref{log q vs log crys} give identifications
\begin{align*}
\Prism_{(\overline{R}^{(1)}, P^{(1)})/ (\phi_* D/(q-1), \phi_* M_D))} &\overset{\cong}{\longrightarrow}
\phi_* R\Gamma_{\delta\CRYS}((X, M_X)/ (D/(q-1), M_D)), \\
&\overset{\cong}{\longleftarrow} \phi_* q\Omega_{(R,P)/ (D, M_D)}\widehat{\otimes}^L D/(q-1),  
\end{align*}
where the first isomorphism is constructed in exactly the same way as $\alpha$. 
Since the construction of $\alpha$ commutes with the base change $D\to D/(q-1)$, we conclude that $\alpha$ is an isomorphism.
\end{proof}

\subsection{Log \texorpdfstring{$q$}{q}-de Rham complexes}
From now on, we work locally and assume that $X=\Spf (R)$ is affine, $(X, M_X)$ admits a smooth lift over $(D, M_D)$, $M_D\to P$ is a chart that is integral and weakly finitely generated. We further assume that $D$ is \emph{flat} over $A$. We introduce logarithmic variants of \cite{BS}*{Construction 16.19, 16.20}. 

\begin{const}[Log $q$-de Rham complex]\label{log q-de Rham}
Let $M_D\oplus \bN^S$ be a monoid free over $M_D$ for some set $S$, and let $(E, M_E)$ be the $(p, [p]_q)$-completion of 
\[
D[\bN^S]=D \otimes_{\bZ_{(p)}[M_{D}]}\bZ_{(p)}[M_D\oplus \bN^S],
\]
regarded as a $\delta_{\log}$-ring over $(D, M_D)$ with $\delta_{\log}(\bN^S)=0$. 
As $D$ is flat over $A$, $E$ is also flat over $A$. 

For each $s\in S$, 
\[
X_s \mapsto q X_s, \quad X_t \mapsto X_t \quad (t\neq s), 
\]
defines an automorphism $\gamma_s$ of $(D [\bN^S], \bN^S)^a$ as a $\delta_{\log}$-ring over $(D, M_D)$ since $q$ is invertible in $A$. It extends to an automorphism of $E$ as it is congruent to the identity modulo a topologically nilpotent element $(q-1)X_s$. 
Since $q-1$ is a nonzerodivisor of $E$, we can define
\[
\nabla_{q,s}^{\log}:E\to E ; \quad f \mapsto \frac{\gamma_s (f)- f}{q-1}. 
\]
There exists a unique $D$-linear continuous map
\[
\nabla_q \colon E \to \Omega^1_{(E, M_E)/(D, M_D)}
\]
such that $\nabla_q (f) =\sum_{s\in S}\nabla_{q,s}^{\log}(f) d\log(X_s)$. 
The associated $(p, [p]_q)$-completed (cohomological) Koszul complex is the \emph{log $q$-de Rham complex} $q\Omega^*_{(E, M_E)/ (D, M_D)}$ of $(E, M_E)$ relative to $(D, M_D)$. 
This construction is functorial on $S$: every map $S\to S'$ induces a map 
\[
q\Omega^*_{(E, M_E)/ (D, M_D)}\to q\Omega^*_{(E', M_{E'})/ (D, M_D)}
\]
of complexes of $D$-modules. 

We need to generalize the construction of log $q$-de Rham complexes. 
Fix $S$ as above. First note that $\gamma_s$ above extends to an automorphism of 
\[
D \otimes_{\bZ_{(p)}[M_{D}]}\bZ_{(p)}[M_D^{\gp}\oplus \bZ^S],
\]
and it is congruent to the identity modulo $(q-1)$. 

Let $N\subset M_D^{\gp}\oplus \bZ^S$ be a submonoid containing $M_D$, and set
\begin{center}
$E_N=$ (the $(p, [p]_q)$-completion of $D \otimes_{\bZ_{(p)}[M_{D}]}\bZ_{(p)}[N]$), 
\end{center}
regarded as a $\delta_{\log}$-ring over $(D, M_D)$ by Proposition \ref{change of monoids}. 
For any $n\in N$, $\gamma_s (n)$ and $\nabla_{q,s}^{\log}(n)$ are defined inside $E_N$, and $\gamma_s $ induces an automorphism of $(E_N, N)^a$. 
Now, the construction of the log $q$-de Rham complex works in this setting, and we obtain $q\Omega_{(E_N, N)/(M_D, D)}^*$. 

This generalized construction is functorial on $(S, N)$: for every map $S\to S'$ and $N\to N'\subset (M_{D}\oplus\bN^{S'})^{\gp}$, we have a map 
\[
q\Omega^*_{(E_N, N)/ (D, M_D)}\to q\Omega^*_{(E_{N'}, N')/ (D, M_D)}.
\]
\end{const}

\begin{const}[Log $q$-de Rham complexes for log $q$-PD envelopes]\label{log q-de Rham for enveloeps}
Using the notation of Construction \ref{log q-de Rham}, we let $(E_N, N)\to (R, P)$ be a surjection from some $N$ and $(E_{N'}, N')$ its exactification. Applying Construction \ref{log q-de Rham} to $N'$, we obtain $\gamma_s$, $\nabla_{q}^{\log}$, and the log $q$-de Rham complex $q\Omega_{(E_{N'}, N')/ (D, M_D)}^*$. 

Let $(F, M_F)$ be the log $q$-PD envelope of $(E_N, N)\to (R, P)$, which is the same as the $q$-PD envelope of $(E_{N'}, N')\to (R, P)$. By the same argument of \cite{BS}*{Construction 16.20} (with $(q-1)X_s$ replaced by $(q-1)$), we can extend $\nabla_{q,s}^{\log}$ to $F$ and thus obtain the log $q$-de Rham complex 
\[
q\Omega_{(F, M_F)/ (D, M_D)}:
F \overset{\nabla^{\log}_{q}}{\longrightarrow} F\widehat{\otimes}_E \Omega^1_{(E,M_E)/ (D, M_D)}
\overset{\nabla^{\log}_{q}}{\longrightarrow} F\widehat{\otimes}_E \Omega^2_{(E,M_E)/ (D, M_D)}\to \cdots.
\]
The reduction $F/(q-1)$ is the classical de Rham complex of the $p$-completed log PD envelope  \cite{Beilinson}*{1.7}.   

This construction is functorial. If $(E_1, N_1)\to (R, P)$ is another surjection with $N_1 \subset M_D^{\gp}\oplus \bZ^{S_1}$, and $S\to S_1$ is any map compatible with two surjections, then we have a natural map 
\[
q\Omega_{(F, M_F)/ (D, M_D)}^* \to q\Omega_{(F_1, M_{F_1})/ (D, M_D)}^*
\]
with obvious notation. This functoriality is slightly nontrivial than previous ones as it involves the $q$-PD envelope, but easily follows from $[p]_q$-torsionfreeness of $F_1$ as in the last sentence of \cite{BS}*{Construction 16.20}. 
\end{const}

\begin{thm}[Log $q$-de Rham cohomology and log $q$-crystalline cohomology]\label{log q-de Rham vs log q-crystalline}
Let $(E, N)\to (R, P)$ be a surjection and $(F, M_F)$ the log $q$-PD envelope as in Construction \ref{log q-de Rham for enveloeps}. 
Then, there is a canonical isomorphism
\[
q\Omega_{(R, P)/ (D, M_D)} \cong q\Omega_{(F, M_F)/ (D, M_D)}^*. 
\]
This isomorphism is functorial on surjections $(E, N)\to (R, P)$. 
\end{thm}

\begin{proof}
The proof is entirely parallel to \cite{BS}*{Theorem 16.22}. To provide references, let us add some details. 

Let $(E^{\bullet}, M_E^{\bullet})$ be the \v{C}ech nerve of $(D, M_D)\to (E, N)$, and $(F^{\bullet}, M_F^{\bullet})$ its log $q$-PD envelopes. 
We shall apply Construction \ref{log q-de Rham for enveloeps} to each surjection $(E^n, M_{E^n})\to (R,P)$, and obtain the log $q$-de Rham complexes of $(F^{\bullet}, M_{\bullet})$ relative to $(D, M_D)$.  
Set $M^{\bullet, *}\coloneqq q\Omega_{F^{\bullet}, M_F^{\bullet}/ (D, M_D)}^*$; this is a cosimplicial complex. 
The proof is done once we check the following items:
\begin{enumerate}
    \item The $0$-th row $M^{\bullet, 0}$ computes $q\Omega_{(R, P)/ (D, M_D)}$. 
    \item The $0$-th column $M^{0,*}$ is exactly the log $q$-de Rham complex $q\Omega_{((F, M_F)/ (D, M_D)}^*$. 
    \item For $i>0$, the $i$-th row $M^{\bullet, i}$ is cosimplicially homotopy equivalent to $0$.  
    \item Any face map $F^i\to F^j$ induces an isomorphism $M^{i,*}\cong M^{j,*}$. 
\end{enumerate}
Remark \ref{computing log q-crystalline cohomology via smooth lift} says (1) holds, and (2) is trivial. 
For (3), the argument in the nonlog case \cite{Bhatt-deJong}*{2.15} works because the problem reduces to showing that $\Omega^1_{(E^{\bullet}, M_E^{\bullet})/ (D, M_D)}$ is cosimplicially homotopy equivalent to $0$; each module $\Omega^1_{(E^{\bullet}, M_E^{\bullet})/ (D, M_D)}$ is free over $E^{\bullet}$ with a basis obtained from $d\log (X_s)$, and their behavior under maps $[n]\to [m]$ are essentially the same as in the nonlog case. 
Finally, (4) reduces to the corresponding statement modulo $(q-1)$, which holds as both sides compute the same log crystalline cohomology \cite{Beilinson}*{(1.8.1)}; see also the last part of the proof of \emph{loc.cit.}.
\end{proof}

\begin{rem}[Frobenius is an isogeny]\label{Isogeny}
It is easy to see that the Frobenius on $q\Omega_{(R, P)/ (D, M_D)}$ is translated into 
\[
d\log (X_s) \mapsto [p]_q d\log (X_s)
\]
on the log $q$-de Rham complex $q\Omega_{(F, M_F)/ (D, M_D)}^*$; compare with \cite{Bhatt:lecture}*{XI, Remark 2.10}. 
In particular, the linearized Frobenius 
\[
\phi_D^* q\Omega_{(F, M_F)/ (D, M_D)}^* \to q\Omega_{(F, M_F)/ (D, M_D)}^* 
\]
factors over $\eta_{[p]_q} q\Omega_{(F, M_F)/ (D, M_D)}^*\to q\Omega_{(F, M_F)/ (D, M_D)}^*$. 

Assume $R$ is topologically of finite presentation over $D/I$ and $(X, M_X)$ has the mod $p$ fiber of Cartier type over $(\Spec (D/(p,I)), M_D)^a$. 
Then $q\Omega_{(R, P)/(D, M_D)}$ sits in $D^{[0, r]}(D)$, where $r$ denotes the rank of $\Omega^1_{(R, P)/ (D/I, M_D)}$. 
(This can be checked by reducing it to the case of the log crystalline cohomology or generalizing the discussion above on log $q$-de Rham complexes.)
Moreover, the linearized Frobenius induces an isomorphism
\[
\phi_D^* q\Omega_{(R, P)/ (D, M_D)} \cong L\eta_{[p]_q} q\Omega_{(R, P)/ (D, M_D)};
\]
this is reduced to the case of log crystalline cohomology, where the Cartier isomorphism implies the claim. 
Thus, we deduce that the linearized Frobenius of $q\Omega_{(R, P)/ (D, M_D)}$ admits an inverse up to $[p]_q^r$ in this case. 
\end{rem}

\section{Comparison with \texorpdfstring{$A\Omega$}{AOmega}}\label{section:comparison with AOmega}
We compare the log $q$-crystalline cohomology with the $A_{\Inf}$-cohomology in the semistable case \cite{CK}. Let $k$ be an algebraically closed field of characteristic $p$ and $C$ the completed algebraic closure of $W(k)[1/p]$. 

\begin{thm}\label{AOmega}
Let $X$ be a $p$-adic formal scheme over $\cO_C$ that is, \'etale locally, \'etale over
\[
\cO_C \langle t_0, \dots, t_r, t_r^{\pm 1}, \dots, t_d^{\pm 1}\rangle/(t_0 \cdots t_r - \pi)
\]
for some non-unit $\pi \in \cO_C$. 
Denote by $M_X$ its canonical log structure \cite{CK}*{1.6}.
There exists an isomorphism
\[
q\Omega_{(X, M_X)/ (A_{\Inf}, \cO_C^\flat \setminus\{0\})} \cong A\Omega_X
\]
in $D(X_{\et}, A_{\Inf})$ compatible with the Frobenius. 
\end{thm}

\begin{rem}
The isomorphism should be an isomorphism of $E_{\infty}$-$A_{\Inf}$-algebras as in \cite{BS}*{Theorem 17.2}, but we do not check it here. 
\end{rem}

\begin{rem}
As the mod $p$ fiber of a semistable formal scheme is of Cartier type, $q\Omega_{(X, M_X)/ (A_{\Inf}, \cO_C^\flat \setminus\{0\})}$ is isomorphic to the $(p, \mu)$-completed base change of $\Prism_{(X, M_X)/ (A_{\Inf}, \cO_C^\flat \setminus\{0\})}$ along $\phi_{A_{\Inf}}$ by Theorem \ref{log q-crystalline vs log prismatic} and base change. 
\end{rem}

\begin{proof}
We shall construct an isomorphism \'etale locally that is functorial (at least) with respect to \'etale maps, and we use (a slight modification of) the notation of \cite{CK}, especially 5.17. Fix a compatible system of rational powers of $p$ in $\cO_C$ as in \cite{CK}*{1.5}. 

We assume that $X=\Spf (R)$ is affine nonempty, that any intersection of irreducible components of $\Spec (R\otimes_{\cO_C}k)$ is nonempty and irreducible, and we have
\begin{itemize}
    \item $R_{\Sigma}^{\square}=\cO_C\langle \bZ^{\Sigma} \rangle$ for a finite set $\Sigma$ of units of $R$, 
    \item a nonempty finite set $\Lambda$ and, for each $\lambda\in \Lambda$,
    \[
    R_{\lambda}^{\square}\coloneqq \cO_C \langle \bN^{r_\lambda +1}\oplus \bZ^{d-r_{\lambda}}\rangle /(t_{\lambda_0}\cdots t_{\lambda, r_{\lambda}}- p^{q_\lambda}) \quad \textnormal{with} \quad q_\lambda \in \bQ_{>0},
    \]
    \item a closed immersion 
    \[
    X=\Spf (R)\to \Spf (R_\Sigma^\square) \times \prod_{\lambda\in\Lambda}\Spf(R_{\lambda}^\square)
    \]
    where $X\to \Spf (R_\Sigma^\square)$ is determined by $\Sigma$ itself, such that 
    $X\to \Spf (R_\Sigma^\square)$ is already a closed immersion and, for each $\lambda\in \Lambda$, $X=\Spf(R)\to \Spf (R_{\lambda}^\square)$ is \'etale.
\end{itemize}
We endow $\Spf (R_{\Sigma}^{\square})$, $\Spf (R_{\lambda}^{\square})$ with the natural log structure over $(\cO_C, \cO_C\setminus\{0\})$. 
Then, each map $X=\Spf(R)\to \Spf (R_{\lambda}^{\square})$ gives a smooth chart that is small in the sense of Definition \ref{small chart}; compare with \cite{Tsuji:purity}*{2.3}. 
Namely, $\Gamma (X, M_X)$ is generated by $\cO_C\setminus\{0\}$, $\bN^{r_\lambda +1}\oplus \bZ^{d-r_{\lambda}}$, and $\Gamma (X, \cO_X^\times)$. It also follows that $(X, M_X)$ is log-affine as in (the proof of) Lemma \ref{small neighborhood}. 

For each $\lambda\in \Lambda$, we use the following embedding
\[
\bN^{r_\lambda +1}\oplus \bZ^{d-r_{\lambda}} \to (\cO_C\setminus\{0\}) \oplus\bZ^{d};
e_0 \mapsto (p^{q_{\lambda}}, -1, \dots, -1), \quad
e_i \mapsto e_i \quad (i\neq 0) 
\]
when we form the log $q$-de Rham complexes; this embedding is consistent and compatible with the choice of $\delta_{\lambda, i}$ in \cite{CK}*{5.18, 5.19}. 

As in \cite{CK}*{5.18}, set
\[
\Delta_{\Sigma}=\bZ_p^{\Sigma}, \quad \Delta_{\lambda}\coloneqq \bZ_p^d, \quad
\Delta_{\Sigma, \Lambda}\coloneqq \Delta_{\Sigma}\times \prod_{\lambda\in \Lambda}\Delta_{\lambda}. 
\]
By extracting all $p$-power roots of coordinates, we obtain a perfectoid pro-(finite \'etale) $\Delta_{\Sigma, \Lambda}$-cover
\[
\textnormal{Spa}(R_{\Sigma, \Lambda, \infty}[1/p], R_{\Sigma, \Lambda, \infty})\to 
\textnormal{Spa}(R[1/p], R)
\]
giving a (cohomological) Koszul complex
\[
\eta_\mu (K_{\bA_{\Inf}(R_{\Sigma, \Lambda, \infty})}((\delta_{\sigma}-1)_{\sigma\in \Sigma}, (\delta_{\lambda,i}-1)_{\lambda, 1\leq i \leq d}); 
\]
see \emph{loc.cit.} for details. 
By \cite{CK}*{3.21, 4.6}, this Koszul complex computes 
\[
A\Omega_R\coloneqq R\Gamma (X_{\et}^{\textnormal{psh}}, A\Omega_X^{\textnormal{psh}})\cong R\Gamma (X_{\et}, A\Omega_X).
\]
Then,  
\[
\varinjlim_{\Sigma, \Lambda}
\eta_\mu (K_{\bA_{\Inf}(R_{\Sigma, \Lambda, \infty})}((\delta_{\sigma}-1)_{\sigma\in \Sigma}, (\delta_{\lambda,i}-1)_{\lambda, 1\leq i \leq d})
\]
and its term-wise $(p,\mu)$-completion both compute $A\Omega_R$. 

Recall the following lifts of $R_{\Sigma}^{\square}$, $R_{\lambda}^{\square}$:
\[
A(R_{\Sigma}^{\square})=A_{\Inf}\langle \bZ^{\Sigma}\rangle, \quad
A(R_{\lambda}^{\square})=A_{\Inf} \langle \bN^{r_\lambda +1}\oplus \bZ^{d-r_{\lambda}}\rangle /(X_{\lambda_0}\cdots X_{\lambda, r_{\lambda}}- [p^\flat]^{q_\lambda}). 
\]
Then, we have a surjection
\[
A_{\Sigma, \Lambda}^{\square}\coloneqq (A(R_{\Sigma}^{\square})\otimes\bigotimes_{\lambda\in\Lambda}A(R_{\lambda}^{\square}))^{\wedge}
\to R. 
\]
Moreover, if we regard $A_{\Sigma, \Lambda}^{\square}$ naturally as a $\delta_{\log}$-ring with 
\[
\delta_{\log}(\bZ^{\Sigma})=\delta_{\log}(\bN^{r_\lambda +1}\oplus \bZ^{d-r_{\lambda}})=0
\]
over $(A_{\Inf} , \cO_C^\flat \setminus\{0\})$, 
\[
(A^{\square}, M_{A^{\square}})=(\varinjlim_{\Sigma, \Lambda}(A_{\Sigma, \Lambda}^{\square}, M_{\Sigma, \Lambda}^{\square}))^{\wedge} \to (R, \Gamma (X, M_X))
\]
is a surjection from a $(p, \mu)$-completely ind-smooth $\delta_{\log}$-ring, where $M_{\Sigma, \Lambda}^{\square}$ denotes the prelog structure. 
Thus, if we write $(D^{\square}, M_{D^{\square}})$ for the $(p, \mu)$-completed log $q$-PD envelope, the following log $q$-de Rham complex 
\[
q\Omega^*_{(D^{\square}, M_{D^{\square}}))/(A_{\Inf}, \cO_C^\flat \setminus\{0\})}
\]
computes the log $q$-crystalline cohomology $q\Omega_{(R, \Gamma (X, M_X))/ (A_{\Inf}, \cO_C^\flat \setminus\{0\})}$. 

We shall construct a comparison map
\begin{multline*}
q\Omega^*_{(D^{\square}, M_{D^{\square}})/ (A_{\Inf}, \cO_C^\flat \setminus\{0\})}) \\
\to 
(\varinjlim_{\Sigma, \Lambda}
\eta_\mu (K_{\bA_{\Inf}(R_{\Sigma, \Lambda, \infty})}((\delta_{\sigma}-1)_{\sigma\in \Sigma}, (\delta_{\lambda,i}-1)_{\lambda, 1\leq i \leq d}))^{\wedge}.
\end{multline*}
It suffices to construct a natural map
\[
D^{\square}\coloneqq D_{\Sigma, \Lambda}^{\square}
\to
\bA_{\Inf}(R^{\square}_{\infty}),
\]
where $R^{\square}_{\infty}=(\varinjlim_{\Sigma, \Lambda}R_{\Sigma, \Lambda, \infty})^{\wedge}$,
that intertwines $\gamma_s$ (in the notation of Construction \ref{log q-de Rham}) and $\delta_{\sigma}, \delta_{\lambda,i}$ in \cite{CK}*{5.18, 5.19}. 

Set $N=(h^{\gp})^{-1}(\Gamma (X, M_X))$, where $h\colon M_{A^{\square}}\to \Gamma (X, M_X)$ denotes the natural map. 
There is a natural map $M_{A^{\square}}\to \bA_{\Inf}(R^{\square}_{\infty})$. 
The proof of \cite{CK}*{5.37} shows that this extends (uniquely) to $N$. 
Now, we get a commutative diagram
\[
\begin{CD}
(A^{\square},M_{A^{\square}}) @>>> (\bA_{\Inf}(R^{\square}_{\infty}), N) \\
@VVV @VVV \\
(R, \Gamma (X, M_X)) @>>> (R^{\square}_{\infty}, \Gamma (X, M_X)).
\end{CD}
\]
We see that the top map is a map of $\delta_{\log}$-rings,  
the bottom arrow is strict,  the left vertical map is surjective, and the right vertical map is exact surjective. Therefore, this commutative diagram induces a desired map $D^{\square}\to \bA_{\Inf}(R^{\square}_{\infty})$ in a unique way. 
The uniqueness also implies that it intertwines $\gamma_s$ and $\delta_{\sigma}, \delta_{\lambda,i}$ as these actions agree in the exactification of the left vertical map: argue as in the proof of \cite{BS}*{Theorem 17.2}. 

Finally, let us show that the comparison map is an isomorphism. 
In the rest of the proof, we regard the comparison map as a map of complexes of $\phi_* A_{\Inf}$-modules instead. 
Theorems \ref{Hodge-Tate}, \ref{log q-crystalline vs log prismatic}, and \cite{CK}*{4.11, 4.17} show that the Bockstein complexes with respect to $[p]_q$ obtained from the source and target of the comparison map both can be identified with the log de Rham complex 
\[
\Omega^*_{(R^{(1)}, \Gamma(X, M_X)^{(1)})/(\phi_* A_{\Inf}/[p]_q, \phi_* \cO_C^\flat \setminus\{0\})}
\]
as commutative differential graded $\phi_* A_{\Inf}/[p]_q$-algebras. 
(Here, $R$ is regarded as an $A_{\Inf}/(\xi)(\cong \cO_C)$-algebra, $R^{(1)}$ is the base change along $A_{\Inf}/(\xi) \cong \phi_* A_{\Inf}/([p]_q)$, and $\phi$ on $\cO_C^\flat$ is the $p$-th power.)
Two identifications are a priori possibly different under the comparison map. 
In the current setting, the Hodge-Tate comparison map is determined uniquely by the nonlog part of the Hodge-Tate comparison map given by the universality of the (nonlog) differentials since the prelog structure consists of nonzerodivisors. 
Therefore, arguing as in the proof of \cite{BS}*{Lemma 17.4}, we need only compare the following maps
to check that two identifications are the same:
\begin{align*}
\eta_{q\Omega}\colon R^{(1)} &\ \to \phi_* H^0 (q\Omega^*_{(D^{\square}, M_{D^{\square}})/ (A_{\Inf}, \cO_C^\flat \setminus\{0\})} /[p]_q), \\
\eta_{A\Omega}\colon R^{(1)} & \to \phi_* H^0 (A\Omega_X /[p]_q). 
\end{align*}
By \'etale localization, we are reduced to the case of 
\[
R=\cO_C\langle t_{\lambda, 0}, \dots, t_{\lambda, r_{\lambda}}, \dots, t^{\pm 1}_{\lambda, d}\rangle/(t_{\lambda, 0}\cdots t_{\lambda, r_{\lambda}}-p^{q_{\lambda}})
\]
for some non-unit $\pi\in \cO_C$, and this case can be checked by a direct computation.
(Strictly speaking, this $R$ does not satisfy our assumption at the beginning of the proof, but we have a simple comparison map
\[
q\Omega^*_{(A(R_{\lambda}^{\square}), M_{\lambda}^{\square})/(A_{\Inf}, \cO_C^\flat \setminus\{0\}))}\to
\eta_{\mu} K_{\bA_{\Inf}(R_{\lambda,\infty}^{\square})}((\delta_{\lambda,i}-1)_{1\leq i \leq d}),
\]
where $M_{\lambda}^{\square}$ is a natural prelog structure, compatible with the previous comparison map, and we may use it.)
\end{proof}

The above proof shows that the comparison isomorphism is compatible with the Hodge-Tate comparison.  

\subsection{More compatibility}
The main result of \cite{CK} is the absolute crystalline comparison \cite{CK}*{5.4}, which identifies $A\Omega\widehat{\otimes}^L_{A_{\Inf}}A_{\crys}$ with the pushforward of the structure sheaf of the absolute log crystalline site. There is a corresponding identification for the log $q$-crystalline cohomology as follows. 

Let $(X, M_X)=(\Spf (R), M_X)$ be a smooth log $p$-adic formal scheme of Cartier type over $(\cO_C, \cO_C \setminus\{0\})$ with a small smooth chart $P\to \Gamma (X, M_X)$. 
Combining Theorems \ref{log q-crystalline vs log prismatic}, \ref{crystalline comparison}, and the base change of log prismatic cohomology, we have an isomorphism
\[
q\Omega_{((R, P)/(A_{\Inf}, \cO_C^\flat \setminus\{0\})}\widehat{\otimes}^L_{A_{\Inf}}A_{\crys} \cong 
R\Gamma_{\crys}((X, M_X)_{\cO_C/p} /(A_{\crys}, \cO_C^\flat \setminus\{0\})). 
\]

\begin{rem}\label{q-crys tensored with Acrys}
Let us write down the map
\[
q\Omega_{((R, P)/(A_{\Inf}, \cO_C^\flat \setminus\{0\}))} \to
R\Gamma_{\crys}((R, P)_{\cO_C/p}/(A_{\crys}, \cO_C^\flat \setminus\{0\})). 
\]
Using the notation of Construction \ref{computing log q-crystalline cohomology via smooth lift}, $F^{\bullet}$ computes $q\Omega_{((R, P)/(A_{\Inf}, \cO_C^\flat \setminus\{0\}))}$. Each $F^{\bullet}$ is obtained as the log $q$-PD envelope, and it is identified with the log PD envelope after the $p$-completed base change to $A_{\crys}$ since $([p]_q)=(p)$ in $A_{\crys}$. (Recall the formula for these envelopes \cite{BS}*{Lemma 16.10, Corollary 2.38}.) 
The compatibility with base change of the log prismatic cohomology can be seen by analyzing the proofs of Theorems \ref{crystalline comparison} and \ref{log q-crystalline vs log prismatic}. 

We further translate it in terms of log $q$-de Rham complexes via Theorem \ref{log q-de Rham vs log q-crystalline}. To do so, we need a map from the log $q$-de Rham complex to the corresponding de Rham complex; this map is given by the formula appearing in \cite{BMS1}*{12.5}, \cite{CK}*{(5.15.2)} relating the $q$-derivative and the usual derivative when we are working over $A_{\crys}$, but now the map involved is in the opposite direction and the coefficients of the source (resp. target) are log $q$-PD envelopes (resp. log PD envelopes)\footnote{Technically speaking, it is only clear that this map is a map of complexes modulo $\mu^\infty$-torsion. So, one has to kill possible $\mu^{\infty}$-torsion of log PD envelopes. This ambiguity does not matter in Theorem \ref{comparison with BdR^+-cohomology} below.}. 
Moreover, one sees that this map gives a map from the cosimplicial complex $M^{\bullet, *}$ in the proof of Theorem \ref{log q-de Rham vs log q-crystalline} to a similar cosimplicial complex giving an identification of the log crystalline cohomology. 
\end{rem}

There are technical issues to compare this identification with \cite{CK}*{5.4}\footnote{In the nonlog case, the comparison is possible indirectly by using the uniqueness result of Li-Liu \cite{Li-Liu}*{Theorem 3.13} after deriving the identification to obtain an automorphism of the derived de Rham cohomology on quasisyntomic $\cO_C$-algebras; see \cite{Li-Liu}*{Corollary 3.15}. The same would probably work in a general log case once a suitable theory of derived log prismatic cohomology is developed.}. For instance, \cite{CK} uses ``truncated'' PD envelopes $D_{\Sigma, \Lambda}^{(m)}$ obtained from a ``non-completed" PD envelope $D_{j_0}$. Fortunately, we have the following analogue of \cite{CK}*{6.8}, which is strong enough to deduce the conjecture $C_{\textnormal{st}}$ in this setting without using \cite{CK}*{5.4}:

\begin{thm}\label{comparison with BdR^+-cohomology}
Let $X$ be as in Theorem \ref{AOmega}. If $X$ is proper over $\cO_C$, there is a commutative diagram
\[
\xymatrixcolsep{0.03in}
\xymatrix{
R\Gamma_{\crys}((X, M_X)/ (A_{\crys}, \cO_C^\flat \setminus\{0\})) \ar[r] &
R\Gamma_{\crys}(X^{\textnormal{ad}}_C/B_{\textnormal{dR}}^+) \ar[d]_{\cong} \\
R\Gamma_{q\CRYS}((X, M_X)/ (A_{\crys}, \cO_C^\flat \setminus\{0\}))\otimes^L_{A_{\Inf}} A_{\crys} \ar[d]_{\cong} \ar[u]^{\cong}&
R\Gamma_{\crys}(X^{\textnormal{ad}}_C, \bZ_p)\otimes^L_{\bZ_p} B_{\textnormal{dR}}^+ \ar[d]_{\cong} \\
R\Gamma (X_{\et}, A\Omega_X)\otimes^L_{A_{\Inf}} A_{\crys} \ar[r] &
R\Gamma_{\et}(X^{\textnormal{ad}}_C, \bA_{\Inf, X^{\textnormal{ad}}_C})\otimes^L_{A_{\Inf}}B_{\textnormal{dR}}^+,
}
\]
where left vertical arrows are given by Theorem \ref{AOmega} and Remark \ref{q-crys tensored with Acrys}, and other maps are the same as in \cite{CK}*{6.8}.
\end{thm}

\begin{proof}
We basically follow the proof of \cite{CK}*{6.8}, and we work locally.  
A description of the composite of the right vertical arrows are given in \emph{loc.cit.}, and in particular it involves the formula of the type of \cite{CK}*{(5.15.2)}.  
On the other hand, the isomorphism of Remark \ref{q-crys tensored with Acrys} completely identifies each term of the log $q$-de Rham complex tensored with $A_{\crys}$ and the log de Rham complex, while differentials are related by the formula \cite{CK}*{(5.15.2)} again, but in the opposite direction as explained in Remark \ref{q-crys tensored with Acrys}. 
The isomorphism of Theorem \ref{AOmega} is simply given by the embedding $D^{\square}\to A^{\square}$ in the notation there.
Thus, as in the proof of \cite{CK}*{6.8}, we are reduced to checking the commutativity of the diagram (6.8.4) in \cite{CK}.  
(More precisely, we need the commutativity after taking colimits and completions, but it is enough to check the commutativity beforehand.)
This commutativity is checked in \emph{loc.cit.}. 
\end{proof}

\appendix
\section{Complements on logarithmic geometry}
We discuss some geometry of log formal schemes whose log structures are not necessarily fine.
For log schemes, \cite{Beilinson} and \cite{ALPT} already contain some useful results in such generality. Other general references are \cites{Kato:log, Ogus} for log schemes, \cite{Shiho} for log $p$-adic formal schemes, and \cite{GR} for log topoi.
Some basic definitions are omitted here. 

Every monoid is assumed to be commutative. 
All our log structures are defined on the small \'etale sites. 

\subsection{Examples}
Suppose $(A, M_A)$ is a prelog ring, and $A$ is classically $I$-complete for a finitely generated ideal $I$. 
We do \emph{not} assume $M_A$ or its associated log structure is fine. 
We sometimes write $(\Spf (A), M_A)^a$ for the corresponding log formal scheme. 

We study log $I$-adic formal schemes over $(A, M_A)$, equivalently over $(\Spf (A), M_A)^a$. 
If $I=0$, they are just log schemes over $(A, M_A)$.

Let us illustrate some examples first; all examples here are integral quasi-coherent, and smooth in the sense of Definition \ref{smooth} below. 

\begin{enumerate}
    \item (Normal crossing divisors) Suppose $X$ is smooth over $A$. Let $D\subset X$ be a relative normal crossing divisor, $j\colon U=X\setminus D \to X$ an open immersion. The sheaf $j_* \cO_U^{\times}\cap \cO_X$ defines a log structure on $X$ over $A$ with the trivial log structure. 
    \item (Generalized semistable formal schemes) Let $\cO$ be a complete discrete valuation ring of mixed characteristic $(0,p)$ with $p>0$. Suppose $(A, M_A)=(\cO, \cO\setminus\{0\})$. Let $(Y, M_Y)$ be a smooth vertical fine saturated log scheme over $(A, M_A)$; vertical means that $M_Y$ is trivial on the generic fiber. In particular, $Y$ is flat over $A$. It is (log) regular by \cite{GR}*{12.5.44}, and $M_Y$ is given by (the sheafification of) $(\cO_{Y}[1/p])^\times \cap \cO_{Y}$ \cite{GR}*{12.5.54}. 
    The log structure $M_X$ of the $p$-adic completion $(X, M_X)$ of $(Y, M_Y)$ is also given by $(\cO_{X}[1/p])^\times \cap \cO_{X}$; cf.~\cite{CK}*{1.6.3}. This includes the case of semistable reduction. 

    Now consider only those $(Y, M_Y)$ with reduced special fibers (equivalently, the special fiber is of Cartier type \cite{Tsuji:saturated}). 
    This class is stable under base change: let $\cO\subset \cO'$ be a finite extension of complete discrete valuation rings. Set $(A', M_{A'})=(\cO', \cO'\setminus\{0\})$. The base change $(Y', M_{Y'})$ of $(Y, M_Y)$ to $(A', M_{A'})$ is clearly a smooth vertical fine log $p$-adic formal scheme over $(A', M_{A'})$. Moreover, $(Y', M_{Y'})$ is saturated with the reduced special fiber by \cite{Tsuji:saturated}*{II.4.2, II.2.11, II.2.1}. In particular, its log structure is given by $(\cO_{Y'}[1/p])^\times \cap \cO_{Y'}$; compare with \cite{CK}*{1.6.1}. 
    \item (Over $\cO_C$) Let $C$ be the $p$-adic completion of the algebraic closure of $\cO[1/p]$. Any log $p$-adic formal scheme $(X, M_X)$ that is \'etale locally isomorphic to the $p$-adic completion of the base change to $(\cO_C, \cO_C \setminus\{0\})$ of a log scheme of the type considered in (2) with the reduced special fiber is a smooth saturated log $p$-adic formal scheme over $(\cO_C, \cO_C \setminus\{0\})$ with the reduced special fiber. 
    The log structure $M_X$ is given by $(\cO_X[1/p])^\times \cap \cO_X$; we refer the reader to \cite{CK} again. 
\end{enumerate}

Let us make some remarks on $\cO \setminus\{0\}$, $\cO'\setminus\{0\}$, and $\cO_C \setminus\{0\}$. For a general valuation ring $\cO$, the multiplicative monoid $\cO \setminus\{0\}$ is \emph{valuative}. Recall that a monoid $M$ is valuative if it is integral and, for $m\in M^{\gp}$, either $m\in M$ or $m^{-1}\in M$ holds. If $M$ is valuative, any map $M\to N$ of monoids is integral \cite{Ogus}*{I.4.6.3.5}, thus the associated map $\bZ[M]\to \bZ[N]$ is flat if $M\to N$ is injective. (These are used in (2) above.)

The monoid $\cO \setminus\{0\}$ has at most $n$ nonempty prime ideals if and only if the ordered group $(\textnormal{Frac} \cO)^\times/ \cO^\times$ is embedded into $\mathbb{R}^n$, i.e., the valuation has height $\leq n$. 
As in the above examples, we often only consider the case of height one. 
For an arbitrary valuation ring $\cO$ of height one, there is a polystable reduction conjecture \cite{ALPT}*{1.2.6} generalizing the semistable reduction conjecture. 
(The assumption of height one is perhaps superfluous.)
This will give rise to another important case to study. 

\subsection{Relative coherence and small charts}
We recall/introduce some finiteness conditions on charts. We also study the global sections of a log structure on a ``small'' neighborhood of a geometric point. 
 
A \emph{chart} of a log formal scheme $(X, M_X)$ is a map $P \to \Gamma (X, M_X)$ of monoids inducing an isomorphism $\underline{P}^{a}\cong M_X$ of log structures. For $(X, M_X)$ over $(A, M_A)$, a chart over $M_A$ is a chart $P$ with a map $M_A \to P$ compatible with $M_A\to \Gamma (X, M_X)$. 
If $X$ is affine and the identity map of $\Gamma (X, M_X)$ is a chart of $M_X$, we say that $(X, M_X)$ is \emph{log-affine}. 

\begin{rem}
In a previous version, it is claimed that $(X, M_X)$ with affine $X$ is log-affine if (and only if) $(X, M_X)$ has a chart. However, as pointed out by Takeshi Tsuji, this is false. See Lemma \ref{small chart} for a sufficient condition. 
\end{rem}

The category of (pre)log structures admits colimits; cf. \cite{Ogus}*{III.2.1.1}.
Note that the colimit of maps $\{P_i \to \Gamma (X, \cO_X)\}_{i \in I}$ gives rise to a chart of the colimit of the associated log structures since the functor sending a prelog structure to the associated log structure is a left adjoint of the forgetful functor. 

A chart $P\to \Gamma (X, M_X)$ is \emph{exact} \cite{Ogus}*{II.2.3.1.1} at a geometric point $\overline{x}$ of $X$ if $P\to M_{X.\overline{x}}$ is exact. We say that the chart is exact if it is exact at every geometric point.

We shall introduce finiteness conditions for charts of morphisms of log formal schemes. 
A monoid \emph{finitely generated} over a monoid $M$ is a monoid $N$ with a map $M \to N$ such that there exist finitely many elements of $N$ that, together with the image of $M$, generate $N$.  
A map $M \to N$ of monoids is \emph{weakly finitely generated} if $N/N^\times$ is finitely generated over $M$. 
\begin{defn}
A chart $P \to \Gamma (X, M_X)$ (resp. over $M_{A}$) is \emph{coherent} (resp. \emph{relatively coherent}) if $P$ is finitely generated (resp. finitely generated over $M_A$). 
\end{defn}

Every log structure is often assumed to be fine in the literature, so coherent charts are used there. Their base change to a non-fine case gives an example of relatively coherent one. 

Any finitely generated monoid is known to be finitely presented. We need some variant. 

\begin{defn}
We say that an integral monoid $P$ over an integral monoid $M$ is \emph{finitely presented as an integral monoid} if there exists a monoid $P_0$ that is finitely presented over $M$ and $P_0^{\textnormal{int}}=P$.  
More precisely, $P_0$ is finitely generated over $M$, and for any surjection $M\oplus \bN^r \to P_0$ over $M$, the equalizer of two projections $(M\oplus \bN^r) \oplus_M (M\oplus \bN^r) \rightrightarrows P_0$ is finitely generated over $M$, and the image of $P_0$ in $P_0^{\gp}$ is isomorphic to $P$. 

Assume $M_A$ is integral. An integral chart $P \to \Gamma (X, M_X)$ over $M_{A}$ is \emph{relatively fine} if $P$ is finitely presented over $M_A$ as an integral monoid. 
\end{defn}

Being finitely presented as integral monoids are stable under pushouts $M \to M'$ in the category of integral monoids. 

Let us record the following lemma that supplies plenty of examples and is useful in the study of smooth log formal schemes. 

\begin{lem}[\cite{ALPT}*{2.1.7}]\label{ALPT finite presentation lemma}
Let $M \hookrightarrow P$ be an injection between integral monoids and assume $P$ is finitely generated over $M$. Then, $P$ is finitely presented over $M$ as an integral monoid. 
\end{lem}

Now we discuss some technical results relying on finite presentation. 
In the case of coherent log schemes, this can be found in \cite{Illusie}*{I.3}. 
A modern reference is \cite{Ogus}*{II.2.2}. 

\begin{lem}\label{Hom stalk}
Assume $M_A$ is integral and $P$ is finitely presented over $M_A$ as an integral monoid and $M_X$ is a sheaf of integral monoids on a $I$-adic formal scheme $X$ over $A$ equipped with a map from the constant sheaf $\underline{M_A}$. Then, for any geometric point $\overline{x}$ of $X$, the natural map
\[
\underline{\textnormal{Hom}}_{\underline{M_A}}(\underline{P}, M_X)_{\overline{x}} \to \textnormal{Hom}_{M_A}(P, M_{\overline{x}})
\]
is an isomorphism. 
\end{lem}

\begin{proof}
It reduces to the case of $P=M_A \oplus \bN$, and this case is easy to check. 
\end{proof}

\begin{cor}\label{Illusie lemma}
Assume $M_A$ is integral and let $M'_A\to \Gamma (\Spf (A), M_{A, \Spf (A)}^a)$ be another integral chart of $(\Spf (A), M_A)^a$. 
Let $X$ be an $I$-adic formal scheme over $A$ and let $M_X$ (resp. $M'_X$) be an integral log structure on $X$ over $(\Spf (A), M_A)^a$. 
Fix a geometric point $\overline{x}$ of $X$. 
\begin{enumerate}
\item Assume $M_X$ admits a relatively fine chart $M_A \to P$. 
If two maps of log structures $f_1, f_2 \colon M_X \to M'_X$ coincide at $\overline{x}$, then they coincide on a \'etale neighborhood of $\overline{x}$. 
\item Let $f\colon M_X \to M'_X$ be a map of log structures over $(\Spf (A), M_A)^a$. Assume $M_X$ (resp. $M'_X$) admits a relatively fine chart $M_A \to P$ (resp. $M'_A \to P'$). 
If $f$ induces an isomorphism at $\overline{x}$, then $f$ is an isomorphism on some \'etale neighborhood of $\overline{x}$. 
\end{enumerate}
\end{cor}

\begin{proof}
(1) Using Lemma \ref{Hom stalk}, one can argue as in \cite{Illusie}*{I.3.7 (a)}. 

(2) We argue as in \cite{Illusie}*{I.3.7 (b)} with Lemma \ref{Hom stalk} and (1). Let us include some details. Let $g_{\overline{x}}\colon M'_{X, \overline{x}}\to M_{X, \overline{x}}$ denote the inverse of $f_{\overline{x}}$, which is automatically a map over $M'_A$. By Lemma \ref{Hom stalk}, it extends to $g\colon M'_X\to M_X$ over a some \'etale neighborhood of $\overline{x}$. It suffices to show that $g\circ f=\id, f\circ g=\id$ after passing to another \'etale neighborhood of $\overline{x}$, and this follows from (1). 
\end{proof}

We have the following generalization of \cite{Kato:log}*{2.10}. 

\begin{cor}\label{constructing charts}
Assume $M_A$ is integral and $(X, M_X)$ over $(A, M_A)$ admits a relatively fine chart $M_A \to P$. Let $\overline{x}$ be a geometric point of $X$, $M'_A \to A$ another chart of $(\Spf (A), M_A)^a$, $G$ an abelian group that contains $(M'_A)^{\gp}$ and is finitely generated over $(M'_A)^{\gp}$. 
Let $h\colon G \to M^{\gp}_{X,\overline{x}}$ be a homomorphism that is compatible with $M'_A \to M_{X, \overline{x}}$ and induces a surjection $G\to M^{\gp}_{X,\overline{x}}/ \cO^{\times}_{X, \overline{x}}$. 
Define $P'\subset G$ to be inverse image of $M_{X, x}/ \cO^{\times}_{X, \overline{x}}\subset M^{\gp}_{X,\overline{x}}/ \cO^{\times}_{X, \overline{x}}$. 

If the kernel of $(M'_A)^{\gp} \to M^{\gp}_{X,\overline{x}}/ \cO^{\times}_{X, \overline{x}}$ is finitely generated over $(M'_A)^\times$, then $P'$ is finitely presented over $M'_A$ as an integral monoid and the natural map $P'\to M_{\overline{x}}$ extends to a chart of $M_X$ on some \'etale neighborhood of $\overline{x}$, and this chart is exact at $\overline{x}$. 
\end{cor}

\begin{proof}
By assumptions, the kernel of $G\to M^{\gp}_{X,\overline{x}}/ \cO^{\times}_{X, \overline{x}}$ is finitely generated over $(M'_A)^\times$. As $P'$ contains $M'_A$, we deduce that $P'$ is finitely presented over $M'_A$ as an integral monoid by Lemma \ref{ALPT finite presentation lemma}. Now, $P'\to M_{\overline{x}}$ extends to a map $P'\to \Gamma (U, M_X)$ for some \'etale neighborhood $U$ of $\overline{x}$. 
It induces an isomorphism of the associated log structures after shrinking $U$ by Lemma \ref{Illusie lemma}. 
So, $P'\to \Gamma (U, M_X)$ induces a chart of $M_X$. 
The exactness at $\overline{x}$ is clear. 
\end{proof}

Under some assumption, the identity of $\Gamma (X, M_X)$ is a chart, but this chart is not relatively coherent in general. Instead let us give a sufficient condition for this chart $M_A \to \Gamma (X, M_X)$ to be weakly finitely generated. 

\begin{defn}\label{small chart}
Let $(X, M_X)$ be a log $I$-adic formal scheme over $(A, M_A)$.  
A relatively coherent chart $P \to \Gamma (X, M_X)$ over $M_{A}$ is \emph{small} if the composite
\[
P \to \Gamma (X, M_X) \to \Gamma (X, M_X) / \Gamma(X, \cO_X^\times)
\]
is surjective. In particular, $\Gamma (X, M_X)$ is weakly finitely generated over $M_A$ in this case.
\end{defn}

\begin{lem}\label{small neighborhood}
Let $(X, M_X)$ be a log $I$-adic formal scheme over $(A, M_{A})$ with a relatively coherent chart $P\to \Gamma (X, M_X)$. Suppose the underlying topological space of $X$ is locally noetherian, and (at least) one of the following conditions holds:
\begin{enumerate}
    \item $\Spec (M_A)$ is finite, or
    \item the image of $M_A\setminus M_A^\times$ in $A$ is contained in $\sqrt{I}$. 
\end{enumerate} 
Then, for any $x\in X$, there exists a Zariski open neighborhood $U$ of $x$ such that, for any open neighborhood $V\subset U$ of $x$, $P\to \Gamma (V, M_X)$ is a small chart. 
In this case, for $V$ affine, $(V, \Gamma (V, M_X))$ is log-affine.  
\end{lem}

\begin{proof}
Case (1): the proof is the same as \cite{Tsuji}*{1.3.3}. Note that $\Spec (P)$ is finite if $P$ is finitely generated over $M_{A}$. 

Case (2): observe that 
\begin{itemize}
    \item for any geometric point $\overline{x}$ of $X$, the inverse image of $\mathfrak{m}_{X, \overline{x}}\subset \cO_{X,\overline{x}}$ in $P$ contains the image of $M_A\setminus M_A^\times\to P$ by assumption, and
    \item there are only finitely many prime ideals of $P$ that contain the image of $M_A\setminus M_A^\times \to P$.  
\end{itemize} 
Only using prime ideals of $P$ containing the image of $M_A\setminus M_A^\times \to P$, the argument of \cite{Tsuji}*{1.3.3} again works. 

The final claim follows by arguing as in $(2)\Rightarrow (3)$ of \cite{Tsuji}*{1.3.2}. 
\end{proof}

In the above situation, the tautological chart $\Gamma (X, M_X) \to \Gamma (X, M_X)$ can be controlled if the small chart is exact at some geometric point:

\begin{lem}\label{chart and global section II}
Suppose that there exists a small chart $P \to \Gamma (X, M_X)$ over $M_A$. 
If $M_A$ is integral and $M_A\to P$ is integral (resp. exact) and $P$ is exact at a geometric point $\overline{x}$ of $X$, then $M_A\to \Gamma (X, M_X)$ is integral (resp. exact). 
\end{lem}

\begin{proof}
The claim is equivalent to the integrality (resp. exactness) of 
\[
M_A / M_A^\times \to \Gamma (X, M_X)/ \Gamma (X, \cO_X^\times),
\]
and it suffices to show that $P/P^\times \to \Gamma (X, M_X)/ \Gamma (X, \cO_X^\times)$ is an isomorphism. 
By the assumption it is surjective. 
The composite
\[
P/P^\times \to \Gamma (X, M_X)/ \Gamma (X, \cO_X^\times) \to M_{X, \overline{x}}/M_{X, \overline{x}} 
\]
is an isomorphism since the chart $P\to \Gamma (X, M_X)$ is exact at $\overline{x}$. So, the first map is injective, hence an isomorphism. 
\end{proof}

\subsection{Smooth log formal schemes}

Let us introduce a notion of smooth log formal schemes in our favorite generality. 
Compared to Kato's original definition \cite{Kato:log}*{3.3}, there are two main differences: we weaken the finiteness condition, and we require the integrality so that the underlying morphism of (formal) schemes is flat. 

\begin{defn}\label{smooth}
Suppose $M_A$ is integral. 
A log $I$-adic formal scheme $(X,M_X)$ over $(A, M_A)$ is \emph{smooth} (resp. \emph{\'etale}) if, \'etale locally on $X$, there exists relatively coherent chart $P \to \Gamma (X, M_X)$ such that the following conditions holds:
\begin{itemize}
    \item $P$ is integral and $M_A\to P$ is integral, 
    \item the map $M_A^{\gp}\to P^{\gp}$ is injective and the torsion part of its cokernel (resp. the whole cokernel) is a finite group whose order is invertible in $X_{A/I}$, and
    \item the induced map
\[
X \to \Spf (A \langle P\rangle)\times_{\Spf (A \langle M_A\rangle)} \Spf(A)
\]
is \'etale.  
\end{itemize}
A \emph{smooth chart} (resp. \emph{\'etale chart}) is a chart satisfying the above conditions. 
If we drop the condition that $M_A \to P$ is integral, we say that $(X, M_X)$ is \emph{Kato-smooth} (resp. \emph{Kato-\'etale}) over $(A, M_A)$, and call $M_A \to P$ a \emph{Kato-smooth chart} (resp. \emph{Kato-\'etale chart}). 
\end{defn}

In particular, $M_X$ is an integral log structure on $X$.
By Lemma \ref{ALPT finite presentation lemma}, every smooth chart is relatively fine. 

\begin{rem}
Let us make initial remarks: 
\begin{enumerate}
    \item We do not know if $\bZ[P]$ is of finite presentation over $\bZ[M_A]$. Therefore, $X$ is possibly not locally of finite presentation over $A$. It is perhaps more natural to assume $P$ is finitely presented over $M_A$, cf. Proposition \ref{descent to the fine case}. 
    \item Recall that $M_A \to P$ is automatically integral if $M_A$ is valuative, and this is the case for our main examples. 
    \item Smooth/\'etale log formal schemes are stable under base change along arbitrary maps $(A, M_A)\to (A', M_{A'})$ of classically $I$-complete prelog rings with integral monoids. In particular, a smooth/\'etale chart exists for the tautological chart $M_A^a \to A$ if $(X, M_X)$ is smooth/\'etale over $(A, M_A)$. 
    For Kato-smoothness and Kato-\'etaleness, the same holds if one takes base changes in the category of integral log structures/monoids.  
    \item The above definition depends on the chart $M_A\to A$. It is more canonical if $M_A$ coincides with the global section of the associated log structure, but we want to allow more flexibility. 
    For smoothness and \'etaleness, stability under base change is one weak version of independence. We show another independence result in Proposition \ref{smoothness is independent of charts} below.  
\end{enumerate}
\end{rem}

\begin{rem}[Relation to Kato's smoothness]
We briefly explain how the above definition is related to Kato's original definition. 
\begin{enumerate}
    \item If $(X, M_X)$ is Kato-smooth over $(A, M_A)$, the reduction $(X, M_X)_{A/I^n}$, for any integer $n\geq 1$, is an integral log scheme over $(A/I^n, M_A)$ that is \emph{formally smooth} over $(A/I^n, M_A)$ in the sense of \cite{Ogus}*{IV.3.1.1} (cf. the proof of \cite{Kato:log}*{3.4}). 
    See also Remarks \ref{lifting property}, \ref{strong lifting property} for stronger lifting properties which we mainly use. 
    \item Assuming $X$ is locally of finite presentation over $A$ and $M_A, M_X$ are fine, Kato defined smoothness \cite{Kato:log}*{3.3} using formally smoothness. 
(In particular, it does not depend on the choice of $M_A\to A$ as long as $M_A$ is fine.)
Furthermore, he showed that it is equivalent to Kato-smoothness we defined above \cite{Kato:log}*{3.5}. 
(Kato-smoothness also implies that $X$ is locally of finite presentation in the setting of fine log schemes.)
(These are stated for fine log schemes, but can be adapted to fine log formal schemes. See also \cite{Shiho}*{2.2.8}.)

Moreover, $f$ is smooth/\'etale in our sense if and only if $f$ is integral and smooth/\'etale in the sense of Kato; cf. Proposition \ref{smoothness is independent of charts}. 
    \item Let $M_{A,0}\subset M_A$ be a fine submonoid. If $(X, M_X)$ is obtained as the base change of a fine log $I$-adic formal scheme $(X, M_{X,0})$ that is smooth and integral over $(A, M_{A,0})$ in the sense of Kato, then $(X, M_X)$ is smooth in our sense. 
    Similarly, working inside the category of integral log structures, the base change of $(X, M_{X,0})$ smooth over $(A, M_{A,0})$ in the sense of Kato is Kato-smooth. 
The converse holds under an additional assumption; see Proposition \ref{descent to the fine case}. 
\end{enumerate}
\end{rem}

\begin{rem}[Lifting property]\label{lifting property}
If $(X, M_X)$ is Kato-smooth over $(A, M_A)$, the sheaves of ($I$-completed) log-differentials $\Omega^i_{(X, M_X)/(A, M_A)}$ are locally free $\cO_X$-modules of finite rank. 
This implies the following lifting property of $(X, M_X)_{A/I^n}$ that is stronger than the formally smoothness. 
If $(S, M_S)\to (T, M_T)$ is an exact closed immersion of affine integral log schemes over $(A/I^n, M_A)$ defined by a \emph{nilpotent} ideal, i.e., a log thickening of finite order \cite{Ogus}*{IV.2.1.1}, then, for every commutative diagram
\[
\begin{CD}
(S, M_S)@>>> (T, M_T) \\
@VVV @VVV \\
(X, M_X)@>>> (\Spec (A/I^n), M_A)^a, 
\end{CD}
\]
there exists $(T, M_T)\to (X, M_X)$ making the resulting diagram commutative. 
This holds because the argument of \cite{Ogus}*{IV.3.1.4.2} works in this setting. 

From now on, the formally smoothness also refers to this lifting property. 
(The original definition of formally smoothness only implies $(T, M_T)\to (X, M_X)$ exists locally on $T$.) 
\end{rem}

\begin{prop}\label{descent to the fine case}
Assume $M_A$ is integral, and let $(X, M_X)$ be a log $I$-adic formal scheme over $(A, M_A)$. 
If $(X, M_X)$ admits a Kato-smooth chart $M_A\to P$, then there exists a fine submonoid $M_{A, 0}\subset M_A$ and a fine log scheme $(X, M_{X, 0})$ with a chart $M_{A, 0}\to P_0$ such that
\begin{itemize}
\item $(X, M_{X, 0})_{A/I^n}$ is formally smooth over $(A/I^n, M_{A, 0})$ for every integer $n\geq 1$, and 
\item the base change of $M_{A,0}\to P_0$ in the category of integral monoids is isomorphic to $M_A\to P$. 
\end{itemize}
The choice of $M_{A,0}\to P_0$ is essentially unique: two choices are isomorphic after replacing $M_{A,0}$ by a sufficiently large submonoid.

If $X$ is moreover locally of finite presentation over $A$, then $(X, M_{X, 0})$ is smooth over $(\Spf (A), M_{A, 0})^a$ in the sense of Kato \cite{Kato:log}*{3.3}. 

If the chart $M_A \to P$ is of finite presentation, then $X$ is locally of finite presentation over $A$ and we can take $M_{A, 0}$ and $M_{X, 0}$ so that $(X, M_{X,0})$ is smooth over $(A, M_{A,0})$. 
\end{prop}

\begin{proof}
The third assertion follows from the first by Kato's definition of smoothness. 

Using \cite{ALPT}*{2.1.9}, one can find $(X, M_{X, 0})$ as in the statement if we drop the condition that $(X, M_{X, 0})$ is formally smooth over $(A, M_{A,0})$ (and the uniqueness part is clear from \cite{ALPT}*{2.1.9}). 
(Compare with \cite{ALPT}*{5.2.11}.)
In fact, this argument provides $(X, M_{X,0})$ with an integral chart $M_{A,0}\to P_0$ whose base change to $M_A$ in the category of integral monoids is $M_A \to P$. 
So, we need only prove that every such $(X, M_{X, 0})$ with $M_{A, 0}\to P_0$ is formally smooth. 
This formally smoothness can be deduced from a natural isomorphism $P_0^{\gp}/ M^{\gp}_{A,0}\cong P^{\gp}/M_A^{\gp}$ and the proof of \cite{Kato:log}*{3.4}. Alternatively, we can argue as follows only using the formally smoothness of $(X, M_X)$. 

We may assume $I=0$. So, we shall show that $(X, M_{X, 0})$ is formally smooth over $(A, M_0)$. Suppose we are given a log thickening of first order
\[
\begin{CD}
(S, M_S) @>>> (T, M_T) \\
@VVV @VVV \\
(X, M_{X, 0}) @>>> (\Spec(A), M_0)
\end{CD}
\]
with $T$ affine. (See \cite{Ogus}*{IV.2.1.1}. We need only consider the case where $(S, M_S)\to (X, M_{X, 0})$ is an open immersion \cite{Ogus}*{IV.3.1.5}.)  
After base change to $(A, M_A)$, we have a map 
\[
\tilde{g}\colon (T, M'_{T})\coloneqq (T, M_T)_{(A,M_A)}\to (X, M_X)
\]
that fits into the base change of the above diagram. 
We need to show that the composite $P_0 \to P \overset{\tilde{g}}{\longrightarrow}\Gamma (T, M'_{T})$ factors through a map $\Gamma (T, M_T)\to \Gamma (T, M'_{T})$. 
It suffices to show that
\[
f \colon P_0  \to \Gamma (T, M'_{T}) / \Gamma (T, \cO_{T}^\times) \subset \Gamma (T, M'_{T} / \cO_{T}^\times)
\]
factors through $\Gamma (T, M_{T}) / \Gamma (T, \cO_T^\times)\to \Gamma (T, M'_{T}) / \Gamma (T, \cO_{T}^\times)$. 

As $(S, M_S)\to (T, M_T)$ is the log thickening of first order and $T$ is affine, 
\[
\Gamma (T, M_T) \to \Gamma (S, M_S), \quad \Gamma(T, M'_{T})\to \Gamma (S, M'_{S})
\]
are surjective quotient maps by actions of $\Gamma (T, \cO_T^\times)$, where $M'_S$ denotes the base change of $M_S$ over $M_A$. 
Therefore, $f$ can be identified with
\[
P_0  \to \Gamma (S, M'_{S}) / \Gamma (S, \cO_{S}^\times) \subset \Gamma (S, M'_{S} / \cO_{S}^\times). 
\]
It is now clear that $f$ factors through 
\[
\Gamma (T, M_{T}) / \Gamma (T, \cO_T^\times)=\Gamma (S, M_{S}) / \Gamma (S, \cO_S^\times)\to 
\Gamma (S, M'_{S}) / \Gamma (S, \cO_{S}^\times). 
\]
This finishes the proof of the first assertion. 

For the final assertion, write $M_A=\varinjlim_j M_{A,j}$ by fine submonoids $M_{A, j}$ containing $M_{A,0}$. It suffices to show that the base change $M_{A, j}\to P_j$ of $M_{A,0} \to P_0$ in the category of integral monoids is integral (and injective). 
For this, observe that
$\bZ[P]=\bZ [P_j]$ is flat over $\bZ[M_A]=\varinjlim_j \bZ[M_{A,j}]$. 
By \cite{Stacks}*{\href{https://stacks.math.columbia.edu/tag/02JO}{Tag 02JO}}, there exists some $j$ such that $\bZ [M_{A,j}]\to \bZ [P_j]$ is flat. So, $M_{A, j}\to P_j$ is integral and injective by \cite{Kato:log}*{4.1}.  
\end{proof}

Let us collect some other properties of smooth log formal schemes. 

\begin{rem}
\begin{enumerate}
    \item Our definition immediately implies that the morphism $(X, M_X)\to (\Spf (A), M_A)^a$ is integral if $(X, M_X)$ is smooth. Moreover, using the notation of Definition \ref{smooth}, we see that $A \langle P\rangle$ is $I$-completely flat over $A \langle M_A\rangle$ in the sense of \cite{BS}. So, if $X=\Spf (R)$ is affine and admits a smooth chart, $R$ is $I$-completely flat over $A$ as well. 
    \item Every smooth $(X, M_X)$ over $(A/I, M_A)$ admits, \'etale locally, a smooth lift over $(A, M_A)$.
    If $(X, M_X)$ is smooth over $(A/I, M_A)$ and $X$ is affine, it lifts uniquely to a smooth log $I$-adic formal scheme $(\widetilde{X}, M_{\widetilde{X}})$ over $(A, M_A)$, cf. the proof of \cite{Kato:log}*{3.14} (see also \cite{Kato:II}*{p.27} for the correction of its statement).  
\end{enumerate}    
\end{rem}

We remarked that smooth log schemes are formally smooth (in the stronger sense as above). 
The following consequence of formally smoothness and flatness is fundamental in the main text. 

\begin{prop}
Let $(X, M_X) \to (Y, M_Y)$ be an exact closed immersion of smooth log schemes over $(A/I, M_A)$. Then, $X \to Y$ is a quasi-regular immersion of schemes in the sense of \cite{Stacks}*{\href{https://stacks.math.columbia.edu/tag/063J}{Tag 063J}}. 
If, moreover, $Y$ is locally of finite presentation over $A/I$, then $X\to Y$ is a regular immersion.   
\end{prop}

\begin{proof}
Since $X, Y$ are flat over $A/I$, the second assertion follows from the first; see \cite{Stacks}*{\href{https://stacks.math.columbia.edu/tag/063U}{Tag 063U}, \href{https://stacks.math.columbia.edu/tag/063S}{Tag 063S}}. 
So, we shall prove the first assertion. 

In the rest of the proof, we follow \cite{EGA-IV-1}*{Ch.0, 19.5.4.2} in the case of schemes. 
We may assume that $X=\Spec (B), Y=\Spec (B/J)$ are affine and $J/J^2$ is a free $B/J$-module by \cite{Ogus}*{IV.2.3.2}. (See also the proof of \cite{Ogus}*{IV.3.2.2.}.) 
Let $M_B$ be a chart of $M_X$, which is a chart of $M_Y$ as well. 
We need to show that a canonical surjection
\[
\bigoplus_{i=0}^\infty \textnormal{Sym}^i(J/J^2) \to
\bigoplus_{i=0}^\infty J^n/J^{n+1}
\]
is an isomorphism of $B/J$-algebras, equivalently, an injective map. 
Set
\[
E_n\coloneqq B/J^{n+1}, \quad F_n \coloneqq \bigoplus_{i=0}^n \textnormal{Sym}^i (J/J^2). 
\]
As $(E_n, M_B)^a \to (B/J, M_B)^a=(Y, M_Y)$ are log thickenings and $(Y, M_Y)$ is formally smooth, we obtain maps of prelog rings $(B/J, M_B) \to (E_n, M_B)^a$ over $(A, M_A)$ lifting the identity of $(B/J, M_B)$, compatibly with $n$. As $J/J^2$ is free, we also have maps $J/J^2 \to J/J^{n+1}$ lifting the identity of $J/J^2$, compatibly with $n$. These induce maps
\[
(\bigoplus_{i=0}^\infty \textnormal{Sym}^i (J/J^2), M_B) \to (E_n, M_B)^a
\]
of prelog rings over $(B/J, M_B)$, which in turn induces surjections $F_n\to E_n$; these maps are compatible with $n$. 
The kernel of $F_n \to E_n$ is contained in degree $\geq 2$, and in particular nilpotent. 
So, the induced map $v_n\colon (F_n, M_B)^a\to (E_n, M_B)^a$ are log thickenings. 
As $(X, M_X)=(B, M_B)^a$ is formally smooth, there exist maps $p_n\colon (B, M_B)\to (E_n, M_B)^a$ of prelog rings over $(A, M_A)$ lifting $v_n$. The image $p_n (J)$ is contained in $J/J^{n+1}$, so $p_n$ factors as
\[
(B, M_B) \longrightarrow (E_n, M_B) \overset{w_n}{\longrightarrow} (F_n, M_B),
\]
and the composite $v_n \circ w_n$ induces the identity of $E_n$.  
As the $\textnormal{gr}^0, \textnormal{gr}^1$ of $v_n$ are identities of $B/J$ and $J/J^2$, so are those of $w_n$. Then, we see that the composite
\[
\textnormal{Sym}^n(J/J^2) \longrightarrow J^n /J^{n+1} \overset{\textnormal{gr}^n (w_n)}{\longrightarrow}
\textnormal{Sym}^n(J/J^2)
\]
is the identity. This proves the desired injectivity, and we complete the proof. 
\end{proof}

\begin{rem}
If $X, Y$ are locally of finite type over a noetherian $A/I$, we can replace smooth by formally smooth: the proof above shows that the immersion is quasi-regular, hence regular by \cite{Stacks}*{\href{https://stacks.math.columbia.edu/tag/063I}{Tag 063I}}. 
If all the log structures are fine saturated, it is a special case of \cite{Kato-Saito}*{4.4.4}. (Note that in the reference log schemes are assumed to be fine saturated \cite{Kato-Saito}*{p.77}, and a regular immersion \cite{Kato-Saito}*{1.6.1} is defined as a Koszul-regular immersion in the sense of \cite{Stacks}*{\href{https://stacks.math.columbia.edu/tag/063J}{Tag 063J}}.)
\end{rem}

For computation of log crystalline cohomology, the following stronger version of lifting property is useful. 
Compare with the PD smoothness in \cite{Beilinson}*{1.4}. 

\begin{rem}[Strong lifting property]\label{strong lifting property}
Suppose $X=\Spf (R)$ is affine of finite presentation over $A$, and admits a Kato-smooth chart $P\to \Gamma (X, M_X)$. Then $(X, M_X)_{A/I^n}$ satisfies a stronger lifting property: if $(S, M_S)\to (T, M_T)$ is an exact closed immersion of log-affine integral log schemes over $(A/I^n, M_A)$ defined by a \emph{locally nilpotent} ideal (a.k.a. nilideal), i.e., a log thickening, then, for every commutative diagram
\[
\begin{CD}
(S, M_S)@>>> (T, M_T) \\
@VVV @VVV \\
(X, M_X)@>>> (\Spec (A/I^n), M_A)^a, 
\end{CD}
\]
there exists $(T, M_T)\to (X, M_X)$ making the resulting diagram commutative. 
If $M_A, M_X$ are fine, this reduces by \cite{Beilinson}*{1.1 Remarks (i)} to the case of nilpotent immersions, and the formally smoothness gives such a map. (Compare with \cite{Stacks}*{\href{https://stacks.math.columbia.edu/tag/07K4TH}{Tag 07K4}}.) 
In general,  by Proposition \ref{descent to the fine case}, $(X, M_X)$ has the form of $(X, M_{X, 0})_{(A, M_A)}^{\textnormal{int}}$ for a fine log structure $M_{X, 0}$ that is smooth in the sense of Kato \cite{Kato:log}*{3.3} over $(A, M_{A_0})$ for a fine submonoid $M_{A_0}\subset M_A$.
Since the strong lifting property is preserved by base change and applying $(-)^{\textnormal{int}}$, $(X, M_X)$ satisfies the strong lifting property.  
\end{rem}

Finally, let us discuss to what extent (Kato-)smoothness is independent of charts. 

\begin{prop}\label{smoothness is independent of charts}
Let $M'_A \to A$ be another integral chart of $(\Spf (A), M_A)^a$, and let $\overline{x}$ be a geometric point of $X$. 
Assume that the kernel of $(M'_A)^{\gp} \to M_{X, \overline{x}}^{\gp}/\cO^\times_{X, \overline{x}}$ is finitely generated over $(M'_A)^{\times}$.
If $(X, M_X)$ is Kato-smooth (resp. smooth) over $(A, M_A)$, then $(X, M_X)$ admits a Kato-smooth (resp. smooth) chart over $(A, M'_A)$ on some \'etale neighborhood of $\overline{x}$. 
Moreover, this chart can be made exact at $\overline{x}$. 
\end{prop}

\begin{rem}
Suppose $(X, M_X)$ is smooth over $(A, M_A)$. 
If $M'_A\to A$ is an exact chart, then the kernel of $(M'_A)^{\gp} \to M_{X, \overline{x}}^{\gp}/\cO^\times_{X, \overline{x}}$ equals $(M'_A)^\times$ since $(X, M_X)$ is integral, hence exact, over $(\Spf (A), M_A)^a$. 

If one works \'etale locally on $\Spf (A)$ as well, then the assumption that the kernel of $(M'_A)^{\gp} \to M_{X, \overline{x}}^{\gp}/\cO^\times_{X, \overline{x}}$ is finitely generated over $(M'_A)^{\times}$ allows one to replace $M'_A \to A$ by an exact chart with the same $(M'_A)^{\gp}$. 
\end{rem}

\begin{proof}
Take a Kato-smooth chart 
\[
M_A\to P \overset{\alpha_P}{\longrightarrow}\Gamma (U, M_X)
\]
on a neighborhood $U$ of $\overline{x}$. 
Let $\alpha_{P, \overline{x}}$ denote the composite $P\to \Gamma (U, M_X) \to M_{X, \overline{x}}$. 
It suffices to find, after replacing $U$ if necessary, a relatively coherent chart
\[
M'_A \to P' \overset{\alpha_{P'}}{\longrightarrow}\Gamma (U, \cO_X) 
\]
satisfying the conditions in Definition \ref{smooth}. 
Let $\alpha'_{\overline{x}}$ denote the map $M'_A \to M_{X, \overline{x}}$. 

Fix an presentation
\[
P^{\gp}/ M_A^{\gp} \cong \bZ^r \oplus \bigoplus_{i=1}^s \bZ/m_i \bZ
\]
and choose lifts $e_1, \dots, e_r, f_1, \dots, f_s \in P^{\gp}$ of $1$ in the corresponding summands. 
In particular, $g_i=f_i^{m_i} \in M_A^{\gp}$. 
As $M_A, M'_A$ induce isomorphic log structures on $X$, we can find elements $g'_1, \dots, g'_r \in M'_A$ such that 
\[
u_i\coloneqq \frac{\alpha'_{\overline{x}} (g'_i)}{\alpha_{P, \overline{x}}(g_i)} \in M^{\gp}_{X, \overline{x}}
\]
belongs to $\cO^\times_{X, \overline{x}}$. 

Let us now start to construct a chart. 
As $\cO^\times_{X, \overline{x}}$ is $m_i$-divisible, take $v_i$ such that $v_i^{m_i}= u_i$. 
Consider the following group $(P^{\gp})'$: it is generated by $M'_A$ and symbols $e'_1, \dots, e'_r, f'_1, \dots, f'_r$ with relations $(f'_i)^{m_i} =g'_i$.
Define $h\colon (P^{\gp})' \to M_{X, \overline{x}}^{\gp}$ compatible with $\alpha'_{\overline{x}}$ by sending $e'_i$ to $\alpha_{P, \overline{x}} (e_i)$, $f'_i$ to $\alpha_{P, \overline{x}} (f_i) v_i$. As $g_i u_i =g'_i$, this indeed determines a map. Moreover, the composite $(P^{\gp})' \to M_{X, \overline{x}}^{\gp} \to M_{X, \overline{x}}^{\gp}/ \cO^\times_{X, \overline{x}}$ is surjective as the same holds for $\alpha^{\gp}_{P, \overline{x}}$ and $M_A, M'_A$ have the isomorphic log structures. 
We take $P'$ to be the inverse image of $M_{X,\overline{x}}/\cO^\times_{X, \overline{x}}$. 
By Lemma \ref{constructing charts}, the restriction $h|_{P'}\colon P'\to M_{X, \overline{x}}$ extends to a chart of $M_X$ on some \'etale neighborhood of $\overline{x}$. 
This chart is exact at $\overline{x}$. 

Finally, we observe that $M'_A \to P'$ is integral if $(X, M_X)$ is smooth over $(A, M_A)$. 
Indeed, the argument of Step 5 of the proof of \cite{KatoF}*{Theorem 4.1} works as the associated morphism of log formal schemes is integral and $M'_A \to P'$ is exact at $\overline{x}$.
\end{proof}

\subsection{Prelog rings}
We collect some of our terminology for prelog rings. 

\begin{defn}
Suppose a ring $A$ is classically $I$-complete for a finitely generated ideal $I\subset A$. 
\begin{enumerate}
    \item A prelog ring $(A, \alpha\colon M\to A)$ is a \emph{log ring} if $\alpha$ induces an isomorphism $\alpha^{-1}(A^\times)\cong A^\times$. Every prelog ring $(A, M)$ has the associated log ring. 
    \item A map $(A, M_A) \to (B,M_B)$ of classically $I$-complete prelog rings is \emph{$I$-completely smooth} (resp. \emph{$I$-completely \'etale}) if the associated log $I$-adic formal scheme $(\Spf (B), M_B)^a$ is smooth (resp. \'etale) over $(A, M_A)$. 
\end{enumerate}
When $I=0$, we drop ``$I$-completely'' from these terminologies. 
\end{defn}

If $(\Spf (B),M_B)^a$ is $I$-completely smooth over $(A, M_A)$, then $A\to B$ is $I$-completely flat in the sense of \cite{BS} because of our definition of smoothness. (More precisely, one can find an $I$-completely \'etale cover $B\to B'$ such that $A\to B'$ is $I$-completely flat, and hence $A\to B$ is $I$-completely flat.)
 
\begin{defn}
A map $(A, M)\to (B, N)$ of prelog rings is
\begin{enumerate}
    \item \emph{strict} if $M\to N$ is an isomorphism,
    \item \emph{surjective} if $A\to B$, $M\to N$ are surjective,
    \item \emph{exact surjective} if it is surjective and $M/M^\times \cong N/N^\times$. 
\end{enumerate}
\end{defn}

If $(A, M)\to (B, N)$ is an exact surjective map of prelog rings with integral monoids, the associated closed immersion of log schemes is exact in the sense of \cite{Kato:log}*{4.6(2)}, equivalently, it is strict, i.e., an exact closed immersion in the sense of \cite{Kato:log}*{3.1}; see the proof of \cite{Kato:log}*{4.10}. 

\section{Cosimplicial computation}
We record cosimplicial computations we need. This is related to \cite{Niziol}*{3.28} and \cite{BS}*{Lemma 5.4}. 

\subsection{Setting}
Fix a prime number $p$. 
Let $k$ be a ring with a prelog structure $M \to k$. 
Let $M\to Q$ be an injective integral map of integral monoids, and set $G\coloneqq Q^{\gp}/M^{\gp}$. 
Let $M \to Q^{(1)}$ denote the base change of $M \to Q$ along the Frobenius $M \to M; m\to m^p$, with the relative Frobenius $Q^{(1)}\to Q$ over $M$. (In particular, $Q^{(1)}$ is integral as well.)
We also set $G^{(1)}\coloneqq (Q^{(1)})^{\gp}/M^{\gp}$; the composite $Q^{(1)}\to Q\to G$ factors through $Q^{(1)}\to G^{(1)}$. 

Set
\[
A^{n} \coloneqq k\otimes_{\bZ [M]} \bZ [Q \oplus G^{\oplus n}], \quad
A^{n(1)} \coloneqq k\otimes_{\bZ [M]} \bZ [Q^{(1)} \oplus G^{(1)\oplus n}]
\]
with the relative Frobenius
\[
A^{n(1)} \to A^{n}
\]
induced by $Q^{(1)}\to Q$ and $G^{(1)}\to G$. 
We also have an intermediate
\[
B^{n} \coloneqq k\otimes_{\bZ [M]} \bZ [Q^{(1)} \oplus G^{\oplus n}]. 
\]

We put the structure of a cosimplicial $k$-algebra on $A^{\bullet}$ in the following way:
face maps are
\[
\delta^n_j \colon A^n \to A^{n+1}; (q, g_1, \dots, g_n) \mapsto 
\begin{cases}
(q, q g_1^{-1}\cdots g_n^{-1}, g_1, \dots, g_n)& j=0, \\
(q, g_1, \dots, g_{j-1}, e, g_j, \dots, g_n)  & j\neq 0, 
\end{cases}
\]
and degeneracy maps are
\[
\sigma^n_j \colon A^n \to A^{n-1}; (q, g_1, \dots, g_n) \mapsto 
\begin{cases}
(q, g_2, \dots, g_n) & j=0, \\
(q, g_1, \dots, g_{j-1}, g_j g_{j+1}, g_{j+2}, \dots, g_n) & j\neq 0, 
\end{cases}
\]
where maps are described at the level of monoids. 
We can also put the structure of cosimplicial $k$-algebras on $A^{\bullet (1)}, B^{\bullet}$ so that $A^{\bullet (1)}\to B^{\bullet} \to A^{\bullet}$ are maps of cosimplicial $k$-algebras, using similar formulas. 

\subsection{A homotopy equivalence}
Now assume that $Q^{(1)}\to Q$ is exact and injective. So, $M\to Q$ is of Cartier type. 
The following map of $k$-modules
\[
k\otimes_{\bZ} \bZ [Q] \to k\otimes_{\bZ [M]} \bZ [Q^{(1)}];
1\otimes q \mapsto
\begin{cases}
1\otimes q & q \in Q^{(1)} \\
0 & q \notin Q^{(1)},
\end{cases}
\]
is $k[Q^{(1)}]$-linear as $Q^{(1)}\to Q$ is exact, and it induces the $B^{0}$-linear map
\[
\pr^0\colon A^0= k\otimes_{\bZ [M]} \bZ [Q] \to B^0=k\otimes_{\bZ [M]} \bZ [Q^{(1)}]
\]
again by the exactness of $Q^{(1)}\to Q$; the induced map is a section of the inclusion $B^0\hookrightarrow A^0$.  
We extend it to a $B^{n}$-linear map 
\[
\pr^n \colon A^n \to B^{n}
\]
by sending $(q, g_1, \dots, g_n)$ to $0$ if $q\notin Q^{(1)}$ and $(q, g_1, \dots, g_n)$ itself otherwise.  
This defines a map of cosimplicial $B^{\bullet}$-modules
\[
\pr^{\bullet}\colon A^{\bullet} \to B^{\bullet}, 
\]
which is a section of the natural inclusion. 

\begin{prop}
The map $\pr^{\bullet}$, regarded as an endomorphism of $A^{\bullet}$, is homotopic to the identity map as a map of cosimplicial $A^{\bullet(1)}$-modules.  
\end{prop}

\begin{proof}
We shall construct a $A^{n(1)}$-linear homotopy
\[
h^n \colon A^{n} \to \prod_{\Delta[1]_n} A^{n}. 
\]
Let $\alpha^n_j\in \Delta[1]_n$ denote the map $[n]\to [1]$ defined by $\alpha^n_j (i)=0$ if and only if $i<j$. We define the corresponding component $h^n (\alpha^n_j)$ as follows:
\[
h^n (m, g_1, \dots, g_n)(\alpha^n_j)=
\begin{cases}
\pr^n & j=0, \\
(m, g_1, \dots, g_n) & j\neq 0, n+1,  g_j\cdots g_n \in \textnormal{Im}(Q^{(1)}\to G), \\ 
0 & j\neq 0, n+1,  g_j\cdots g_n \notin \textnormal{Im}(Q^{(1)}\to G), \\
(m, g_1, \dots, g_n) & j=n+1. 
\end{cases}
\]
Note that these maps are $A^{n(1)}$-linear as $Q^{(1)}\to Q$ is exact. 
One can check that this indeed defines a homotopy between $\pr^{\bullet}$ and $\textnormal{id}$. 
\end{proof}

\begin{cor}\label{cosimplicial equiv}
For any cosimplicial $A^{\bullet(1)}$-module $M^{\bullet}$, the natural map
\[
B^{\bullet}\otimes_{A^{\bullet(1)}} M^{\bullet} \to A^{\bullet}\otimes_{A^{\bullet(1)}} M^{\bullet}
\]
is a homotopy equivalence of cosimplicial $A^{\bullet(1)}$-modules. 
\end{cor}

\begin{proof}
The homotopy constructed above can be extended to the tensor product because of $A^{\bullet(1)}$-linearity. 
\end{proof}

\subsection{The relative Frobenius is a quasi-isomorphism}

\begin{prop}\label{Relative Frobenius is a quasi-isomorphism}
Assume that $G= Q^{\gp}/M^{\gp}$ is a free abelian group and $Q^{(1)}\to Q$ is exact.  
For any cosimplicial $A^{\bullet (1)}$-module $M^{\bullet}$, the natural map
\[
M^{\bullet} \to A^{\bullet}\otimes_{A^{\bullet (1)}} M^{\bullet}
\]
gives a quasi-isomorphism on the associated cochain complexes of $k$-modules. 
\end{prop}

Given Corollary \ref{cosimplicial equiv}, it suffices to show the following lemma:

\begin{lem}
If $G= Q^{\gp}/M^{\gp}$ is a free abelian group, then, for any cosimplicial $A^{\bullet (1)}$-module $M^{\bullet}$, the natural map
\[
M^{\bullet} \to B^{\bullet}\otimes_{A^{\bullet (1)}} M^{\bullet}
\]
gives a quasi-isomorphism on the associated cochain complexes of $k$-modules. 
\end{lem}

\begin{proof}
We will deduce this from \cite{BS}*{Lemma 5.4}.  
First observe that we only need the structure over
\[
k[G^{(1)\oplus n} ] \subset A^{n(1)}, 
\]
i.e., 
\[
B^{\bullet}\otimes_{A^{\bullet (1)}} M^{\bullet}\cong 
k[G^{\oplus \bullet} ] \otimes_{k[G^{(1)\oplus \bullet}]} M^{\bullet}. 
\]
We shall prove that, for any cosimplicial $k[G^{(1)\oplus \bullet}]$-module $M^{\bullet}$, the natural map
\[
M^{\bullet} \to k[G^{\oplus \bullet} ] \otimes_{k[G^{(1)\oplus \bullet}]} M^{\bullet}
\]
gives a quasi-isomorphism on the associated complexes of $k$-modules. 

Let $C^{\bullet}$ denote the \v{C}ech nerve of $k\to k[H]$. Then, we have the following embedding
\[
k[G^{\oplus \bullet}] \to C^{\bullet}; (e, g_1, \dots, g_n) \mapsto (g_1^{-1} \dots g_n^{-1}, g_1, \dots, g_n), 
\]
and we have the corresponding map
\[
k[G^{(1)\oplus\bullet}] \to C^{\bullet(1)}. 
\]
By Lemma \cite{BS}*{Lemma 5.4} ($k$ is assumed to have characteristic $p$ but the argument actually works in general with our formulation), the natural map
\begin{equation}
C^{(1)\bullet}\otimes_{k[G^{(1)\oplus\bullet}]} N^{\bullet} \to
C^{\bullet}\otimes_{k[G^{(1)\oplus\bullet}]} N^{\bullet} \tag{BS}
\end{equation}
gives a quasi-isomorphism on the associated complexes. 
For simplicity, assume $G\cong \bZ$. 
In this case, both sides of (BS) are naturally graded. For instance, there is a decomposition
\[
C^{n}\cong\bigoplus_{a\in \bZ} \bigoplus_{a_0 +\cdots a_n =a} k \cdot X_0^{a_0}\cdots X_n^{a_n}, 
\]
and the degree $a$ component defines a cosimplicial $k[G^{\bullet}]$-submodule $C^{\bullet}_{a}$ of $C^{\bullet}$, and $C^{\bullet}_{0}= k[G^{\bullet}]$. 
This induces the following decomposition
\[
C^{\bullet}\otimes_{k[G^{(1)\oplus\bullet}]} N^{\bullet}
\cong\bigoplus_{a\in \bZ} C^{\bullet}_a \otimes_{k[G^{(1)\oplus\bullet}]} N^{\bullet}. 
\]
Similarly, there is a decomposition
\[
C^{\bullet (1)}\cong\bigoplus_{a\in p\bZ} C^{\bullet (1)}_a
\]
as cosimplicial $k[G^{(1)\oplus\bullet}]$-modules with $C^{\bullet(1)}_0=k[G^{(1)\oplus\bullet}]$, inducing maps
\[
C^{\bullet (1)}_a \to C^{\bullet}_a 
\]
of cosimplicial $k[G^{(1)\oplus\bullet}]$-modules. (Note that $G^{(1)}\cong pG\subset G$ canonically.)
Thus, the map (BS) factors over the cosimplicial $C^{\bullet}_0$-submodule
\[
\bigoplus_{a\in p\bZ} C^{\bullet}_a \otimes_{k[G^{(1)\oplus\bullet}]} N^{\bullet}, 
\]
and the resulting map decomposes into the direct sum of
\[
C^{(1)\bullet}_a\otimes_{k[G^{(1)\oplus\bullet}]} N^{\bullet} \to
C^{\bullet}_a \otimes_{k[G^{(1)\oplus\bullet}]} N^{\bullet}, 
\]
and it gives a quasi-isomorphism on the associated complexes for every $a$. 
The case of $a=0$ is exactly what we want. 
The case of general $G$ is similar. 
\end{proof}

\end{document}